\pdfoutput=1
\RequirePackage{ifpdf}
\ifpdf 
\documentclass[pdftex]{sigma}
\else
\documentclass{sigma}
\fi

\usepackage{tikz}
\usetikzlibrary{arrows,matrix,positioning,automata}
\usepackage{mathtools}
\usepackage{eucal}

\numberwithin{equation}{section}

\newtheorem{thm}{Theorem}[section]
\newtheorem*{Theorem*}{Theorem}
\newtheorem{cor}[thm]{Corollary}
\newtheorem{lem}[thm]{Lemma}
\newtheorem{prop}[thm]{Proposition}
 { \theoremstyle{definition}

\newtheorem{exmp}[thm]{Example}
\newtheorem{rem}[thm]{Remark} }

\def\C{\mathbb C}
\def\Z{\mathbb Z}
\def\gl{\mathfrak{gl}}

\let\on\operatorname

\def\End{\on{End}}

\def\Res{\on{Res}}

\def\Wr{\on{Wr}}

\def\St{\mathcal{\acute S}}
\def\Sh{\widehat S}
\def\Dh{\widehat D}
\def\Wh{\widehat W}
\def\Vh{\widehat V}

\def\red{\mathrm{red}}

\let\alb\allowbreak

\def\>{\relax\ifmmode\mskip.666667\thinmuskip\relax\else\kern.111111em\fi}
\def\:{\relax\ifmmode\mskip.333333\thinmuskip\relax\else\kern.0555556em\fi}
\def\<{\relax\ifmmode\mskip-.333333\thinmuskip\relax\else\kern-.0555556em\fi}
\def\?{\relax\ifmmode\mskip-.666667\thinmuskip\relax\else\kern-.111111em\fi}
\def\vsk#1>{\vskip#1\baselineskip}
\def\vv#1>{\vadjust{\vsk#1>}\ignorespaces}
\def\vvn#1>{\vadjust{\nobreak\vsk#1>\nobreak}\ignorespaces}

\def\lsym#1{#1\alb\dots\relax#1\alb} \def\lc{\lsym,}

\makeatletter
\newcommand{\subalign}[1]{%
\vcenter{%
\Let@ \restore@math@cr \default@tag
\baselineskip\fontdimen10 \scriptfont\tw@
\advance\baselineskip\fontdimen12 \scriptfont\tw@
\lineskip\thr@@\fontdimen8 \scriptfont\thr@@
\lineskiplimit\lineskip
\ialign{\hfil$\m@th\scriptstyle##$&$\m@th\scriptstyle{}##$\hfil\crcr
#1\crcr
}%
}%
}
\makeatother

\begin{document}
\allowdisplaybreaks

\newcommand{\arXivNumber}{2202.06405}

\renewcommand{\PaperNumber}{081}

\FirstPageHeading

\ShortArticleName{Difference Operators and Duality for Trigonometric Gaudin and Dynamical Hamiltonians}

\ArticleName{Difference Operators and Duality for Trigonometric\\ Gaudin and Dynamical Hamiltonians}

\Author{Filipp UVAROV}

\AuthorNameForHeading{F.~Uvarov}

\Address{Higher School of Economics, 6 Usacheva Str., Moscow, 119048, Russia}
\Email{\href{mailto:fuvarov@hse.ru}{fuvarov@hse.ru}}

\ArticleDates{Received February 28, 2022, in final form September 26, 2022; Published online October 25, 2022}

\Abstract{We study the difference analog of the quotient differential operator from [Tara\-sov~V., Uvarov~F., \textit{Lett. Math. Phys.} \textbf{110} (2020), 3375--3400, arXiv:1907.02117]. Starting with a space of quasi-exponentials $W=\langle \alpha_{i}^{x}p_{ij}(x),\, i=1\lc n,\, j=1\lc n_{i}\rangle$, where $\alpha_{i}\in\C^{*}$ and $p_{ij}(x)$ are polynomials, we consider the formal conjugate $\check{S}^{\dagger}_{W}$ of the quotient difference operator $\check{S}_{W}$ satisfying $\Sh =\check{S}_{W}S_{W}$. Here, $S_{W}$ is a linear difference operator of order $\dim W$ annihilating $W$, and $\Sh$ is a linear difference operator with constant coefficients depending on $\alpha_{i}$ and $\deg p_{ij}(x)$ only. We construct a space of quasi-exponentials of dimension $\on{ord} \check{S}^{\dagger}_{W}$, which is annihilated by $\check{S}^{\dagger}_{W}$ and describe its basis and discrete exponents. We also consider a similar construction for differential operators associated with spaces of quasi-polynomials, which are linear combinations of functions of the form $x^{z}q(x)$, where $z\in\mathbb C$ and~$q(x)$ is a~polynomial. Combining our results with the results on the bispectral duality obtained in [Mukhin~E., Tarasov~V., Varchenko~A., \textit{Adv. Math.} \textbf{218} (2008), 216--265, arXiv:math.QA/0605172], we relate the construction of the quotient difference operator to the $(\mathfrak{gl}_{k},\mathfrak{gl}_{n})$-duality of the trigonometric Gaudin Hamiltonians and trigonometric dynamical Hamiltonians acting on the space of polynomials in $kn$ anticommuting variables.}


\Classification{82B23; 17B80; 39A05; 34M35}

\section{Introduction}

{\bf 1.1.}~Consider an operator $T$ acting on functions of a variable $x$ by the rule $(Tf)(x)=f(x+1)$. An operator $S$ of the form
$S=\sum_{i=0}^{N}a_{i}(x)T^{N-i}$, where $a_{0}(x)\lc a_{N}(x)$ are complex valued functions
of $x$ and $a_{0}(x)\ne 0$, is called a linear difference operator of
order $N$. Say that the operator $S$ is monic if $a_{0}(x)=1$. Let us write $\on{ord}(S)$ for the order of~$S$.

Let us fix a branch of $\ln x$ and write $\alpha^{x}$ for ${\rm e}^{x\ln\alpha}$ for any non-zero complex number $\alpha$. A~quasi-exponential is a function of the form $\alpha^{x}p(x)$ for some non-zero $\alpha$ and polynomial $p(x)$. We will say that a complex vector space $W$ is a space of quasi-exponentials if $W$ has a basis consisting of quasi-exponentials.
Let $W$ be a space of quasi-exponentials with a basis
$\{\alpha_{i}^{x}p_{ij}(x),\,i=1\lc n,\,j=1\lc n_{i}\}$,
where the numbers $\alpha_{1}\lc\alpha_n$ are distinct, and $p_{ij}$ are some polynomials. Set $d_i=\max_j\bigl(\deg p_{ij}(x)\bigr)$.
It can be shown that there exists a unique monic linear difference operator $S_{W}$ of order $\dim{W}$ annihilating $W$ and a monic linear difference operator $\check S_{W}$ such that{\samepage
\[\prod_{i=1}^{n}(T-\alpha_{i})^{d_i+1} = \check{S}_{W} S_{W},\]
see Sections~\ref{4.1}--\ref{qdo qe} for details. We will call $\check{S}_{W}$ the quotient difference operator.}

Write $\check{S}_{W}=\sum_{i=1}^{m}\check{a}_{i}(x)T^{m-i}$ and denote $T_{-}=T^{-1}$. The formal conjugate $\check{S}^{\dagger}_{W}$ of $\check{S}_{W}$ is a~linear difference operator acting on a function $f(x)$ as follows:
\[
\big(\check{S}^{\dagger}_{W}f\big)(x)=\sum_{i=1}^{m}T_{-}^{m-i}(\check{a}_{i}(x)f(x)).
\]

In Section \ref{qdo qe}, we construct a vector space of functions $Q(W)$ of dimension $\on{ord}\big(\check{S}^{\dagger}_{W}\big)=m$ such that $\check{S}^{\dagger}_{W}$ annihilates $Q(W)$. We prove that $Q(W)$ has a basis of the form
\begin{equation*}
\{\alpha_{i}^{-x}q_{ij}(x),\,i=1\lc n,\,j=1\lc l_{i}\},\qquad q_{ij}\in\C [x],
\end{equation*}
and describe the degrees of the polynomials $q_{ij}(x)$.

For a space of quasi-exponentials $W$ and a point $z\in\C$, we define
the discrete exponents of $W$ at $z$ associated with the operator $T$ and
the $T_-$-discrete exponents of $W$ at $z$ associated with the operator
$T_{-}$. In Sections~\ref{tde} and~\ref{4.5}, we compute the $T_-$-discrete
exponents of the space~$Q(W)$ at the point $z-1$ using
the discrete exponents of $W$ at the point $z$.

{\bf 1.2.}~In Section \ref{spaces with data}, we introduce spaces of quasi-exponentials
with difference data $\big(\bar{\alpha},\bar{\mu};\bar{z},\bar{\lambda}\big)$,
where $\bar{\alpha}=(\alpha_1\lc\alpha_n)$, $\bar{z}=(z_1\lc z_k)$ are
sequences of distinct complex numbers, and $\bar{\mu}=\big(\mu^{(1)}\lc\mu^{(n)}\big)$,
$\bar{\lambda}=\big(\lambda^{(1)}\lc\lambda^{(k)}\big)$ are sequences of partitions.
A space $W$ with the difference data
$\big(\bar{\alpha},\bar{\mu};\bar{z},\bar{\lambda}\big)$ has a basis of the form
$\{\alpha_{i}^{x}p_{ij}(x)\}$, and for each $i=1\lc n$, the partition~$\mu^{(i)}$ describes the degrees of the polynomials $p_{ij}(x)$ with given
$i$. The numbers $z_1\lc z_k$ are singular points (not all) of $W$, and
for each $a=1\lc k$, the partition $\lambda^{(a)}$ describes the discrete
exponents of~$W$ at the point~$z_a$. We denote the set of all spaces of quasi-exponentials with the fixed difference data as $\mathcal{E}\big(\bar{\alpha},\bar{\mu};\bar{z},\bar{\lambda}\big)$.

Applying the results
of Sections \ref{qdo qe}--\ref{4.5}, we define a map
\[
\mathfrak{T}_{1}\colon \ \mathcal{E}\big(\bar{\alpha},\bar{\mu};\bar{z},\bar{\lambda}\big)\rightarrow
\mathcal{E}\big(\bar{\alpha},\bar{\mu}';1-\bar{z},\bar{\lambda}'\big)
\]
by sending the space
$W\in\mathcal{E}\big(\bar{\alpha},\bar{\mu};\bar{z},\bar{\lambda}\big)$
to the image of the space $Q(W)$ under the map
$f(x)\mapsto f(-x)$. Here, the sequences $\bar{\mu}'$, $\bar{\lambda}'$ are
obtained from $\bar{\mu}$, $\bar{\lambda}$ by replacing all partitions
$\mu^{(i)}$,~$\lambda^{(a)}$ by their conjugate, $\big(\mu^{(i)}\big)'$,~$\big(\lambda^{(a)}\big)'$, and
$1-\bar{z}=(1-z_{1}\lc 1-z_{k})$, see details in
Section \ref{spaces with data}.

{\bf 1.3.}~Besides quasi-exponentials, we consider quasi-polynomials, which are functions of the form $x^{z}p(x)$, where $z\in\C$ and $p(x)$ is a polynomial. We introduce the notion of a space of quasi-polynomials with data $\big(\bar{z},\bar{\lambda};\bar{\alpha},\bar{\mu}\big)$, which is analogous to the notion of a space of quasi-exponentials with difference data. Denote the set of all spaces of quasi-polynomials with the fixed data $\big(\bar{z},\bar{\lambda};\bar{\alpha},\bar{\mu}\big)$ as $\mathcal{P}\big(\bar{z},\bar{\lambda};\bar{\alpha},\bar{\mu}\big)$. We introduce an analog of the map $\mathfrak{T}_{1}$ for the spaces of quasi-polynomials:
\[\mathfrak{T}_{2}\colon \ \mathcal{P}\big(\bar{z},\bar{\lambda};\bar{\alpha},\bar{\mu}\big)\rightarrow \mathcal{P}\big(1-\bar{z}-\bar{\lambda}'_{1}-\bar{\lambda}_{1},\bar{\lambda}';\bar{\alpha},\bar{\mu}'\big),
\]
where $1-\bar{z}-\bar{\lambda}_{1}-\bar{\lambda}'_{1}=
\big(1-z_1-\lambda^{(1)}_1-\big(\lambda^{(1)}\big)'_1\lc
1-z_k-\lambda^{(k)}_1-\big(\lambda^{(k)}\big)'_1\big)$ and
$\lambda^{(i)}_1$,~$\big(\lambda^{(i)}\big)'_1$ are the first components of the partitions
$\lambda^{(i)}$, $\big(\lambda^{(i)}\big)'$. The map~$\mathfrak{T}_{2}$ provides a space of quasi-polynomials, which is annihilated by the formal conjugate of the quotient differential operator, an analog of the quotient difference operator introduced above.

The map $\mathfrak{T}_{2}$ is constructed as the counterpart of the map $\mathfrak{T}_{1}$ under the bispectral duality introduced and studied in paper \cite{MTV4}, see also Section \ref{BD}. More precisely, the bispectral duality establishes a bijection
\[
\mathfrak{T}_{3}\colon \ \mathcal{P}\big(\bar{z},\bar{\lambda};\bar{\alpha},\bar{\mu}\big)\rightarrow \mathcal{E}\big(\bar{\alpha},\bar{\mu};,\bar{z}+\bar{\lambda'}_{1},\bar{\lambda}\big),
\]
where $\bar{z}+\bar\lambda'_1=\big(z_1+\big(\lambda^{(1)}\big)'_1\lc z_k+\big(\lambda^{(k)}\big)'_1\big)$. We define $\mathfrak{T}_{2}=\mathfrak{T}_{3}^{-1}\mathfrak{T}_{1}\mathfrak{T}_{3}$ and prove that for a space of quasi-polynomials $V$, the space $\mathfrak{T}_{2}(V)$ is annihilated by the formal conjugate $\check{D}^{\dagger}_{V}$ quotient differential operator $\check{D}_{V}$ (see Theorem~\ref{map T_2}).

{\bf 1.4.}~To study relations between the quotient difference operator and the quotient differential operator, we use the notion of pseudo-differnce operators, see Section \ref{psdifference op}. Let $V$ be a space of quasi-polynomials with the data $\big(\bar{z},\bar{\lambda};\bar{\alpha},\bar{\mu}\big)$, and denote $W=\mathfrak{T}_{1}(\mathfrak{T}_{3}(V))$. To the spaces~$V$ and~$W$, one can associate pseudo-difference operators~$\mathcal{S}_{V}$ and~$\mathcal{S}_{W}$ called the fundamental pseudo-difference operators of $V$ and $W$, respectively. Then $W=\mathfrak{T}_{1}(\mathfrak{T}_{3}(V))$ implies
\[
\mathcal{S}_{V}=\mathcal{S}_{W}^{-1},
\]
see Theorem \ref{discrete main 1}.

For convenience of a reader, we depict the relations between $\mathfrak{T}_{1}$, $\mathfrak{T}_{2}$, and $\mathfrak{T}_{3}$ on the following commutative diagram:
\begin{equation*}
\begin{tikzpicture}[baseline=(current bounding box.center), scale=0.8]
\node(1) at(-1,-2) {$\mathcal{P}\big(\bar{z},\bar{\lambda};\bar{\alpha},\bar{\mu}\big)$};
\node(2) at(4,0) {$\mathcal{E}\big(\bar{\alpha},\bar{\mu};,\bar{z}+\bar{\lambda'}_{1},\bar{\lambda}\big)$};
\node(3) at(4,-4) {$\mathcal{P}\big(1-\bar{z}-\bar{\lambda}'_{1}-\bar{\lambda}_{1},\bar{\lambda}';\bar{\alpha},\bar{\mu}'\big)$};
\node(4) at(9,-2) {$\quad\quad\quad\mathcal{E}\big(\bar{\alpha},\bar{\mu}';1-\bar{z}-\bar{\lambda'}_{1},\bar{\lambda}'\big)$.};
\node(5) at(1.2,-2) {$\mathcal{S}_{V}$};
\node(6) at(6.5,-2) {$\mathcal{S}_{V}^{-1}$};

\draw[->] (1)--node[above left] {$\mathfrak{T}_{3}$} (2);
\draw[->] (3)--node[below right] {$\mathfrak{T}_{3}$} (4);

\draw[->] (1)--node[below left] {$\mathfrak{T}_{2}$} (3);
\draw[->] (2)--node[above right] {$\mathfrak{T}_{1}$} (4);

\draw[|->] (5)--(6);
\end{tikzpicture}
\end{equation*}

{\bf 1.5.}~Our study of the map $\mathfrak{T}_{1}$ is motivated by the $(\gl_{k},\gl_{n})$-duality between the trigonometric Gaudin Hamiltonians $H_{1}\lc H_{n}\in U(\gl_{k})^{\otimes n}$ and the trigonometric dynamical Hamiltonians $G_{1}\lc G_{n}\in U(\gl_{n})^{\otimes k}$, see \cite{J,TV2}, and Section \ref{trig Gaudin and Dyn}. Both $U(\gl_{k})^{\otimes n}$ and $U(\gl_{n})^{\otimes k}$ act on the space~$\mathfrak{P}_{kn}$ of polynomials in $k$ times $n$ anticommuting variables $\xi_{ai}$, $a=1\lc k$, $i=1\lc n$. Let $\rho(H_{1})\lc \rho(H_{n})$ be the images of the trigonometric Gaudin Hamiltonians in $\End (\mathfrak{P}_{kn})$, and let $\rho(G_{1})\lc\rho(G_{n})$ be the images of the trigonometric dynamical Hamiltonians in $\End (\mathfrak{P}_{kn})$. It is known that
\begin{equation}\label{duality intro}
 \rho(H_{i})=-\rho(G_{i}),\qquad i=1\lc n,
\end{equation}
see \cite{TU2} and Proposition \ref{Gaudin Dyn duality}. In particular, any common eigenvector of $H_{1}\lc H_{n}$ in $\mathfrak{P}_{kn}$ is a~common eigenvector of $G_{1}\lc G_{n}$, and vice versa.

Common eigenvectors of the Hamiltonians can be found using the Bethe ansatz
method. For an ``admissible'' space of quasi-polynomials
$V\in\mathcal{P}\big(\bar{z},\bar{\lambda};\bar{\alpha},\bar{\mu}\big)$,
the Bethe ansatz associates an eigenvector $v_{W}$ of $H_{1}\lc H_{n}$
acting in $\mathfrak{P}_{kn}$, see \cite{MV2} and Sections \ref{BA for trig Gaudin}, \ref{qpol and BA} for details. Denote the corresponding eigenvalues as $h^{V}_{1}\lc h^{V}_{n}$. Similarly, for an ``admissible'' space of
quasi-exponentials
$W\in\mathcal{E}\big(\bar{\alpha},\bar{\mu}';1-\bar{z}-\bar\lambda'_1,\bar{\lambda}'\big)$,
the Bethe ansatz associates an eigenvector $v_{W}$ of $G_{1}\lc G_{n}$
acting in $\mathfrak{P}_{kn}$, see \cite{MV2} and
Sections~\ref{BA for XXX},~\ref{qexp to BAE} for details. Denote the corresponding eigenvalues as $g^{W}_{1}\lc g^{W}_{n}$. We will show that if $W=\mathfrak{T}_{1}(\mathfrak{T}_{3}(V))$, then
\[h_{i}^{V}=-g_{i}^{W},\]
see Theorems \ref{discrete main 2} and \ref{discrete main 2 extended}.
This ``matches'' the $(\gl_{k},\gl_{n})$-duality \eqref{duality intro}, so, using that for generic~$\bar{z}$,~$\bar{\alpha}$, the common eigenspaces of the Hamiltonians are one-dimensional, we conclude that for such~$\bar{z}$,~$\bar{\alpha}$, the vector $v_{V}$ is proportional to $v_{W}$, see Sections~\ref{qdo and duality}, \ref{non-reduced data}.
Here and below, when we say ``for generic~$\bar{z}$,~$\bar{\alpha}$'', we mean ``for all~$\bar{z}$,~$\bar{\alpha}$, except, maybe, solutions of some algebraic equation''.

The exchange of the trigonometric Gaudin and dynamical Hamiltonians under
the $(\gl_{k},\gl_{n})$-duality is expected to be a part of the duality between
the Bethe algebras of the trigonometric Gaudin model and the XXX-type
spin chain model. The Bethe algebra of the trigonometric Gaudin model is
a commutative subalgebra of the universal enveloping algebra
$U\big(\widetilde{\gl_{k}}\big)$ of the loop algebra $\widetilde{\gl_{k}}$,
see \cite{MR}, and the Bethe algebra of the XXX-type spin chain model is
a commutative subalgebra of the Yangian $Y(\gl_{n})$, see \cite{MTV6}. Both
$U\big(\widetilde{\gl_{k}}\big)$ and $Y(\gl_{n})$ act on the space $\mathfrak{P}_{kn}$. The images of the trigonometric Gaudin Hamiltonians in $\End (\mathfrak{P}_{kn})$ belong to the image of the Bethe algebra of the trigonometric Gaudin model, and the images of the trigonometric dynamical Hamiltonians in $\End (\mathfrak{P}_{kn})$ belong to the image of the Bethe algebra of the XXX-type spin chain model.
It is expected that the equality of the images of the Hamiltonians extends to the equality of the images of the Bethe algebras. The corresponding result for
the rational Gaudin model was established in \cite{TU1}, where we developed and used the differential analogs of the results for the quotient difference operator studied here. Therefore, the results of this paper can be considered as the first steps in establishing the duality between the Bethe algebras of the trigonometric Gaudin model and the XXX-type spin chain model.

The results of this work and our previous works \cite{TU1,TU2} are devoted to the $(\gl_{k},\gl_{n})$-duality in quantum integrable models on the space $\mathfrak{P}_{kn}$ of polynomials in anticommuting variables. The parallel results for the space $P_{kn}$ of polynomials in commuting variables were obtained earlier, see works \cite{MTV2,MTV4,MTV1,TV4}. In particular, our map $\mathfrak{T}_{1}\circ\mathfrak{T}_{3}$ is the $\mathfrak{P}_{kn}$-analog of the map $\mathfrak{T}_{3}$ introduced in \cite{MTV4}.

\medskip

{\bf 1.6.~Summary of the results.}
\begin{enumerate}\itemsep=0pt
 \item For a space of quasi-exponentials $W$ and the formal conjugate of the quotient difference operator $\check{S}^{\dagger}_{W}$, we construct a space of quasi-exponentials $Q(W)$ of dimension $\on{ord} \check{S}^{\dagger}_{W}$ annihilated by $\check{S}^{\dagger}_{W}$. We describe quasi-exponential basis of $Q(W)$ and its $T_{-}$-discrete exponents. Our findings allow us to define the map $\mathfrak{T}_{1}$ between sets of spaces of quasi-exponentials with difference data.
 \item We prove that if $W=\mathfrak{T}_{1}(\mathfrak{T}_{3}(v))$, where $\mathfrak{T}_{3}$ is the bispectral duality studied earlier in \cite{MTV4}, then for the fundamental pseudo-difference operators $\mathcal{S}_{V}$ and $\mathcal{S}_{W}$ of $V$ and $W$, respectively, we have $\mathcal{S}_{V}=\mathcal{S}_{W}^{-1}$ (Theorem \ref{discrete main 1}).
 \item We prove that $\mathfrak{T}_{2}=\mathfrak{T}_{3}^{-1}\mathfrak{T}_{1}\mathfrak{T}_{3}$ provides the space of quasi-polynomials annihilated by the quotient differential operator.
 \item For the eigenvalues $h_{1}^{V}\lc h_{n}^{V}$ of the trigonometric Gaudin Hamiltonians given by an admissible space of quasi-polynomials $V$ with the data $\big(\bar{z},\bar{\lambda};\bar{\alpha},\bar{\mu}\big)$ and the eigenvalues $g_{1}^{W}\lc g_{n}^{W}$ of the trigonometric dynamical Hamiltonians given by an admissible space of quasi-exponentials $W$ with the difference data $\big(\bar{\alpha},\bar{\mu}';1-\bar{z}-\bar\lambda'_1,\bar{\lambda}'\big)$, we show that if $W=\mathfrak{T}_{1}(\mathfrak{T}_{3}(v))$, then $h_{i}^{V}=-g_{i}^{W}$ (Theorems~\ref{discrete main 2} and~\ref{discrete main 2 extended}).
\end{enumerate}

{\bf 1.7. Plan of the paper.}
The paper is organized as follows. In Section \ref{qdo}, we construct and study
the quotient difference operator, and define the map $\mathfrak{T}_{1}$.
In Section \ref{q differential op}, we introduce the spaces of quasi-polynomials with the data $\big(\bar{z},\bar{\lambda};\bar{\alpha},\bar{\mu}\big)$ and announce the existence of the map $\mathfrak{T}_{2}$.
We recall the bispectral duality map $\mathfrak{T}_{3}$ in Section \ref{BD}.
In Section \ref{psdifference op}, we study relations between quotient differential and quotient difference operators using
pseudo-difference operators and use these relations to construct and study the map $\mathfrak{T}_{2}$. In Section \ref{duality}, we consider
the $(\gl_{k},\gl_{n})$-duality for the trigonometric Gaudin and dynamical
Hamiltonians and relate it to the composition map
$\mathfrak{T}_{1}\circ\mathfrak{T}_{3}$. Identities for discrete Wronskian
used in the paper are collected in Appendix~\ref{appendixA}.

\section{Quotient difference operator}\label{qdo}
The results of Sections \ref{4.1}--\ref{qdo qe} for difference operators
are analogous to that of \cite[Sec\-tions~6.1--6.4]{TU1} for differential
operators.
\subsection{Factorization of a difference operator}\label{4.1}
For any functions $g_{1}\lc g_{n}$, let
\[
\mathcal{W}{\rm r}(g_{1}\lc g_{n})=\det\big(\big(T^{j-1}g_{i}\big)_{i,j=1}^{n}\big)
\]
 be their discrete Wronskian.
Let $\mathcal{W}{\rm r}_{i}(g_{1}\lc g_{n})$ be the determinant of the $n\times n$ matrix whose $j$-th row is $g_{j}, Tg_{j}\lc T^{n-i-1}g_{j},$ $T^{n-i+1}g_{j}\lc T^{n}g_{j}$.

Fix functions $f_{1}\lc f_{n}$ such that $\mathcal{W}{\rm r}(f_{i_{1}}\lc f_{i_{m}})\neq 0$ for any $1\leq i_{1}< \dots < i_{m}\leq n$. In particular, the functions $f_{1}\lc f_{n}$ are linearly independent.

\begin{lem}\label{wronskian formula for difference op}
There exists a unique monic linear difference operator $S= T^{n}+\sum_{i=1}^{n}a_{i}T^{n-i}$ of order $n$ such that $Sf_{i}=0$, $i=1\lc n$.
Moreover, the coefficients $a_{1}\lc a_{n}$ of the difference operator~$S$ are given by the formula
\begin{equation}\label{Ta_{i}(x)}
a_{i}=(-1)^{i} \frac{\mathcal{W}{\rm r}_{i}(f_{1}\lc f_{n})}{\mathcal{W}{\rm r}(f_{1}\lc f_{n})}, \qquad i=1\lc n ,
\end{equation}
and for any function $g$, we have
\begin{equation}\label{S}
Sg=\frac{\mathcal{W}{\rm r}(f_{1}\lc f_{n}, g)}{\mathcal{W}{\rm r}(f_{1}\lc f_{n})} .
\end{equation}
\end{lem}
\begin{proof}
Solving
\[
\begin{pmatrix}
f_{1} & Tf_{1} & \dots & T^{n-1}f_{1} \\
\vdots & \vdots & & \vdots \\
f_{n} & Tf_{n} & \dots & T^{n-1}f_{n}
\end{pmatrix}
\begin{pmatrix}
a_{n} \\ \vdots \\a_{1}
\end{pmatrix}
=
\begin{pmatrix}
T^n f_1 \\ \vdots \\ T^n f_n
\end{pmatrix}
\]
for $a_{1}\lc a_{n}$ by Cramer's rule yields formula \eqref{Ta_{i}(x)}, and this solution is unique. Formula \eqref{S} follows
from the last row expansion of the determinant in the numerator.
\end{proof}

\begin{prop}\label{Sfact} The difference operator $S$ can be written in the following form:
\begin{equation}\label{S2}
S=\left(T - \frac{g_{1}(x+1)}{g_{1}(x)}\right)\left(T - \frac{g_{2}(x+1)}{g_{2}(x)}\right)\cdots \left(T - \frac{g_{n}(x+1)}{g_{n}(x)}\right),
\end{equation}
where $g_{n}=f_{n}$, and
\begin{equation}\label{g i}
g_{i}=\frac{\mathcal{W}{\rm r}(f_{n},f_{n-1}\lc f_{i} )}{\mathcal{W}{\rm r}(f_{n}, f_{n-1}\lc f_{i+1})},\qquad
i=1\lc n-1.
\end{equation}
\end{prop}
\begin{proof}
Denote by $S_{1}$ the difference operator in the right-hand side of \eqref{S2}. By uniqueness of the operator $S$ stated in Lemma \ref{wronskian formula for difference op}, it is sufficient to prove that $S_{1}f_{i}=0$ for all $i=1\lc n$. We will prove it by induction on $n$.

If $n=1$, then $g_{1}=f_{1}$, and $S_{1}f_{1}=(T-f_{1}(x+1)/f_{1}(x))f_{1}(x)=0$.
Let $S_{2}$ be the monic linear difference operator of order $n-1$ such that $S_{2}f_{i}=0$, $i=2\lc n$. By induction assumption,
\[S_{2}=\left(T - \frac{g_{2}(x+1)}{g_{2}(x)}\right)\left(T - \frac{g_{3}(x+1)}{g_{3}(x)}\right)\cdots \left(T - \frac{g_{n}(x+1)}{g_{n}(x)}\right).\]

Since $S_{1}= (T - g_{1}(x+1)/g_{1}(x))S_{2}$, we have $S_{1}f_{i}=0$, $i=2\lc n$. Formulas~\eqref{S} and~\eqref{g i} yeild $S_{2}f_{1}=g_{1}$, thus $S_{1}f_{1}=0$ as well.
\end{proof}

\subsection{Formal conjugate difference operator}
Denote $T_{-}=T^{-1}$. Then $(T_{-}f)(x)=f(x-1)$. Let $f_{1}\lc f_{n}$, and $S$ be like in the previous section.
Define the \textit{formal conjugate} of $S$ by the formula:
\begin{equation*}
S^{\dagger}h(x)=(T_{-})^{n}h(x)+\sum_{i=1}^{n}(T_{-})^{n-i} \bigl(a_{i}(x)h(x)\bigr) .
\end{equation*}

By Proposition \ref{Sfact}, we have
\begin{equation}\label{Sddagger}
S^{\dagger}=\left(T_{-}-\frac{g_{n}(x+1)}{g_{n}(x)}\right)\left(T_{-}-\frac{g_{n-1}(x+1)}{g_{n-1}(x)}\right)\cdots \left(T_{-}-\frac{g_{1}(x+1)}{g_{1}(x)}\right).
\end{equation}

Define \[h_{i} = T\left(\frac{\mathcal{W}{\rm r}(f_{1}\lc
f_{i-1},f_{i+1}
\lc f_{n})}{\mathcal{W}{\rm r}(f_{1}\lc f_{n})}\right) .\]

\begin{prop}\label{Kernel for difference conjugate}
We have $S^{\dagger}h_{i}=0$ for all $i=1\lc n$.
\end{prop}

\begin{proof}Since $h_{1}=(-1)^{n-1}/g_{1}(x+1)$, formula~\eqref{Sddagger} immediately gives $S^{\dagger}h_{1}=0$. To prove that~$S^{\dagger}$ annihilates $h_{2}\lc h_{n}$, one can consider factorization~\eqref{S2} of~$S$, where functions $g_{1}\lc g_{n}$ are defined using a different order of functions $f_{1}\lc f_{n}$, see the proof of Proposition~6.3 in~\cite{TU1} for a differential analog of this argument.
\end{proof}

\subsection{Quotient difference operator}\label{s63S}
Consider functions $f_{1},f_{2}\lc f_{n}, h_1\lc h_k$ such that $\mathcal{W}{\rm r}(g_{1}\lc g_{m})\neq 0$ for any subset $\{g_{1}\lc g_{m}\}$ of $\{f_{1},f_{2}\lc f_{n}, h_1\lc h_k\}$. Let $S$ and $\Sh$ be the monic linear difference operators of order~$n$ and~$n+k$ annihilating $f_{1},f_{2}\lc f_{n}$ and $f_{1},f_{2}\lc f_{n}, h_1\lc h_k$, respectively. Then there is a unique difference operator $\check{S}$ such that $\Sh=\check{S}S$. Indeed, the existence of $\check{S}$ can be seen from the factorization formula \eqref{S2}, and the uniqueness follows from the long division algorithm. We will call $\check{S}$ the \textit{quotient difference operator}.

Define functions $\phi_1\lc\phi_k$ by the formula
\begin{equation}\label{phi}
\phi_{a}=T\left(\frac{\mathcal{W}{\rm r}(f_{1}\lc f_{n},h_{1}\lc
h_{a-1},h_{a+1}
\lc h_{k})}{\mathcal{W}{\rm r}(f_{1}\lc f_{n},h_{1}\lc h_{k})}\right) .
\end{equation}

\begin{prop}\label{kernel for difference quotient conjugate}
We have $\check{S}^{\dagger}\phi_{a}=0$ for all $a=1\lc k$.
\end{prop}
\begin{proof}
Set
$\tilde{h}_{a}=Sh_{a}$, $a=1\lc k$. Formula \eqref{S} yields $\tilde{h}_{i}=\mathcal{W}{\rm r}(f_{1}\lc f_{n},h_{i})/\mathcal{W}{\rm r}(f_{1}\lc f_{n})$. Using this and the Wronskian identities~\eqref{Wr1} and~\eqref{Wr2}, it is easy to check that
\begin{equation}\label{h tilde wronskian}
 \mathcal{W}{\rm r}\big(\tilde{h}_{i_{1}},\lc \tilde{h}_{i_{m}}\big)=\frac{\mathcal{W}{\rm r}(f_{1}\lc f_{n},h_{i_{1}}\lc h_{i_{m}})}{\mathcal{W}{\rm r}(f_{1}\lc f_{n})}
\end{equation}
for any $1\leq i_{1}<\dots <i_{m}\!\leq k$. In particular, $\mathcal{W}{\rm r}\big(\tilde{h}_{i_{1}},\lc \tilde{h}_{i_{m}}\big)\!\neq 0$ for any ${1\leq i_{1}<\dots <i_{m}\!\leq k}$.

By Proposition \ref{Kernel for difference conjugate}, the functions
\[\tilde{\phi}_{a}=T\left(\frac{\mathcal{W}{\rm r}\big(\tilde{h}_{1}\lc
\tilde h_{a-1},\tilde h_{a+1}
\lc\tilde{h}_{k}\big)}{\mathcal{W}{\rm r}\big(\tilde{h}_{1}\lc\tilde{h}_{k}\big)}\right),\qquad a=1\lc k,\]
vanish under the action of $\check{S}^{\dagger}$.

Taking $\{i_{1}\lc i_{m}\} = \{1\lc a-1, a+1\lc k\}$ and $\{i_{1}\lc i_{m}\} = \{1\lc k\}$ in formula~\eqref{h tilde wronskian}, it is easy to see that $\phi_{a} = \tilde{\phi}_{a}$, $a=1\lc k$. The proposition is proved.
\end{proof}

Let $W$ and $\widehat{W}$ be the vector spaces with the bases $f_{1}\lc f_{n}$ and $f_{1}\lc f_{n}, h_{1}\lc h_{k}$, respectively. We will call the span of $\phi_{1}\lc \phi_{k}$ the quotient conjugate space for the pair $\big(W,\widehat{W}\big)$.

\subsection{Quotient difference operator and spaces of quasi-exponentials}\label{qdo qe}
Recall that a quasi-exponential is a function of the form $\alpha^{x}p(x)$ for some non-zero $\alpha$ and a~polynomial~$p(x)$, and a~space of quasi-exponentials is a vector space with a basis consisting of quasi-exponentials. It is straightforward to check that if $g_{1}\lc g_{m}$ are quasi-exponentials, then $\mathcal{W}{\rm r}(g_{1}\lc g_{m})=0$ if and only if $g_{1}\lc g_{m}$ are linearly dependent. Therefore, by Lemma~\ref{wronskian formula for difference op}, for any space of quasi-exponentials $W$, there exists a unique monic linear difference operator~$S_{W}$ of order $\dim{W}$ annihilating~$W$. We will call $S_{W}$ the \textit{fundamental difference operator} of $W$. The following lemma will be useful for us later.
\begin{lem}\label{unique kernel}
If for two spaces of quasi-exponentials $W_{1}$ and $W_{2}$, we have $S_{W_{1}}=S_{W_{2}}$, then $W_{1}=W_{2}$.
\end{lem}
\begin{proof}
Let $f_{1}\lc f_{n}$ and $h_{1}\lc h_{n}$ be the quasi-exponential bases of $W_{1}$ and $W_{2}$, respectively. Using formula \eqref{S}, for each $i=1\lc n$, we have $\mathcal{W}{\rm r}(f_{1}\lc f_{n}, h_{i})=\mathcal{W}{\rm r}(f_{1}\lc f_{n})S_{W_{1}}h_{i}=0$. Therefore, $f_{1}\lc f_{n}, h_{i}$ are linearly dependent for each $i=1\lc n$, and $W_{2}\subset W_{1}$. Similarly, one proves that $W_{1}\subset W_{2}$.
\end{proof}

In this paper, a partition $\mu=(\mu_{1},\mu_{2},\dots{})$ is an infinite nonincreasing sequence of nonnegative integers stabilizing at zero. Let $\mu'=(\mu'_{1},\mu'_{2},\dots{})$ denote the conjugated partition,
that is, $\mu'_{i}=\#\{j\,\vert\,\mu_{j}\geq i\}$.
In particular,
$\mu_{1}'$ equals the number of nonzero entries in $\mu$.

Fix nonzero complex numbers $\alpha_{1}\lc\alpha_{n}$ and nonzero partitions
$\mu^{(1)}\lc\mu^{(n)}$. Assume that $\alpha_{i}\neq\alpha_{j}$ for $i\neq j$.
For each $i=1\lc n$,
denote $n_{i}=\big(\mu^{(i)}\big)'_{1}$.
Let $W$ be a space of quasi-exponentials with a basis
\[
\{\alpha_{i}^{x}q_{ij}(x) ,\,i=1\lc n,\, j=1\lc n_{i}\},\]
where $q_{ij}(x)$ are polynomials such that
$\deg q_{ij} =n_i +\mu^{(i)}_{j}-j$.

Denote $p_{i}=\mu_{1}^{(i)}+n_{i}=\max_{j}\deg q_{ij}+1$, and take $\widehat{W}$ to be the span the functions $\alpha _{i}^{x}x^{p}$,
$i=1\lc n$, $p=0\lc p_{i}-1$. Let $Q(W)$ denote the quotient conjugate space for the pair $\big(W,\widehat{W}\big)$.

Let $S_{W}$ be the monic linear difference operator of order $\dim W$ annihilating $W$. We will say that $S_{W}$ is the fundamental difference operator of $W$. On the other hand, the difference operator $\Sh=\prod_{i=1}^{n} (T-\alpha_{i} )^{p_{i}}$ annihilates $\widehat{W}$. Then there exists a difference operator $\check{S}_{W}$ such that $\Sh=\check{S}_{W}S_{W}$,
see Section~\ref{s63S}. By Proposition~\ref{kernel for difference quotient conjugate}, the difference operator $\check{S}_{W}^{\dagger}$ annihilates~$Q(W)$.

\begin{prop}\label{mualpha}
The space $Q(W)$ has a basis of the form
\[\big\{\alpha_{i}^{-x}\check q_{ij}(x)\, |\, i=1\lc n,\, j=1\lc\mu_1^{(i)}\big\} ,\]
where $\deg \check q_{ij} = \mu_{1}^{(i)}+\big(\mu^{(i)}\big)'_{j}-j $, $i=1\lc n$, $j=1\lc\mu_1^{(i)}$.
\end{prop}
\begin{proof}
Denote
\begin{gather*}
\mathcal{W}{\rm r}\big({\Wh}\big)=\mathcal{W}{\rm r}\big(\alpha _{1}^{x},\alpha _{1}^{x}x\lc \alpha _{1}^{x}x^{p_{1}-1},
\dots,
\alpha _{n}^{x}, \alpha _{n}^{x}x\lc \alpha _{n}^{x}x^{p_{n}-1}\big),
\\
\mathcal{W}{\rm r}_{ij}\big(\Wh\big)=\mathcal{W}{\rm r}\big({}\dots , \widehat{\alpha _{i}^{x}x^{j}},\dots{}\big).
\end{gather*}
The functions in the second line are the same except
the function $\alpha _{i}^{x}x^{j}$ is omitted.

For each $i=1\lc n$, set
\begin{equation*}
\boldsymbol{d}_{i}=\big\{n_{i}+\mu^{(i)}_j-j ,\, j=1\lc n_{i}\big\} ,\qquad
\boldsymbol{d}_{i}^{c}=\{0,1,2\lc p_{i}-1\}\setminus\boldsymbol{d}_{i} .
\end{equation*}
Notice that the functions $\alpha_{i}^{x}x^{l}$, $i=1\lc n$, $l\in \boldsymbol{d}_{i}^{c}$, complement
the basis ${\{\alpha_{i}^{x}q_{ij}(x),\, i=1\lc n}$, $j=1\lc n_{i}\}$ of $W$ to a basis of $\Wh$. Therefore, from the construction of the space $Q(W)$, in particular, from formula~\eqref{phi}, it follows that $Q(W)$ is spanned by functions $f_{ij}$, $i=1\lc n$, $j\in \boldsymbol{d}_{i}^{c}$ of the form
\begin{equation*}
f_{ij} = T\frac{\mathcal{W}{\rm r}_{ij}\big(\Wh\big)}{\mathcal{W}\big(\Wh\big)}+T\sum_{s=j+1}^{p_i-1} C_{ils} \frac{\mathcal{W}{\rm r}_{is}\big(\Wh\big)}{\mathcal{W}{\rm r}\big(\Wh\big)},
\end{equation*}
where $C_{ils}$ are complex numbers.

Using an induction similar to what we used in the proof of Lemma 6.5 in \cite{TU1}, we obtain the following formulas:
\begin{gather}\label{Sdenom}
\mathcal{W}{\rm r}\big(\Wh\big)=\prod_{i=1}^{n}\left(\alpha_{i}^{p_{i}x}\prod_{s=1}^{p_i-1}
\alpha_{i}^{s}s! \right) \prod_{1\leq i<j\leq n}(\alpha_{j}-\alpha_{i})^{p_{i}p_{j}} ,
\\
\mathcal{W}{\rm r}_{ij}\big(\Wh\big)=r_{ij}(x)
 \prod_{l=1}^n\Bigg(\alpha_{l}^{(p_{l}-\delta_{il})x}\prod_{\substack{s=1\\(l,s)\ne(i,j)}}^{p_l-1}
\alpha_{l}^{s}s!\Bigg) \prod_{1\leq l<l'\leq n}(\alpha_{l'}-\alpha_{l})^{(p_{l}-\delta_{li})(p_{l'}-\delta_{l'i})},\nonumber
\end{gather}
where $r_{ij}(x)$ is a monic polynomial in $x$ and $\deg r_{ij}=p_{i}-j-1$. Then for the functions $f_{ij}$, we have $f_{ij}=\alpha_{i}^{-x}\tilde{r}_{ij}(x)$, where $\deg \tilde{r}_{ij}=p_{i}-j-1$.

Notice that $\boldsymbol{d}_{i}^{c}=\big\{n_{i}-\big(\mu^{(i)}\big)'_{l}+l-1\,|\,l=1\lc\mu_{1}^{(i)}\big\}$. This can be illustrated by enumerating, starting from~$0$, the sides of boxes in the Young diagram for $\mu^{(i)}$ that form the bottom-right boundary, see the example with $\mu^{(i)}=(7,4,2,0,\dots)$ on the picture below:

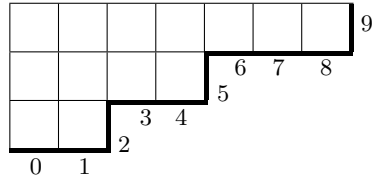
\vsk1.2>
\hbox to\hsize{\hss\hsize130.2pt\parindent0pt
\vtop{\vskip-\baselineskip\vtop to 0pt{\offinterlineskip
\vrule height.4pt depth 0pt width 130.2pt\par
\vrule height18pt \kern18pt
\vrule height18pt \kern18pt
\vrule height18pt \kern18pt
\vrule height18pt \kern18pt
\vrule height18pt \kern18pt
\vrule height18pt \kern18pt
\vrule height18pt \kern17.8pt
\vrule height18pt width1.6pt\par
\vrule height.4pt depth0pt width73.8pt
\vrule height.4pt depth1.2pt width56.3pt\vskip-1.2pt
\vrule height18pt \kern18pt
\vrule height18pt \kern18pt
\vrule height18pt \kern18pt
\vrule height18pt \kern18pt
\vrule height18pt width1.6pt
\par
\vrule height.4pt depth0pt width36.8pt
\vrule height.4pt depth1.2pt width38.4pt
\vskip-1.2pt
\vrule height18pt \kern18pt
\vrule height18pt \kern18pt
\vrule height18pt width1.6pt
\par
\vrule height1.6pt width38.4pt
\vss}}\kern-130.2pt
\vtop{\vtop to 0pt{\strut\kern133pt\raise2.8pt\rlap{\footnotesize9}\kern-7pt\vskip-0.0pt
\strut\kern76.0pt\rlap{}\kern9.2pt\rlap{\footnotesize6}\kern9.2pt\rlap{}\kern5.2pt\rlap{\footnotesize7}\kern18pt\rlap{\footnotesize8}\kern9pt\rlap{}\vskip-3.5pt
\strut\kern78.6pt\rlap{\footnotesize5}\vskip-4.7pt
\strut\kern49.4pt\rlap{\footnotesize3}\kern9.2pt\rlap{}\kern4.2pt\rlap{\footnotesize4}\vskip-3.4pt
\strut\kern41.2pt\rlap{\footnotesize2}\vskip-5.1pt
\strut\kern7.8pt\rlap{\footnotesize0}\kern18.4pt\rlap{\footnotesize1}\vss}
\vskip64pt}\hss}

Then the set $\big\{n_{i}+\mu^{(i)}_j-j,\, j=1\lc n_{i}\big\}$
corresponds to the
right-most sides of the rows, which are
the vertical bonds of the boundary, and the set
$\big\{n_{i}-\big(\mu^{(i)}\big)'_j+j-1 ,\, j=1\lc\mu_{1}^{(i)}\big\}$
corresponds to the bottom sides of the columns, which are the horizontal bonds
of the boundary. For instance, in the given example, $\big\{n_{i}+\mu^{(i)}_j-j ,\,
j=1,2,3\big\}=\{2,5,9\}$ and $\big\{n_{i}-\big(\mu^{(i)}\big)'_j+j-1 ,\, j=1\lc
7\big\}=\{0,1,3,4,6,7,8\}$. Since the horizontal bonds of the boundary complement the vertical bonds, we have $\boldsymbol{d}_{i}^{c}=\{0,1,2\lc p_{i}-1\}\setminus\big\{n_{i}+\mu^{(i)}_j-j ,\, j=1\lc n_{i}\big\}= \big\{n_{i}-\big(\mu^{(i)}\big)'_j+j-1 ,\, j=1\lc\mu_{1}^{(i)}\big\}$.

Denote $j_{l}=n_{i}-\big(\mu^{(i)}\big)'_{l}+l-1$, $l=1\lc \mu_{1}^{(i)}$, so that $\boldsymbol{d}_{i}^{c}=\big\{j_{l},\,l=1\lc \mu_{1}^{(i)}\big\}.$
Denote $\check{q}_{il}=\tilde{r}_{ij_{l}}$, $l=1\lc \mu_{1}^{(i)}$. Then
\[\big\{\alpha_{i}^{-x}\check q_{il}(x)\, |\, i=1\lc n,\, l=1\lc\mu_1^{(i)}\big\}\]
is a basis of $Q(W)$, and
\[
\deg \check{q}_{il} = p_i-j_{l}-1 =
\mu_{1}^{(i)}+n_{i}-\big(n_{i}-\big(\mu^{(i)}\big)'_{l}+l-1\big)-1=\mu_{1}^{(i)}+\big(\mu^{(i)}\big)'_{l}-l.\tag*{\qed}
\]\renewcommand{\qed}{}
\end{proof}

\subsection{Transform of discrete exponents}\label{tde}

Denote $M'=\sum_{i=1}^n\big(\mu^{(i)}\big)'_{1}=\dim W$ and $M=\sum_{i=1}^n\mu^{(i)}_{1}=\dim Q(W)$. For $z\in\C$, define \textit{the sequence of discrete exponents of~$W$ at $z$} as
a unique sequence of integers $(e_{1}>\dots>e_{M'})$
with the property:
there exists a basis $\psi_{1}\lc \psi_{M'}$ of $W$ such that for each $i=1\lc M'$, $\big(T^{j}\psi_{i}\big)(z)=0$ for $j=0\lc e_{i}-1$ and $\big(T^{e_{i}}\psi_{i}\big)(z)\neq 0$.

The sequence of discrete exponents of $W$ at $z$ differs from the sequence $(M'-1,M'-2\lc 0)$ if and only if $z$ is a root of $\mathcal{W}{\rm r}(g_{1}\lc g_{M'})$, where $g_{1}\lc g_{M'}$ is any basis of $W$. If $z$ is such a root, we will call it a \textit{discrete singular point} of $W$.

Define \textit{the sequence of $T_{-}$-discrete exponents of $W$ at $z$} by replacing the operator $T$ in the definition of the sequence of discrete exponents by the operator $T_{-}=T^{-1}$.

\begin{prop}\label{Q_+ exponents}
Let $(e_{1}\lc e_{M'})$ be the sequence of discrete exponents of $W$ at some point $z\in \C$. Define a partition $\lambda = (\lambda_{1},\lambda_{2},\dots)$ by $e_{i}=M'+\lambda_{i}-i$, $i=1\lc M'$ and $\lambda_{M'+1}=0$ for $i>M'$. Let $(\check e_{1}\lc\check e_{M})$ be the sequence of $T_{-}$-discrete exponents of $Q(W)$ at $z-1$. Define a~partition $\eta=(\eta_{1},\eta_{2},\dots)$ by $\check e_{a}=M+\eta_{a}-a$, $a=1\lc M$, and $\eta_{M+1}=0$. Then $\eta_{a}\geq\lambda'_{a}$ for all $a=1,2,\dots$.
\end{prop}
\begin{proof}
Let $\{\psi_{1}\lc\psi_{M'}\}$ be
a basis of $W$ such that for each $i=1\lc M'$, $j=0\lc e_{i}-1$, we have $\big(T^{j}\psi_{i}\big)(z)=0$ and $\big(T^{e_{i}}\psi_{i}\big)(z)\neq 0$.

By formula \eqref{Sdenom}, the Wronskian $\mathcal{W}{\rm r}\big(\Wh\big)$ has no zeros,
thus~$z$ is not a discrete singular point of~$\Wh$.
Therefore, there is
a basis $\{f_{1},f_{2}\lc f_{M+M'}\}$ of $\Wh$ such that it contains the set $\{\psi_{1}\lc\psi_{M'}\}$ and for
each $i=0\lc M+M'-1$, $j=0\lc i$, we have $f_{i+1} (z+j)=0$ and $f_{i+1} (z+i)\neq 0$.

Consider a matrix-valued function
\[
F_{a}(x)=\big(T^{j}f_{i}\big)_{\subalign{ &i=1\lc M+M',\quad i\neq a\\ &j=0\lc M+M'-2}},\]
and denote
\[\mathcal{W}{\rm r}_{a}\big(\Wh\big)=\det F_{a}(x) = \mathcal{W}{\rm r} (f_{1}\lc f_{a-1},f_{a+1}\lc f_{M+M'}).\]

Notice that since $\{\psi_{1}\lc\psi_{M'}\}\subset\{f_{1}\lc f_{M+M'}\}$, we have $\{e_{1}\lc e_{M'}\}\subset\{0,1,2\lc M+M'-1\}$, in particular, $\lambda_{1}\leq M$. Denote $\boldsymbol{e}^{c}=\{0,1,2\lc M+M'-1\}\setminus\{e_{1}\lc e_{M'}\}$. Then by the construction of the space $Q(W)$, the functions
\[
\chi_{a} \coloneqq T\left(\frac{\mathcal{W}{\rm r}_{a+1}\big(\Wh\big)}{\mathcal{W}{\rm r}\big(\Wh\big)}\right),\qquad a\in \boldsymbol{e}^{c},
\]
span $Q(W)$. Let us prove that
\begin{equation}\label{2.13}
(T_{-})^{b}\chi_{a}(z-1)=0,\qquad b=0\lc M+M'-a-2.
\end{equation}

The matrix $F_{a}(z)$ is upper-triangular, and the diagonal of $F_{a}(z)$ is of the form $\{d_{1},d_{2}\lc d_{a-1},0,0\dots\}$, where $d_{b}\neq 0$, $b=1\lc a-1$. An example with $M+M'=6$, $a=4$ is shown below:
\begin{equation*}
F_{4}(z)=
\begin{pmatrix}
d_{1} & \star & \star & \star & \star \\
0 & d_{2} & \star & \star & \star \\
0 & 0 & d_{3} & \star & \star \\
0 & 0 & 0 & 0 & d_{4} \\
0 & 0 & 0 & 0 & 0
\end{pmatrix}.
\end{equation*}

For every $b=0\lc M+M'-2$, let $F_{ab}$ be an $(M+M'-b-1)\times (M+M'-b-1)$ submatrix of $F_{a}(z)$ located in the upper-left corner. We have{\samepage
\begin{equation}\label{shifts of det}
\det\big[((T_{-})^{b}F_{a})(z)\big]=C_{ab}\cdot \det (F_{ab}),\qquad b=0\lc M+M'-2,
\end{equation}
where $C_{ab}$ are some functions of $z$.}

The relations \eqref{shifts of det} are illustrated by the example with $M+M'=6$, $a=4$, $b=1,2$ below:
\begin{equation*}
((T_{-})F_{4})(z)=
\begin{tikzpicture}[baseline=-0.5ex]
\matrix [matrix of math nodes,left delimiter=(,right delimiter=),] (m)
{
\star & d_{1} & \star & \star & \star \\
\star & 0 & d_{2} & \star & \star \\
\star & 0 & 0 & d_{3} & \star \\
\star & 0 & 0 & 0 & 0 \\
\star & 0 & 0 & 0 & 0 \\
};
\draw (m-1-2.north west) rectangle (m-4-5.south east);
\draw (m-5-1.north west) rectangle (m-5-1.south east);

\end{tikzpicture},
\qquad
\big((T_{-})^{2}F_{4}\big)(z)=
\begin{tikzpicture}[baseline=-0.5ex]
\matrix [matrix of math nodes,left delimiter=(,right delimiter=),] (m)
{
\star & \star & d_{1} & \star & \star \\
\star & \star & 0 & d_{2} & \star \\
\star & \star & 0 & 0 & d_{3} \\
\star & \star & 0 & 0 & 0 \\
\star & \star & 0 & 0 & 0 \\
};
\draw (m-1-3.north west) rectangle (m-3-5.south east);
\draw (m-4-1.north west) rectangle (m-5-2.south east);
\end{tikzpicture}.
\end{equation*}

In each matrix above, we boxed two minors, whose product gives the determinant of the corresponding matrix up to a sign. The lower-left boxed minor in each case corresponds to the factor $C_{ab}$ in formula~\eqref{shifts of det}. The upper-right boxed minor of~$((T_{-})F_{4})(z)$ is the determinant of~$F_{41}$ and the upper-right boxed minor of $\big((T_{-})^{2}F_{4}\big)(z)$ is the determinant of~$F_{42}$.

Since $\det (F_{ab})=0$ for all $b=0\lc M+M'-a-1$, formula~\eqref{shifts of det} implies~\eqref{2.13}.

Notice that
$\boldsymbol{e}^{c}=\{M'-\lambda'_{a}+a-1,\,a=1\lc M \}.$ This can be illustrated by a similar picture to what we used for the set $\boldsymbol{d}_{i}^{c}$ in the proof of Proposition \ref{mualpha}, except now we should enumerate the path which contains $M$ horizontal intervals and $M'$ vertical intervals, where $M$ and $M'$ might be greater then the number of columns and the number of rows in the diagram for $\lambda$, respectively, see the example with $\lambda = (7,4,2,0,0,\dots)$, $M=10$, and $M'=5$ below:

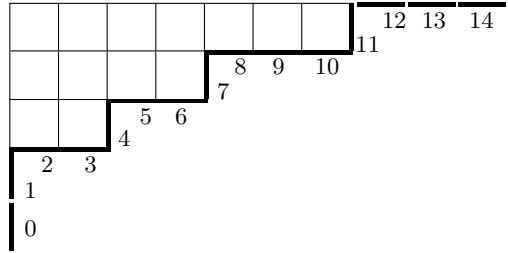
\vsk1.2>
\hbox to\hsize{\hss\hsize130.2pt\parindent0pt
\vtop{\vskip-\baselineskip\vtop to 0pt{\offinterlineskip
\vrule height.4pt depth 0pt width 130.2pt\kern1.5pt\vrule height.4pt depth1.2pt width18pt\kern1pt\vrule height.4pt depth1.2pt width18pt\kern1pt\vrule height.4pt depth1.2pt width18pt\par
\vskip-1.6pt\vrule height18pt \kern18pt
\vrule height18pt \kern18pt
\vrule height18pt \kern18pt
\vrule height18pt \kern18pt
\vrule height18pt \kern18pt
\vrule height18pt \kern18pt
\vrule height18pt \kern17.8pt
\vrule height18pt width1.6pt\par
\vrule height.4pt depth0pt width73.8pt
\vrule height.4pt depth1.2pt width56.3pt\vskip-1.2pt
\vrule height18pt \kern18pt
\vrule height18pt \kern18pt
\vrule height18pt \kern18pt
\vrule height18pt \kern18pt
\vrule height18pt width1.6pt
\par
\vrule height.4pt depth0pt width36.8pt
\vrule height.4pt depth1.0pt width38.4pt
\vskip-1.2pt
\vrule height18pt \kern18pt
\vrule height18pt \kern18pt
\vrule height18pt width1.6pt
\par
\vrule height1.6pt width38.4pt
\par
\vrule height18pt width1.6pt \vskip1.5pt
\par
\vrule height18pt width1.6pt
\vss}}\kern-130.2pt
\vtop{\vtop to 0pt{\strut\kern131pt\raise-5pt\rlap{\footnotesize11}\kern10pt\raise4pt\rlap{\footnotesize12}\kern15pt\raise4pt\rlap{\footnotesize13}\kern18pt\raise4pt\rlap{\footnotesize14}\vskip-2.0pt
\strut\kern76.0pt\rlap{}\kern9.2pt\rlap{\footnotesize8}\kern9.2pt\rlap{}\kern5.2pt\rlap{\footnotesize9}\kern16pt\rlap{\footnotesize10}\kern9pt\rlap{}\vskip-4.4pt
\strut\kern78.6pt\rlap{\footnotesize7}\vskip-4.0pt
\strut\kern49.4pt\rlap{\footnotesize5}\kern9.2pt\rlap{}\kern4.2pt\rlap{\footnotesize6}\vskip-5.4pt
\strut\kern41.2pt\rlap{\footnotesize4}\vskip-3.5pt
\strut\kern12pt\rlap{\footnotesize2}\kern16.4pt\rlap{\footnotesize3}\vskip-4.1pt
\strut\kern5.8pt\rlap{\footnotesize1}\vskip1pt
\strut\kern5.8pt\rlap{\footnotesize0}
\vss}
\vskip64pt}\hss}

\vsk3>
\noindent

Denote $e^{c}_{a}=M'-\lambda'_{a}+a-1$, $a=1\lc M$, so that $\boldsymbol{e}^{c}=\{e^{c}_{a},\,a=1\lc M\}$.

Notice that $M+M'-e_{a}^{c}-2 = M+\lambda'_{a}-a-1$. Therefore, formula \eqref{2.13} yields
\begin{equation}\label{p47_2}
(T_{-})^{b}\chi_{e^{c}_{a}+1}(z-1)=0,\qquad b=0\lc M+\lambda'_{a}-a-1.
\end{equation}

Let $(\check e_{1}\lc\check e_{M})$ be the sequence of $T_{-}$-discrete exponents of $Q(W)$ at $z-1$, and let $\eta=(\eta_{1},\eta_{2},\dots)$ be a partition such that $\check{e}_{a}=M+\eta_{a}-a$, $a=1\lc M$, and $\eta_{M+1}=0$. Denote by $\tilde{\phi}_{1}\lc\tilde{\phi}_{M}$ the basis of $Q(W)$ such that for every $a=1\lc M$, we have $(T_{-})^{b}\tilde{\phi}_{a}(z-1)=0$, $b=0\lc \check e_{a}-1$, and $(T_{-})^{\check{e}_{a}}\tilde{\phi}_{a}(z-1)\neq 0$.

For each $a=1\lc M$, consider the subspace $V_{a}$ of all functions $f$ in $Q(W)$ such that $(T_{-})^{b}f(z-1)=0$, $b=0\lc \check{e}_{a}$. Then the set $\{\tilde{\phi}_{1}\lc\tilde{\phi}_{a-1}\}$ is a basis of $V_{a}$, in particular, $\dim V_{a} = a-1$.

Suppose that $\eta_{a}<\lambda'_{a}$ for some $a=1\lc M$. Then formula~\eqref{p47_2} implies that the span $\tilde{V}_{a}$ of $\chi_{1}\lc\chi_{a}$ is a subspace of~$V_{a}$. But this is impossible since $\dim \tilde{V}_{a}=a>\dim V_{a}$. Therefore, $\eta_{a}\geq\lambda'_{a}$ for all $a=1\lc M$.

As we mentioned above, $\lambda_{1}\leq M$. Therefore, $\lambda^{'}_{M+1}=0$, and the inequality $\eta_{a}\geq\lambda'_{a}$ holds for all $a=1,2,\dots$.

The proposition is proved.
\end{proof}
\begin{rem}
In the next section, we will prove that in Proposition \ref{Q_+ exponents}, we actually have $\eta = \lambda '$, see Corollary \ref{eta=lambda'}.
\end{rem}

\subsection{Quotient for a difference operator with left shifts}\label{4.5}
For any functions $g_{1}\lc g_{n}$, denote
\[ \mathcal{W}{\rm r}_{-}(g_{1}\lc g_{m})=\det\big(\big(T_{-}^{j-1}g_{i}\big)_{i,j=1}^{m}\big).\]
Let $f_{1},f_{2}\lc f_{n}, h_1\lc h_k$ be functions such that $\mathcal{W}{\rm r}_{-}(g_{1}\lc g_{m})\neq 0$ for any subset $\{g_{1}\lc g_{m}\}$ of $\{f_{1},f_{2}\lc f_{n}, h_1\lc h_k\}$. Denote the span of $f_{1}\lc f_{n}$ as $W_{-}$ and the span of $f_{1},f_{2}\lc f_{n}, h_1\lc h_k$ as $\Wh_{-}$. Then define the quotient conjugate space with left shifts for the pair $\big(W_{-},\Wh_{-}\big)$ to be the span of
\begin{equation*}
T_{-}\left(\frac{\mathcal{W}{\rm r}_{-}(f_{1}\lc f_{n},h_{1}\lc
h_{a-1},h_{a+1}
\lc h_{k})}{\mathcal{W}{\rm r}_{-}(f_{1}\lc f_{n},h_{1}\lc h_{k})}\right),\qquad a=1\lc k.
\end{equation*}

Let $W_{-}$ be a vector space with a basis of the form
\[\big\{\alpha_{i}^{-x} q_{ij}(x)\, |\, i=1\lc n,\, j=1\lc\mu_1^{(i)}\big\} ,\]
where $q_{ij}(x)$ are polynomials and $\deg q_{ij} = \mu_{1}^{(i)}+\big(\mu^{(i)}\big)'_{j}-j$, $i=1\lc n$, $j=1\lc\mu_1^{(i)}$.
Also, take $\Wh_{-}$ to be the vector space with a basis $\alpha_{i}^{-x}x^{p}$, $p=0\lc p_{i}-1$. Denote by $Q_{-}(W_{-})$ the quotient conjugate space with left shifts for the pair $\big(W_{-}, \Wh_{-}\big)$.

We have $\dim W_{-}=\sum_{i=1}^{n}\mu^{(i)}_{1}=M$. Similarly to the case of right shifts, it can be shown that there exists a difference operator $S^{-}_{W_{-}}$ of the form
\[ S^{-}_{W_{-}}=(T_{-})^{M}+\sum_{i=1}^{M}b_{i}(x)(T_{-})^{M-i}\]
annihilating $W_{-}$, and that the difference operator $\Sh_{-}=\prod_{i=1}^{n}(T_{-}-\alpha_{i})^{p_{i}}$ annihilating $\Wh_{-}$ is divisible by $S_{W_{-}}$ from the right. Write $\check S^{-}_{W_{-}}$ for the difference operator such that $\Sh_{-}=\check S^{-}_{W_{-}}S^{-}_{W_{-}}$.

For a difference operator $S=\sum_{i=1}^{l}a_{i}(x)(T_{-})^{l-i}$, define its formal conjugate $S^{\dagger}$ by the formula
\begin{gather*}
S^{\dagger}h(x)=\sum_{i=1}^{l}T^{l-i} (a_{i}(x)h(x) ).
\end{gather*}
\begin{prop}\label{kernel for difference quotient conjugate2}
The difference operator $\big(\check{S}^{-}_{W_{-}}\big)^{\dagger}$ annihilates the space $Q_{-}(W_{-})$.
\end{prop}
Proposition \ref{kernel for difference quotient conjugate2} is proved similarly to Proposition \ref{kernel for difference quotient conjugate}.

\begin{prop}\label{mualpha2}
The space $Q_{-}(W_{-})$ has a basis of the form
\[
\big\{\alpha_{i}^{x}\check q_{ij}(x),\,i=1\lc n,\, j=1\lc n_{i}\big\},\]
where $\check q_{ij}(x)$ are polynomials such that
$\deg \check q_{ij} =\big(\mu^{(i)}\big)'_{1}+\mu^{(i)}_{j}-j$.
\end{prop}
Proposition \ref{mualpha2} is proved similarly to Proposition~\ref{mualpha}.

Denote the sequences $(\alpha_{1}\lc\alpha_{n})$ and $\big(\mu^{(1)}\lc\mu^{(n)}\big)$ as $\bar{\alpha}$ and $\bar{\mu}$, respectively. Let $\mathcal{E}(\bar{\alpha},\bar{\mu})$ be the set of all spaces of quasi-exponentials with a basis of the form
\[
\big\{\alpha_{i}^{x}q_{ij}(x)\,|\,i=1\lc n,\, j=1\lc n_{i}\big\},\]
where $q_{ij}(x)$ are polynomials such that
$\deg q_{ij} =\big(\mu^{(i)}\big)'_{1}+\mu^{(i)}_{j}-j$.

Let us write $\bar{\alpha}^{-1}$ for the sequence $\big(\alpha_{1}^{-1}\lc\alpha_{n}^{-1}\big)$ and $\bar{\mu}'$ for the sequence $\big(\big(\mu^{(1)}\big)'\lc \big(\mu^{(n)}\big)'\big)$. By Propositions~\ref{mualpha} and~\ref{mualpha2}, we have maps $Q\colon \mathcal{E}(\bar{\alpha},\bar{\mu})\rightarrow\mathcal{E}\big(\bar{\alpha}^{-1},\bar{\mu}'\big)$, $W\mapsto Q(W)$ and $Q_{-}\colon \mathcal{E}\big(\bar{\alpha}^{-1},\bar{\mu}'\big)\rightarrow\mathcal{E}(\bar{\alpha},\bar{\mu})$, $W_{-}\mapsto Q_{-}(W_{-})$. Let us prove that~$Q_{-}$ is the inverse for~$Q$.

\begin{prop}\label{Q_- inverse}
For any $W\in\mathcal{E}(\bar{\alpha},\bar{\mu})$ and $W_{-}\in\mathcal{E}\big(\bar{\alpha}^{-1},\bar{\mu}'\big)$, the following holds:
\[Q_{-}(Q(W))=W,\qquad Q(Q_{-}(W_{-}))=W_{-}.\]
\end{prop}
\begin{proof}
For any $W\in\mathcal{E}(\bar{\alpha},\bar{\mu})$, define $Q(S_{W})$ to be the difference operator $\check{S}_{W}^{\dagger}$. Similarly, for any $W_{-}\in\mathcal{E}\big(\bar{\alpha}^{-1},\bar{\mu}'\big)$, define $Q_{-}(S^{-}_{W_{-}})$ to be the difference operator $\big(\check{S}^{-}_{W_{-}}\big)^{\dagger}$.

Recall that $\Sh=\prod_{i=1}^{n}(T-\alpha_{i})^{p_{i}}=\big(\Sh_{-}\big)^{\dagger}$ and $\Sh=(Q(S_{W}))^{\dagger}S_{W}$.
We have
\begin{equation}\label{p4.9_1}
\Sh_{-}=\big(\Sh\big)^{\dagger}= (S_{W})^{\dagger}Q(S_{W}).
\end{equation}

In the relation $\Sh_{-}=(Q_{-}(S_{W_{-}}))^{\dagger}S_{W_{-}}$, take $W_{-}=Q(W)$. This yields
\begin{equation}\label{p4.9_2}
\Sh_{-}=(Q_{-}(Q(S_{W})))^{\dagger}Q(S_{W}).
\end{equation}

Comparing formulas \eqref{p4.9_1} and \eqref{p4.9_2}, we have $Q_{-}(Q(S_{W}))=S_{W}$. Therefore, the fundamental difference operators of $W$ and $Q_{-}(Q(W))$ coincide, and the relation $Q_{-}(Q(W))=W$ follows from Lemma~\ref{unique kernel}.

The relation $Q(Q_{-}(W_{-}))=W_{-}$ is proved in a similar way.
\end{proof}

\begin{prop}\label{Q_- exponents}
Fix $z\in\C$. Let $(e_{1}\lc e_{M})$ be the sequence of $T_{-}$-discrete exponents of $W_{-}\in\mathcal{E}\big(\bar{\alpha}^{-1},\bar{\mu}'\big)$ at $z-1$. Define a partition $\lambda = (\lambda_{1},\lambda_{2},\dots)$ by $e_{i}=M+\lambda_{i}-i$, $i=1\lc M$ and $\lambda_{M+1}=0$. Let $(\check e_{1}\lc\check e_{M'})$ be the sequence of discrete exponents of $Q_{-}(W_{-})$ at $z$. Define a partition $\eta=(\eta_{1},\eta_{2},\dots)$ by $\check e_{a}=M'+\eta_{a}-a$, $a=1\lc M'$, and $\eta_{M'+1}=0$. Then $\eta_{a}\geq\lambda'_{a}$ for all $a=1,2,\dots$.
\end{prop}
Proposition \ref{Q_- exponents} is proved similarly to Proposition~\ref{Q_+ exponents}.
\begin{cor}\label{eta=lambda'}
In both Propositions~{\rm \ref{Q_+ exponents}} and~{\rm \ref{Q_- exponents}}, we have $\eta=\lambda'$.
\end{cor}
\begin{proof}
Consider a space $W\in\mathcal{E}(\bar{\alpha},\bar{\mu})$, and let partitions $\lambda$ and $\eta$ be like in Proposition \ref{Q_+ exponents}, in particular $\eta_{a}\geq\lambda'_{a}$ for all $a=1,2,\dots$. But by Proposition~\ref{Q_- inverse} and~\ref{Q_- exponents}, we have $\lambda_{i}\geq\eta'_{i}$ for all $i=1,2,\dots$, which is the same as $\lambda'_{a}\geq\eta_{a}$ for all $a=1,2,\dots$. Therefore, we have $\eta=\lambda'$.

The equality $\eta=\lambda'$ for Proposition~\ref{Q_- exponents} is proved in a similar way.
\end{proof}

\subsection[Spaces of quasi-exponentials with the difference data (bar alpha, bar mu; bar z, bar lambda)]{Spaces of quasi-exponentials with the difference data $\boldsymbol{\big(\bar{\alpha},\bar{\mu};\bar{z},\bar{\lambda}\big)}$}\label{spaces with data}

Let $W$ be a space from the set $\mathcal{E}(\bar{\alpha},\bar{\mu})$. Assume that there exists a sequence of complex numbers $\bar{z}=(z_{1}\lc z_{k})$ and a sequence of partitions $\bar{\lambda}=\big(\lambda^{(1)}\lc\lambda^{(k)}\big)$ such that $z_{1}\lc z_{k}$ are discrete singular points of $W$, $z_{a}-z_{b}\notin\Z$ for $a\neq b$, sequence $\big(e^{(a)}_{1}\lc e^{(a)}_{M'}\big)$ of discrete exponents at $z_{a}$ is given by $e^{(a)}_{i}=M'+\lambda^{(a)}_{i}-i$ for $i=1\lc M'$, $\lambda^{(a)}_{i}=0$ for $i>M'$, and $\sum_{a=1}^{k}\big|\lambda^{(a)}\big|=\sum_{i=1}^{n}\big|\mu^{(i)}\big|$. Here $|\lambda|$ denotes the number of boxes in the Young diagram corresponding to the partition $\lambda$. We will say that $W$ is a \textit{space of quasi-exponentials with the difference data $\big(\bar{\alpha},\bar{\mu};\bar{z},\bar{\lambda}\big)$.}

\begin{exmp}
Let $W$ be the span of the functions $x-2/3$, $x^{2}$, and $2^{x}x$. This space belongs to the set $\mathcal{E}(\bar{\alpha},\bar{\mu})$, where $n=2$, $\alpha_{1}=1$, $\alpha_{2}=2$, $\mu^{(1)}=(1,1,0,\dots)$, $\mu^{(2)}=(1,0,\dots)$. Since $\mathcal{W}{\rm r}(x-2/3, x^{2}, 2^{x}x)=2^{x}x(x-1)(x+8/3)$, the discrete singular points of $W$ are $0$, $1$, and~$-8/3$. The sequence of discrete exponents of $W$ at $x=0$ and $x=-8/3$ is $(3,1,0)$, and the corresponding partition is $\lambda_{1}=(1,0,\dots)$. The sequence of discrete exponents of $W$ at $x=1$ is $(3,2,0)$, and the corresponding partition is $\lambda_{2}=(1,1,0,\dots)$. Therefore, the space $W$ is a space of quasi-exponentials with the data $\big(\bar{\alpha},\bar{\mu};\bar{z},\bar{\lambda}\big)$, where $\bar{z}=(-8/3,1)$ and $\bar{\lambda}=(\lambda_{1},\lambda_{2})$.
\end{exmp}

\begin{exmp}\label{example 2}
Let $W$ be the span of the functions $x$, $x^{2}$, and $(-1/2)^{x}x$. This space belongs to the set $\mathcal{E}(\bar{\alpha},\bar{\mu})$, where $n=2$, $\alpha_{1}=1$, $\alpha_{2}=-1/2$, $\mu^{(1)}=(1,1,0,\dots)$, $\mu^{(2)}=(1,0,\dots)$. Since $\mathcal{W}{\rm r}\big(x, x^{2}, (-1/2)^{x}x\big)=(-1/2)^{x}x(x+1)(x+2)$, the discrete singular points of $W$ are $0$, $-1$, and $-2$. The sequence of discrete exponents of $W$ at $x=0$ is $(3,2,1)$, and the corresponding partition is $\lambda_{1}=(1,1,1,0,\dots)$. The sequence of discrete exponents of $W$ at $x=-1$ is $(4,2,0)$, and the corresponding partition is $\lambda_{2}=(2,1,0,\dots)$. The sequence of discrete exponents of $W$ at $x=-2$ is $(3,1,0)$, and the corresponding partition is $\lambda_{3}=(1,0,\dots)$. Therefore, the space $W$ is a space of quasi-exponentials with the data $\big(\bar{\alpha},\bar{\mu};\bar{z},\bar{\lambda}\big)$, where either $\bar{z}=(0)$ and $\bar{\lambda}=(\lambda_{1})$, or $\bar{z}=(-1)$ and $\bar{\lambda}=(\lambda_{2})$.
\end{exmp}

Define the map $\operatorname{refl}\colon \mathcal{E}\big(\bar{\alpha}^{-1},\bar{\mu}'\big)\rightarrow \mathcal{E}(\bar{\alpha},\bar{\mu}')$ by $\operatorname{refl} (W) = \{f(-x)\,\vert\, f(x)\in W\}$. Denote $\mathfrak{T}_{1}= \operatorname{refl}\circ Q$.
If for a space $W\in\mathcal{E}(\bar{\alpha},\bar{\mu})$, the difference operator $Q(S_{W})$ is written as $Q(S_{W})=(T_{-})^{M}+\sum_{i=1}^{M}b_{i}(x)(T_{-})^{M-i}$, then
\begin{equation}\label{Q rightarrow}
Q^{\rightarrow}(S_{W})=T^{M}+\sum_{i=1}^{M}b_{i}(-x)T^{M-i}
\end{equation}
is the fundamental difference operator of $\mathfrak{T}_{1}(W)$.

For a sequence $\bar{z}=(z_{1}\lc z_{k})$, denote $1-\bar{z}=(1-z_{1}\lc 1-z_{k})$. Recall that for a sequence of partitions $\bar{\eta}=\big(\eta^{(1)}\lc\eta^{(s)}\big)$, $\bar{\eta}'$ denotes the sequence of the conjugated partitions: $\bar{\eta}'=\big(\big(\eta^{(1)}\big)'\lc \big(\eta^{(s)}\big)'\big)$. The next theorem is the main result of Section~\ref{qdo}, and it is an easy consequence of Propositions \ref{mualpha}, \ref{Q_+ exponents}, and Corollary \ref{eta=lambda'}.
\begin{thm}
Let $W$ be a space of quasi-exponentials with the data $\big(\bar{\alpha},\bar{\mu};\bar{z},\bar{\lambda}\big)$. Then $\mathfrak{T}_{1}(W)$ is a space of quasi-exponentials with the data $\big(\bar{\alpha},\bar{\mu}'; 1-\bar{z},\bar{\lambda}'\big)$.
\end{thm}

Let us write $\mathcal{E}\big(\bar{\alpha},\bar{\mu};\bar{z},\bar{\lambda}\big)$ for the set of all spaces of quasi-exponentials with the difference data $\big(\bar{\alpha},\bar{\mu};\bar{z},\bar{\lambda}\big)$. We constructed a map
\begin{equation}\label{T_1}
\begin{split}
\mathfrak{T}_{1}\colon \ \mathcal{E}\big(\bar{\alpha},\bar{\mu};\bar{z},\bar{\lambda}\big)&\rightarrow \mathcal{E}\big(\bar{\alpha},\bar{\mu}';1-\bar{z},\bar{\lambda}'\big),\\
W & \mapsto \mathfrak{T}_{1}(W).
\end{split}
\end{equation}
In Section~\ref{qdo and duality}, we will show that this map is closely related to the $(\gl_{n},\gl_{k})$-duality of the trigonometric Gaudin and dynamical Hamiltonians.

\section{Quotient differential operator}\label{q differential op}
\subsection{Spaces of quasi-polynomials}
By quasi-polynomial we mean a function of the form $x^{z}p(x)$, where $z\in\C$ and $p(x)$ is a polynomial.

Fix complex numbers $z_{1}\lc z_{k}$ and nonzero partitions
$\lambda^{(1)}\lc\lambda^{(k)}$. Assume that $z_{a}-z_{b}\notin\Z$ for $a\neq b$.
Let $V$ be a vector space of functions in one variable with a basis
$\{x^{z_{a}}q_{ab}(x)\,|\, a=1\lc k,\, b=1\lc \big(\lambda^{(a)}\big)'_{1}\}$,
where $q_{ab}(x)$ are polynomials and $\deg q_{ab} =
\big(\lambda^{(a)}\big)'_{1}+\lambda^{(a)}_{b}-b$.
Assume that the space $V$ satisfies the following property, which we will call the non-degeneracy at $0$: for each $a=1\lc k$ and any $b=1\lc \big(\lambda^{(a)}\big)'_{1}$,
there exists a linear combination of polynomials $q_{a1}, q_{a2}\lc q_{a(\lambda^{(a)})'_{1}}$ which has a root at $x=0$ of multiplicity $b-1$.

Denote $L'=\sum_{a=1}^k\big(\lambda^{(a)}\big)'_{1}=\dim V$. For $\alpha\in\C^{*}$, define \textit{the sequence of exponents of $V$ at $\alpha$} as
a unique sequence of integers $(e_{1}>\dots>e_{L'})$,
with the property:
there exists a basis $f_{1}\lc f_{L'}$ of $V$ such that for each $a=1\lc L'$, we have $f_{a}(x)=(x-\alpha)^{e_a} (1+o(1) )$ as $x\to \alpha$.

For any sufficiently differentiable functions $g_{1}\lc g_{s}$, let
\[\Wr(g_{1}\lc g_{s})=\det\big(\big( ({\rm d}/{\rm d}x )^{j-1}g_{i}(x)\big)_{i,j=1}^{s}\big)\]
be their Wronskian. The sequence of exponents of $V$ at $\alpha$ differs from the sequence $(L'-1,L'-2\lc 0)$ if and only if $\alpha$ is a root of $\Wr(g_{1}\lc g_{L'})$, where $g_{1}\lc g_{L'}$ is any basis of~$V$. If $\alpha$ is such a root, we will call it a \textit{singular point} of~$V$.

Let $\alpha_{1}\lc \alpha_{n}$ be the singular points of $V$ and for each $i=1\lc n$, let $\big(e^{(i)}_{1}\lc e^{(i)}_{L'}\big)$ be the sequence of exponents of $V$ at $\alpha_{i}$. For each $i=1\lc n$, define a partition $\mu^{(i)}=\big(\mu^{(i)}_{1},\mu^{(i)}_{2},\dots{}\big)$ as follows: $e^{(i)}_{a}=L'+\mu^{(i)}_{a}-a$ for $a=1\lc L'$, and $\mu^{(i)}_{a}=0$ for $a>L'$.
Clearly, all partitions $\mu^{(1)}\!\lc\mu^{(n)}$ are nonzero.

Denote the sequences $(z_{1}\lc z_{k})$, $\big(\lambda^{(1)}\lc\lambda^{(k)}\big)$, $\big(\alpha_{1}\lc\alpha_{n}\big)$, and $\big(\mu^{(1)}\lc\mu^{(n)}\big)$ as $\bar{z}$, $\bar{\lambda}$, $\bar{\alpha}$, and $\bar{\mu}$, respectively. We will say that $V$ is a \textit{space of quasi-polynomials with the data $\big(\bar{z},\bar{\lambda};\bar{\alpha},\bar{\mu}\big)$.}

\begin{lem}\label{lambda=mu}
Let $V$ be a space of quasi-polynomials with the data $\big(\bar{z},\bar{\lambda};\bar{\alpha},\bar{\mu}\big)$. Then
\begin{equation}\label{lambda=mu formula}
\sum_{a=1}^{k}\big|\lambda^{(a)}\big|=\sum_{i=1}^{n}\big|\mu^{(i)}\big|.
\end{equation}
Here $|\lambda|$ denotes the number of boxes in the Young diagram corresponding to the partition $\lambda$.
\end{lem}

\begin{proof}
Let $g_{1}\lc g_{L'}$ be some basis of the space $V$. Denote $N_{a}=\big(\lambda^{(a)}\big)'_{1}$. Then
\[
\Wr(g_{1}\lc g_{L'})=x^{\sum_{a=1}^{k}N_{a}z_{a}-\sum_{a,b=1}^{k}N_{a}N_{b}}p(x),
\]
where $p(x)$ is a polynomial of degree $\sum_{a=1}^{k}\big|\lambda^{(a)}\big|$. On the other hand, the numbers $\alpha_{1}\lc\alpha_{n}$ are zeros of $p(x)$ with multiplicities $|\mu^{(1)}\big|\lc |\mu^{(n)}\big|$, respectively, and $p(x)$ has no other zeros.
\end{proof}
\begin{rem}
Notice that in the case of spaces of quasi-exponentials with the difference data $\big(\bar{\alpha},\bar{\mu};\bar{z},\bar{\lambda}\big)$, we had to include the condition \eqref{lambda=mu formula} into the definition. As Lemma \ref{lambda=mu} shows, in case of quasi-polynomials, this condition holds automatically. This can be explained by the fact that for the space of quasi-polynomials with the data $\big(\bar{z},\bar{\lambda};\bar{\alpha},\bar{\mu}\big)$, we additionally require the non-degeneracy at $0$.
\end{rem}
\begin{rem}
Notice that if $V$ is a space of quasi-polynomials with some data, then this data
is defined uniquely. This is not the case for spaces of quasi-exponentials with a difference
data, see Example~\ref{example 2}.
\end{rem}
\begin{exmp}
Let $V$ be the span of the functions $f_{1}=x-1$, $f_{2}=(x-1)^{2}$, and $f_{3}=\sqrt{x}(x-1)$. Then $\Wr (f_{1},f_{2},f_{3})=-1/4\, x^{-3/2}(x-1)^{3}$. The sequence of exponents of $V$ at~$1$ is $(3,2,1)$. Therefore, $V$ is a space of quasi-polynomials with the data $\big(\bar{z},\bar{\lambda};\bar{\alpha},\bar{\mu}\big)$, where $\bar{z}=(0,1/2)$, $\bar{\lambda}=(\lambda_{1},\lambda_{2})$ with $\lambda_{1}=(1,1,0,\dots)$, $\lambda_{2}=(1,0,\dots)$, $\bar{\alpha}=(1)$, and $\bar{\mu}=(\mu_{1})$ with $\mu_{1}=(1,1,1,0,\dots)$.
\end{exmp}

\subsection{Spaces of quasi-polynomials and quotient differential operator}
We will use the following two facts about linear differential operators. For proofs, see for example,~\cite{TU1}.
\begin{enumerate}\itemsep=0pt
\item Let $f_{1}(x)\lc f_{s}(x)$ be sufficiently differentiable functions such that $\Wr (f_{1}\lc f_{s})\neq 0$. Then there is a unique monic linear differential operator $D=({\rm d}/{\rm d}x)^{s}+\sum_{i=1}^{s}\!a_{i}(x)({\rm d}/{\rm d}x)^{s-i}\!$ of order $s$ such that $Df_{i}=0$, $i=1\lc s$. The coefficients of the operator $D$ are given by the formulas
\begin{equation}\label{a_{i}(x)}
a_{i}(x)=(-1)^{i} \frac{\Wr_{i}(f_{1}\lc f_{s})}{\Wr(f_{1}\lc f_{s})} , \qquad i=1\lc s ,
\end{equation}
where $\Wr_{i}(f_{1}\lc f_{s})$ is the determinant of the $s\times s$ matrix whose $j$-th row is $f_{j}, ({\rm d}/{\rm d}x)f_{j}\lc ({\rm d}/{\rm d}x)^{s-i-1}f_{j},$ $({\rm d}/{\rm d}x)^{s-i+1}f_{j}\lc ({\rm d}/{\rm d}x)^{s}f_{j}$.

\item Let $V$ and $\Vh$ be two spaces of functions such that $V\subset \Vh$, and for any $f_{1}\lc f_{m}\in \Vh$, $\Wr(f_{1}\lc f_{m})\neq 0$ if and only if $f_{1}\lc f_{m}$ are linearly independent. Let $D$ and $\Dh$ be linear differential operators of order $\dim V$ and $\dim \Vh$ annihilating~$V$ and~$\Vh$, respectively. Then there exists a differential operator $\check{D}$ such that $\Dh=\check{D}D$.
\end{enumerate}

Consider a space $V$ like in the previous section. By item~(1) above, there exists a unique monic differential operator $D_{V}$ of order $L'$ annihilating $V$. We will say that $D_{V}$ is
\textit{the fundamental differential operator of}~$V$.

Denote $l_{a} = \lambda^{(a)}_{1}+\big(\lambda^{(a)}\big)'_{1}-1$. Introduce a differential operator
\begin{equation*}
\Dh=\prod_{a=1}^{k}\prod_{b=0}^{l_{a}}\left(x\frac{{\rm d}}{{\rm d}x}-z_{a}-b\right).
\end{equation*}
Then the span $\Vh$ of the functions $x^{z_{a}+b}$, $a=1\lc k$, $b=0\lc l_{a}$ is annihilated by $\Dh$.

Since $V\subset\Vh$, there exists a differential operator $\check{D}_{V}$ such that $\Dh=\check{D}_{V}x^{k}D_{V}$, see item~(2) in the beginning of the section.

For a differential operator $D=\sum_{i=0}^{s}b_{i}(x)({\rm d}/{\rm d}x)^{s-i}$, define its \textit{formal conjugate} $D^{\dagger}$ by the formula:
\[
D^{\dagger}f(x)=\sum_{i=0}^{s}\left(-\frac{{\rm d}}{{\rm d}x}\right)^{s-i}(b_{i}(x)f(x)),
\]
where $f(x)$ is any sufficiently differentiable function.

Let $\check{D}_{V}^{\dagger}$ be the formal conjugate of $\check{D}_{V}$. Denote $1-\bar{z}-\bar{\lambda}'_{1}-\bar{\lambda}_{1} = \big(1-z_{1}-\big(\lambda^{(1)}\big)'_{1}-\lambda^{(1)}_{1},1-z_{2}-\big(\lambda^{(2)}\big)'_{1}-\lambda^{(2)}_{1}\lc 1-z_{k}-\big(\lambda^{(k)}\big)'_{1}-\lambda^{(k)}_{1}\big)$. We have the following theorem

\begin{thm}\label{map T_2}
Let $V$ be a space of quasi-polynomials with the data $\big(\bar{z},\bar{\lambda};\bar{\alpha},\bar{\mu}\big)$. Then there exists a unique space $\mathfrak{T}_{2}(V)$ of quasi-polynomials with the data $\big(1-\bar{z}-\bar{\lambda}'_{1}-\bar{\lambda}_{1},\bar{\lambda}';\bar{\alpha},\bar{\mu}'\big)$, which is annihilated by $\check{D}_{V}^{\dagger}$.
\end{thm}
We will prove Theorem \ref{map T_2} in Section~\ref{proof of T_{2}}.

Let us write $\mathcal{P}\big(\!\bar{z},\bar{\lambda};\bar{\alpha},\bar{\mu}\big)$ for the set of all spaces of quasi-polynomials with the data $\big(\!\bar{z},\bar{\lambda};\bar{\alpha},\bar{\mu}\big)$. By Theorem~\ref{map T_2}, we have a map
\begin{equation*}
\begin{split}
\mathfrak{T}_{2}\colon \ \mathcal{P}\big(\bar{z},\bar{\lambda};\bar{\alpha},\bar{\mu}\big)&\rightarrow \mathcal{P}\big(1-\bar{z}-\bar{\lambda}'_{1}-\bar{\lambda}_{1},\bar{\lambda}';\bar{\alpha},\bar{\mu}'\big),\\
V & \mapsto \mathfrak{T}_{2}(V).
\end{split}
\end{equation*}

\section{Bispectral duality}\label{BD}
In this section, we recall a transformation introduced in~\cite{MTV4}.

Fix sequences $\bar{z}$, $\bar{\alpha}$, $\bar{\lambda}$, and $\bar{\mu}$, where $\bar{z}=(z_{1}\lc z_{k})$ is a sequence of complex numbers such that $z_{a}-z_{b}\notin\Z$ for $a\neq b$, $\bar{\alpha}=(\alpha_{1}\lc\alpha_{n})$ is a sequence of nonzero complex numbers such that $\alpha_{i}\neq\alpha_{j}$ for $i\neq j$, and $\bar{\lambda}=\big(\lambda^{(1)}\lc\lambda^{(k)}\big)$, $\bar{\mu}=\big(\mu^{(1)}\lc\mu^{(n)}\big)$ are sequences of non-zero partitions. Denote $L'=\sum_{a=1}^{k}\big(\lambda^{(a)}\big)'_{1}$, $M'=\sum_{i=1}^{n}\big(\mu^{(i)}\big)'_{1}$, and $n_{ab}=\big(\lambda^{(a)}\big)'_{1}+\lambda^{(a)}_{b}-b$.

Define polynomials $p_{\bar{\alpha},\bar{\mu}}(x)$ and $q_{\bar{z},\bar{\lambda}}(x)$ as follows:
\begin{gather}\label{p}
p_{\bar{\alpha},\bar{\mu}}(x)=\prod_{i=1}^{n}(x-\alpha_{i})^{(\mu^{(i)})'_{1}},
\\
\label{q}
q_{\bar{z},\bar{\lambda}}(x)=\prod_{a=1}^{k}\prod_{b=1}^{(\lambda^{(a)})'_{1}}(x-z_{a}-n_{ab}).
\end{gather}

Let $V$ be a space of quasi-polynomials with the data $\big(\bar{z},\bar{\lambda};\bar{\alpha},\bar{\mu}\big)$.
Let $D_{V}$ be the fundamental differential operator of~$V$. Define the functions $\beta_{1}(x)\lc\beta_{L'}(x)$ by
\[
x^{L'}D_{V}=\left(x\frac{{\rm d}}{{\rm d}x}\right)^{L'}+\sum_{a=1}^{L'}\beta_{a}(x)\left(x\frac{{\rm d}}{{\rm d}x}\right)^{L'-a}.
\]

\begin{lem}\label{properties of beta}
The following holds
\begin{enumerate}\itemsep=0pt
\item[$1.$] The functions $\beta_{1}(x)\lc\beta_{L'}(x)$ are rational functions regular at infinity.
Denote $\beta_{a}(\infty)=\lim\limits_{x\to\infty}\beta_{a}(x)$, $a=1\lc L'$. Then
\begin{equation}\label{D coef are rational f}
u^{L'}+\sum_{a=1}^{L'}\beta_{a}(\infty)u^{L'-a}=q_{\bar{z},\bar{\lambda}}(u).
\end{equation}
\item[$2.$] For each $a=1\lc L'$, $p_{\bar{\alpha},\bar{\mu}}(x)\beta_{a}(x)$ is a polynomial in $x$.
\end{enumerate}
\end{lem}
\begin{proof}
The fact that $\beta_{1}(x)\lc\beta_{L'}(x)$ are rational functions regular at infinity follows from formula \eqref{a_{i}(x)}. Notice that $\ker \prod_{b=1}^{(\lambda^{(a)})'_{1}}(x({\rm d}/{\rm d}x)-z_{a}-n_{ab})$ is the span of $\big\{x^{z_{a}+n_{ab}}\,|\,a=1\lc k,b=1\lc \big(\lambda^{(a)}\big)'_{1}\big\}$, which implies formula~\eqref{D coef are rational f}.

Part (2) of the lemma follows from formula \eqref{a_{i}(x)} and the following observations:
\begin{itemize}\itemsep=0pt
\item Let $g_{1}\lc g_{L'}$ be a basis of $V$. Denote $N_{a}=\big(\lambda^{(a)}\big)'_{1}$. For each $a=1\lc L'$, define an integer $c_{a}$ by $\sum_{b=c_{a}}^{L'}N_{b}>a$, $\sum_{b=c_{a}+1}^{L'}N_{b}<a$. Then one can check that
\begin{equation}\label{Wr_a}
\Wr_{a}(g_{1}\lc g_{L'})=x^{\sum_{a=1}^{k}N_{a}z_{a}-\sum_{a,b=1}^{k}N_{a}N_{b}-\sum_{b=c_{a}+1}^{L'}N_{b}}\tilde{p}(x),
\end{equation}
where $\tilde{p}(x)$ is a polynomial, and for each $i=1\lc n$, $\alpha_{i}$ is a zero of $\tilde{p}(x)$ of multiplicity not less than $\sum_{\substack{j=1\\j\neq i}}^{n}\big(\mu^{(j)}\big)'_{1}$.
\item As noted in the proof of Lemma \ref{lambda=mu}, we have
\begin{equation}\label{Wr}
\Wr(g_{1}\lc g_{L'})=x^{\sum_{a=1}^{k}N_{a}z_{a}-\sum_{a,b=1}^{k}N_{a}N_{b}}p(x),
\end{equation}
where $p(x)$ is a polynomial, the numbers $\alpha_{1}\lc\alpha_{n}$ are zeros of $p(x)$ with multiplicities $\big|\mu^{(1)}\big|\lc \big|\mu^{(n)}\big|$, respectively, and $p(x)$ has no other zeros.\hfill \qed
\end{itemize}\renewcommand{\qed}{}
\end{proof}

We will call the differential operator $\bar{D}_{V}=x^{L'}p_{\bar{\alpha},\bar{\mu}}(x)D_{V}$ \textit{the regularized fundamental differential operator} of~$V$.

Let $W$ be a space of quasi-exponentials with the difference data $\big(\bar{\alpha},\bar{\mu};\bar{z},\bar{\lambda}\big)$.

Let $b_{1}(x)\lc b_{M'}$ be the coefficients of the fundamental difference operator $S_{W}$ of $W$:
\[
S_{W}=T^{M'}+\sum_{i=1}^{M'}b_{i}(x)T^{M'-i}.
\]

Denote $\bar{z}-\bar{\lambda'}_{1}=\big(z_{1}-\big(\lambda^{(1)}\big)'_{1}\lc z_{k}-\big(\lambda^{(k)}\big)'_{1}\big)$ and $\bar{z}+\bar{\lambda'}_{1}=\big(z_{1}+\big(\lambda^{(1)}\big)'_{1}\lc z_{k}+\big(\lambda^{(k)}\big)'_{1}\big)$.
\begin{lem}\label{p_{W}} The following holds.
\begin{enumerate}\itemsep=0pt
\item[$1.$] The coefficients $b_{i}(x)$ of $S_{W}$ are rational functions regular at infinity. Denote $b_{i}(\infty)=\lim\limits_{x\to\infty}b_{i}(x)$. Then
\[u^{M'}+\sum_{i=1}^{M'}b_{i}(\infty)u^{M'-i}=p_{\bar{\alpha},\bar{\mu}}(u).\]
\item[$2.$] For each $i=1\lc M'$, $q_{\bar{z}-\bar{\lambda'}_{1},\bar{\lambda}}(x)b_{i}(x)$ is a polynomial in~$x$.
\end{enumerate}
\end{lem}
\begin{proof}
Item (1) of the lemma can be proved similarly to item (1) in Lemma~\ref{properties of beta}. For a proof of item (2), see \cite[Lemma~3.9]{MTV4}.
\end{proof}

We will call the difference operator $\bar{S}_{W}=q_{\bar{z}-\bar{\lambda'}_{1},\bar{\lambda}}(x)S_{W}$ \textit{the regularized fundamental difference operator} of $W$.

For any complex numbers $b_{ai}$, $a=0\lc s$, $0=1\lc r$, consider a differential operator $D$ and a difference operator $S$ defined by
\[
D=\sum_{a=0}^{s}\sum_{i=0}^{r}b_{ai} x^{a}\left(x\frac{{\rm d}}{{\rm d}x}\right)^{i},\qquad S=\sum_{a=0}^{s}\sum_{i=0}^{r}b_{ai} x^{i} T^{a}.
\]
We will say that $D$ is \textit{bispectral dual} to $S$, and vice versa, and write $D=S^{\#}$, $S=D^{\#}$.

The following theorem was proved in~\cite{MTV4}.
\begin{thm}\label{bisp dual}
There exists a bijection
\begin{equation}\label{T_3}
\mathfrak{T}_{3}\colon \ \mathcal{P}\big(\bar{z},\bar{\lambda};\bar{\alpha},\bar{\mu}\big) \rightarrow \mathcal{E}\big(\bar{\alpha},\bar{\mu};\bar{z}+\bar{\lambda'}_{1},\bar{\lambda}\big)
\end{equation}
such that for every $V\in \mathcal{P}\big(\bar{z},\bar{\lambda};\bar{\alpha},\bar{\mu}\big)$, $\bar{D}_{V}^{\#}$ is the regularized fundamental difference operator of~$\mathfrak{T}_{3}(V)$.
\end{thm}

\begin{rem}
Theorem~\ref{bisp dual} follows from the proofs of Theorems 4.1 and 4.2 in \cite{MTV4}. The latter theorems state the duality for spaces called non-degenerate in~\cite{MTV4}. We will not need the duality for non-degenerate spaces here.
\end{rem}
\begin{exmp}
Consider the space $W$ from Example~\ref{example 2}. Then
\[S_{W}=T^{3}-\frac{3(x+3)}{2(x+2)} T^{2}+\frac{x+3}{2x}.\]

If we choose the difference data $\big(\bar{\alpha},\bar{\mu};\bar{z},\bar{\lambda}\big)$ for $W$ with $\bar{z}=(0)$ and $\bar{\lambda}=(\lambda_{1})$, $\lambda_{1}=(1,1,1,0\dots)$, then $\bar{S}_{W}=x(x+1)(x+2)S_{W}$ and $\mathfrak{T}^{-1}_{3}(W)$ is the span of the functions $1+(1/2)x^{-3}$, $x^{-1}$, and $x^{-2}-(1/2)x^{-3}$.

If we choose the difference data $\big(\bar{\alpha},\bar{\mu};\bar{z},\bar{\lambda}\big)$ for $W$ with $\bar{z}=(-1)$ and $\bar{\lambda}=(\lambda_{2})$, $\lambda_{2}=(2,1,0,\dots)$, then $\bar{S}_{W}=x(x+2)S_{W}$ and $\mathfrak{T}^{-1}_{3}(W)$ is the span of the functions $1-(3/8)x^{-3}$ and $x^{-2}-x^{-3}$.
\end{exmp}

We will call the space $\mathfrak{T}_{3}(V)$ bispectral dual to $V$, and vice versa. In Section~\ref{psdifference op}, we will construct the map $\mathfrak{T}_{2}$ as the counterpart of the map $\mathfrak{T}_{1}$ under the bispectral duality~$\mathfrak{T}_{3}$, see formula~\eqref{commutative diagram relation} for the precise statement.

\section{Algebra of pseudo-difference operators}\label{psdifference op}
A pseudo-difference operator is a formal series of the form
\begin{equation}\label{pseudodifference operator}
\sum_{m=-\infty}^{M}\sum_{l=-\infty}^{L}C_{lm}x^{l}T^{m},
\end{equation}
where $C_{lm}$ are some complex numbers. Using the operator relations $T^{m}x^{l}=(x+m)^{l}T^{m}$, $l,m\in\Z$, and identifying $(x+m)^{l}$ with its Laurent series at infinity, one can multiply series~\eqref{pseudodifference operator}. This multiplication is associative. Denote the algebra of pseudo-difference operators as~$\Psi\mathfrak{D}_{q}$.

\begin{lem}\label{invertible pseudodifference operators}
If $\mathcal{S}=\sum_{m=-\infty}^{M}\sum_{l=-\infty}^{L}C_{lm} x^{l}T^{m}$ with $C_{LM}\neq 0$, then $\mathcal{S}$ is invertible in~$\Psi\mathfrak{D}_{q}$.
\end{lem}
\begin{proof}
Define $\St$ by the rule $1+\St=C_{LM}^{-1} x^{-L}\mathcal{S} T^{-M}$.
Then $\sum_{j=0}^{\infty}(-1)^{j}\St^{j}$ is a well-defined element of $\Psi\mathfrak{D}_{q}$ and the inverse of~$\mathcal{S}$ is given by the formula:
\[
\mathcal{S}^{-1}=C_{LM}^{-1}T^{-M}\left(\sum_{j=0}^{\infty}(-1)^{j}\St^{j}\right)x^{-L}.\tag*{\qed}
\]\renewcommand{\qed}{}
\end{proof}

We consider a difference operator $S=\sum_{i=0}^{M}a_{i}(x)T^{M-i}$ with rational coefficients $a_{0}(x)\lc a_{M}(x)$ as an element of $\Psi\mathfrak{D}_{q}$ replacing each $a_{i}(x)$ by its Laurent series at infinity. By Lemma~\ref{invertible pseudodifference operators}, if $a_{0}(x)=1$, and $a_{1}(x)\lc a_{M}(x)$ are regular at infinity, then $S$ is invertible in $\Psi\mathfrak{D}_{q}$.

Denote by $\bar{\mathfrak{D}}$ the algebra of differential operators with rational coefficients. One can check that the assignment
\begin{equation}\label{tau}
\tau\colon \ x\frac{{\rm d}}{{\rm d}x} \mapsto -x,\qquad x \mapsto T
\end{equation}
defines a monomorphism of algebras $\tau\colon \bar{\mathfrak{D}}\rightarrow\Psi\mathfrak{D}_{q}$.

As before, fix sequences $\bar{z}$, $\bar{\alpha}$, $\bar{\lambda}$, and $\bar{\mu}$, where $\bar{z}=(z_{1}\lc z_{k})$ is a sequence of complex numbers such that $z_{a}-z_{b}\notin\Z$ for $a\neq b$, $\bar{\alpha}=(\alpha_{1}\lc\alpha_{n})$ is a sequence of nonzero complex numbers such that $\alpha_{i}\neq\alpha_{j}$ for $i\neq j$, and $\bar{\lambda}=\big(\lambda^{(1)}\lc\lambda^{(k)}\big)$, $\bar{\mu}=\big(\mu^{(1)}\lc\mu^{(n)}\big)$ are sequences of non-zero partitions.

Let $V$ be a space of quasi-polynomials with the data $\big(\bar{z},\bar{\lambda};\bar{\alpha},\bar{\mu}\big)$. Let $\bar{D}_{V}\in\bar{\mathfrak{D}}$ be the fundamental regularized differential operator of $V$. Define \textit{the fundamental pseudo-difference operator}~$\mathcal{S}_{V}$ of $V$ by the following formula:
\begin{equation*}
\mathcal{S}_{V}=(p_{\bar{\alpha},\bar{\mu}}(T))^{-1} \tau \big(\bar{D}_{V}\big) (q_{\bar{z},\bar{\lambda}}(-x))^{-1},
\end{equation*}
where the polynomials $p_{\bar{\alpha},\bar{\mu}}(x)$ and $q_{\bar{z},\bar{\lambda}}(x)$ are defined in formulas~\eqref{p} and~\eqref{q}, respectively.

Let $W$ be a space of quasi-exponentials with the difference data $\big(\bar{\alpha},\bar{\mu};\bar{z},\bar{\lambda}\big)$. Let $\bar{S}_{W}$ be the fundamental regularized difference operator of $W$. Define \textit{the fundamental pseudo-difference operator} $\mathcal{S}_{W}$ of $W$ by the following formula:
\begin{equation*}
\mathcal{S}_{W}=(q_{\bar{z}-\bar{\lambda'}_{1},\bar{\lambda}}(x))^{-1} \bar{S}_{W} (p_{\bar{\alpha},\bar{\mu}}(T))^{-1}.
\end{equation*}

Notice that both $\mathcal{S}_{V}$ and $\mathcal{S}_{W}$ have the form $1+\sum_{l,m\leq 1}C_{lm}x^{l}T^{m}$. Therefore, by Lemma~\ref{invertible pseudodifference operators}, the operators $\mathcal{S}_{V}$ and $\mathcal{S}_{W}$ are invertible in $\Psi\mathfrak{D}_{q}$.

Recall the maps $\mathfrak{T}_{1}$ and $\mathfrak{T}_{3}$, see formulas \eqref{T_1} and \eqref{T_3}, respectively. Denote $1-\bar{z}-\bar{\lambda'}_{1}=\big(1-z_{1}-\big(\lambda^{(1)}\big)'_{1}\lc 1-z_{k}-\big(\lambda^{(k)}\big)'_{1}\big)$.

\begin{thm}\label{discrete main 1}
Consider a space $V\in \mathcal{P}\big(\bar{z},\bar{\lambda};\bar{\alpha},\bar{\mu}\big)$. Denote $W=\mathfrak{T}_{1}(\mathfrak{T}_{3}(V))\in\mathcal{E}\big(\bar{\alpha},\bar{\mu}';1-\bar{z}-\bar{\lambda'}_{1},\bar{\lambda}'\big)$. Let~$\mathcal{S}_{V}$ and $\mathcal{S}_{W}$ be the fundamental pseudo-difference operators of $V$ and $W$, respectively. Then
\[
\mathcal{S}_{W}=\mathcal{S}_{V}^{-1}.
\]
\end{thm}
\begin{proof}
For any pseudo-difference operator $\mathcal{S}=\sum_{i=-\infty}^{N}\sum_{j=-\infty}^{K}C_{ij}x^{i}T^{j}$, define a pseudo-dif\-fer\-ence operator $\mathcal{S}^{\ddagger}$ by
\begin{equation}\label{ddagger}
\mathcal{S}^{\ddagger}=\sum_{i=-\infty}^{N}\sum_{j=-\infty}^{K}C_{ij}T^{j}(-x)^{i}.
\end{equation}
It is easy to check that $(\cdot)^{\ddagger}$ is an involutive antiautomorphism on $\Psi\mathfrak{D}_{q}$.

Let $V$ be a space of quasi-polynomials with the data $\big(\bar{z},\bar{\lambda};\bar{\alpha},\bar{\mu}\big)$. Let $\bar{D}_{V}$ be the fundamental regularized differential operator of $V$. Denote $\bar{S}_{V}=\tau\big(\bar{D}_{V}\big)$, where $\tau$ is given by formula \eqref{tau}.

Denote $U=\mathfrak{T}_{3}(V)\in\mathcal{E}\big(\bar{\alpha},\bar{\mu};\bar{z}+\bar{\lambda'}_{1},\bar{\lambda}\big)$. Let $S_{U}$ be the fundamental difference operator of~$U$. Then $\bar{S}_{U}=q_{\bar{z},\bar{\lambda}}(x)S_{U}$ is the regularized fundamental difference operator of~$U$, where the polynomial $q_{\bar{z},\bar{\lambda}}(x)$ is defined in formula~\eqref{q}.
We have $\bar{S}_{U}=\bar{D}_{V}^{\#}=\bar{S}_{V}^{\ddagger}$.

Therefore, for the fundamental pseudo-difference operator $\mathcal{S}_{V}$ of~$V$, we get
\begin{align}
\mathcal{S}_{V}^{\ddagger} & = \big((q_{\bar{z},\bar{\lambda}}(-x))^{-1}\big)^{\ddagger}\big(\bar{S}_{V}\big)^{\ddagger}\big((p_{\bar{\alpha},\bar{\mu}}(T))^{-1}\big)^{\ddagger} = (q_{\bar{z},\bar{\lambda}}(x))^{-1} \bar{S}_{U} (p_{\bar{\alpha},\bar{\mu}}(T))^{-1}\nonumber\\
& = S_{U} (p_{\bar{\alpha},\bar{\mu}}(T))^{-1}.\label{5.1}
\end{align}

By construction, for the fundamental difference operator $Q^{\rightarrow}(S_{U})$ of $\mathfrak{T}_{1}(U)$, see~\eqref{Q rightarrow}, we have
\[
p_{\bar{\alpha},\bar{\mu}'}(T) p_{\bar{\alpha},\bar{\mu}}(T)=(Q^{\rightarrow}(S_{U}))^{\ddagger}S_{U}.
\]

Let us rewrite the last formula as follows
\[
\big[(p_{\bar{\alpha},\bar{\mu}'}(T))^{-1} (Q^{\rightarrow}(S_{U}))^{\ddagger}\big] \big[S_{U} (p_{\bar{\alpha},\bar{\mu}}(T))^{-1}\big]=1.
\]

This, together with formula \eqref{5.1}, gives
\begin{equation}\label{5.2}
\big(\mathcal{S}_{V}^{\ddagger}\big)^{-1}=(p_{\bar{\alpha},\bar{\mu}'}(T))^{-1} (Q^{\rightarrow}(S_{U}))^{\ddagger}.
\end{equation}

Applying the involutive antiautomorphism $(\cdot )^{\ddagger}$ to both sides of equation \eqref{5.2}, we obtain
\begin{equation}\label{p3.2_1}
\mathcal{S}_{V}^{-1}=Q^{\rightarrow}(S_{U})(p_{\bar{\alpha},\bar{\mu}'}(T))^{-1}.
\end{equation}

Let $\mathcal{S}_{W}$ be the fundamental pseudo-difference operator of $W$. By definition, we have $\mathcal{S}_{W}=Q^{\rightarrow}(S_{U})(p_{\bar{\alpha},\bar{\mu}'}(T))^{-1}$. Therefore, formula~\eqref{p3.2_1} gives $\mathcal{S}_{V}^{-1}=\mathcal{S}_{W}$.

Theorem \ref{discrete main 1} is proved.
\end{proof}

\subsection{Proof of Theorem \ref{map T_2}}\label{proof of T_{2}}
For each space $V$ of quasi-polynomials with the data $\big(\bar{z},\bar{\lambda};\bar{\alpha},\bar{\mu}\big)$, define
\begin{equation}\label{commutative diagram relation}
\mathfrak{T}_{2}(V)=\mathfrak{T}_{3}^{-1}\mathfrak{T}_{1}\mathfrak{T}_{3}(V).
\end{equation}

Let $D_{V}$ be the fundamental differential operator of $V$. We need to show that $\mathfrak{T}_{2}(V)$ is annihilated by $\check{D}^{\dagger}_{V}$.
By definition, the regularized fundamental differential operator~$\bar{D}_{V}$ of~$V$ is given by the formula $\bar{D}_{V}=p_{\bar{\alpha},\bar{\mu}}(x) x^{L'} D_{V}$, where $p_{\bar{\alpha},\bar{\mu}}(x)$ is the polynomial defined in formula~\eqref{p}. Denote $\bar{S}_{V}=\tau(\bar{D}_{V})$, where $\tau$ is given by formula~\eqref{tau}. Then
\begin{equation}\label{4.15.1}
\tau\big(x^{L'}D_{V}\big)=\tau \big((p_{\bar{\alpha},\bar{\mu}}(x))^{-1}\big)\tau(\bar{D}_{V})=(p_{\bar{\alpha},\bar{\mu}}(T))^{-1}\bar{S}_{V}.
\end{equation}

Denote $l_{a} = \lambda^{(a)}_{1}+\big(\lambda^{(a)}\big)'_{1}-1$. By definition of $\check{D}_{V}$, we have
\begin{equation}\label{4.15.2}
\prod_{a=1}^{k}\prod_{b=0}^{l_{a}}\left(x\frac{{\rm d}}{{\rm d}x}-z_{a}-b\right)=\check{D}_{V}x^{L'}D_{V}.
\end{equation}

Applying the homomorphism $\tau$ to both sides of relation \eqref{4.15.2} and using formula~\eqref{4.15.1}, we get
\begin{gather}\label{p3.2}
\prod_{a=1}^{k}\prod_{b=0}^{l_{a}}(-x-z_{a}-b)=\tau\big(\check{D}_{V}\big) (p_{\bar{\alpha},\bar{\mu}}(T))^{-1} \bar{S}_{V}.
\end{gather}

Denote $\Delta _{a}=\{0\lc l_{a}\}\setminus\big\{\big(\lambda^{(a)}\big)'_{1}+\lambda^{(a)}_{b}-b,\,b=1\lc \big(\lambda^{(a)}\big)'_{1}\big\}$, and set
\[\bar{q}_{\bar{z},\bar{\lambda}}(x)=\prod_{a=1}^{k}\prod_{b\in\Delta_{a}}(x-z_{a}-b).\]

Notice that \[\prod_{a=1}^{k}\prod_{b=0}^{l_{a}}(-x-z_{a}-b)=\bar{q}_{\bar{z},\bar{\lambda}}(-x) q_{\bar{z},\bar{\lambda}}(-x),\]
where $q_{\bar{z},\bar{\lambda}}(x)$ is defined in formula~\eqref{q}.

Then we can rewrite relation \eqref{p3.2} as follows:
\begin{equation}\label{4.15.3}
\big[(\bar{q}_{\bar{z},\bar{\lambda}}(-x))^{-1} \tau(\check{D}_{V})\big] \big[(p_{\bar{\alpha},\bar{\mu}}(T))^{-1} \bar{S}_{V} (q_{\bar{z},\bar{\lambda}}(-x))^{-1}\big]=1,
\end{equation}

Since, by definition, $\mathcal{S}_{V}=(p_{\bar{\alpha},\bar{\mu}}(T))^{-1} \bar{S}_{V} (q_{\bar{z},\bar{\lambda}}(-x))^{-1}$, formula \eqref{4.15.3} gives
\begin{equation}\label{4.29}
\mathcal{S}_{V}^{-1}=(\bar{q}_{\bar{z},\bar{\lambda}}(-x))^{-1} \tau(\check{D}_{V}).
\end{equation}

Let $W=\mathcal{T}_{1}\mathcal{T}_{3}(V)$. Let $\mathcal{S}_{W}$ and $\bar{S}_{W}$ be the fundamental pseudo-difference operator of $W$ and the regularized fundamental difference operator of $W$, respectively. Denote $\bar{\eta}=\big(1-z_{1}-\big(\lambda^{(1)}\big)'_{1}-\lambda^{(1)}_{1}\lc 1-z_{k}-\big(\lambda^{(k)}\big)'_{1}-\lambda^{(k)}_{1}\big)$. Then by Theorem~\ref{discrete main 1}, we have
\begin{equation}\label{4.15.4}
\mathcal{S}_{V}^{-1}=\mathcal{S}_{W}=(q_{\bar{\eta},\bar{\lambda}'}(x))^{-1} \bar{S}_{W} (p_{\bar{\alpha},\bar{\mu}'}(T))^{-1}.
\end{equation}

Notice that for each $a=1\lc k$, $\Delta_{a}=\big\{(\lambda^{(a)})'_{1}-(\lambda^{(a)})'_{b}+b-1,\,b=1\lc\lambda^{(a)}_{1}\big\}$. This can be illustrated by enumerating sides of boxes in the Young diagram for the partition $\lambda^{(a)}$ similarly to what we did in the proof of Proposition~\ref{mualpha}. Using this description of $\Delta_{a}$, one can check that $\bar{q}_{\bar{z},\bar{\lambda}}(-x)=(-1)^{L'}q_{\bar{\eta},\bar{\lambda}'}(x)$. Therefore, formulas~\eqref{4.29} and~\eqref{4.15.4} give
\[
\bar{S}_{W}=(-1)^{L'}\tau\big(\check{D}_{V}\big) p_{\bar{\alpha},\bar{\mu}'}(T).
\]

Thus,
\[
\bar{S}_{W}^{\ddagger}=(-1)^{L'} p_{\bar{\alpha},\bar{\mu}'}(T)\big(\tau\big(\check{D}_{V}\big)\big)^{\ddagger}.
\]

Using that $\big(\tau\big(\check{D}_{V}\big)\big)^{\ddagger}=\tau\big(\check{D}_{V}^{\dagger}\big)$, we obtain
\[
(-1)^{L'} p_{\bar{\alpha},\bar{\mu}'}(x)\check{D}_{V}^{\dagger}=\tau^{-1}\big(\bar{S}_{W}^{\ddagger}\big)=\bar{S}_{W}^{\#}.
\]
Notice that by definition of the map $\mathfrak{T}_{3}$, the differential operator $\bar{S}_{W}^{\#}$ annihilates the space~$\mathfrak{T}_{2}(V)$, therefore, we proved that $\check{D}_{V}^{\dagger}$ annihilates $\mathfrak{T}_{2}(V)$.

The uniqueness of the space $\mathfrak{T}_{2}(V)$ follows from an analog of Lemma~\ref{unique kernel} for differential operators.

Theorem~\ref{map T_2} is proved.

\section[Duality for trigonometric Gaudin and dynamical Hamiltonians]{Duality for trigonometric Gaudin\\ and dynamical Hamiltonians}\label{duality}

\subsection[(gl\_k,gl\_n)-duality for trigonometric Gaudin and dynamical Hamiltonians]{$\boldsymbol{(\gl_{k},\gl_{n})}$-duality for trigonometric Gaudin and dynamical Hamiltonians}\label{trig Gaudin and Dyn}

Let $\mathfrak{X}_{n}$ be the vector
space of
all polynomials in anticommuting variables $\xi_{1}\lc\xi_{n}$.
Since $\xi_{i}\xi_{j}=-\xi_{j}\xi_{i}$ for any $i$, $j$, in particular,
$\xi_{i}^{2}=0$ for any~$i$,
the monomials $\xi_{i_{1}}\dots\xi_{i_{l}}$,
$1\leq i_{1}<i_{2}<\dots<i_{l}\leq n$, form
a basis of $\mathfrak{X}_{n}$. Notice that the space $\mathfrak{X}_{n}$ coincides with the exterior algebra of~$\C^{n}$.

The left derivations $\partial_1\lc\partial_n$ on $\mathfrak{X}_{n}$ are
linear maps such that
\begin{gather}
\label{leftderivation}
\partial_{i} (\xi_{j_{1}}\dots\xi_{j_{l}} ) =
\begin{cases}
(-1)^{s-1}\xi_{j_{1}}\dots\xi_{j_{s-1}} \xi_{j_{s+1}}\dots\xi_{j_{l}} ,
 &\text{if $i=j_s$ for some } s,
\\
=0,& \text{otherwise}.
\end{cases}
\end{gather}
It is easy to check that
$\partial_{i}\partial_{j}=-\partial_{j}\partial_{i}$ for any $i$, $j$,
in particular, $\partial_{i}^{2}=0$ for any $i$, and
$\partial_{i}\xi_{j}+\xi_{j}\partial_{i}=\delta_{ij}$ for any $i$, $j$.

Let $e_{ij}$, $i,j=1\lc n$, be the standard basis of the Lie algebra
$\gl_{n}$, in particular, we have
$[e_{ij},e_{kl}]=\delta_{jk}e_{il}-\delta_{il}e_{kj}$.
Define a $\gl_{n}$-action on~$\mathfrak{X}_{n}$ by
the rule
$e_{ij}\mapsto \xi_{i}\partial_{j}$. As a $\gl_{n}$-module, $\mathfrak{X}_{n}$ is isomorphic to $\bigoplus_{l=0}^{n}L_{\omega_{l}}$,
where $L_{\omega_{l}}$ is the irreducible finite-dimensional $\gl_{n}$-module of highest weight
\begin{equation*}
\omega_{l}=(\underset{l}{\underbrace{ 1\lc 1}},0\lc 0) .
\end{equation*}
The component $L_{\omega_l}$ in $\mathfrak{X}_{n}$ is spanned by the monomials of degree~$l$.

\begin{rem}
As we mentioned in Introduction, the $(\gl_{n},\gl_{k})$-duality for integrable systems was first studied in works of Mukhin, Tarasov, and Varchenko for the case, when instead of the space $\mathfrak{X}_{n}$, one considers the space $P_{n}$ of polynomials in commuting variables. The latter is also a $\gl_{n}$-module, and it decomposes into irreducibles as $\bigoplus_{i=1}^{\infty}L_{s_{i}}$, where $L_{s_{i}}$ is the irreducible finite-dimensional $\gl_{n}$-module of highest weight $s_{i}=(i,0,0,\dots)$.
\end{rem}

From now on, we will consider the Lie algebras $\gl_{n}$ and $\gl_{k}$ together.
We will write superscripts~$\langle n\rangle$ and~$\langle k\rangle$ to distinguish objects associated with algebras $\mathfrak{gl}_{n}$ and $\mathfrak{gl}_{k}$, respectively. For example, $e_{ij}^{\langle n\rangle}$, $i,j=1\lc n$,
is the basis of~$\mathfrak{gl}_{n}$, and $e_{ab}^{\langle k\rangle}$, $a,b=1\lc k$,
is the basis of~$\mathfrak{gl}_{k}$.

Let $\mathfrak{P}_{kn}$ be the vector space of polynomials in $kn$ pairwise anticommuting variables $\xi_{ai}$, $a=1\lc k$, $i=1\lc n$.
We have two vector space isomorphisms $\psi_{1} \colon (\mathfrak{X}_{k})^{\otimes n}\to \mathfrak{P}_{kn},$ and $\psi_{2}\colon (\mathfrak{X}_{n})^{\otimes k}\to \mathfrak{P}_{kn}$ given by
\begin{gather*}
\psi_{1} \colon \ (p_{1}\otimes\dots\otimes p_{n})\mapsto p_{1}(\xi_{11}\lc\xi_{k1})p_{2}(\xi_{12}\lc\xi_{k2})\cdots p_{n}(\xi_{1n}\lc\xi_{kn}),
\\
\psi_{2} \colon \ (p_{1}\otimes\dots\otimes p_{k})\mapsto p_{1}(\xi_{11}\lc\xi_{1n})p_{2}(\xi_{21}\lc\xi_{2n})\cdots p_{k}(\xi_{k1}\lc\xi_{kn}).
\end{gather*}

Let $\partial_{ai}$, $a=1\lc k$, $i=1\lc n$, be the left derivations on $\mathfrak{P}_{kn}$ defined similarly to the left derivations
on $\mathfrak{X_{n}}$, see~\eqref{leftderivation}. For any $g\in U(\gl_k)$, denote $g_{(i)}=1^{\otimes(i-1)} \otimes g\otimes 1^{\otimes(n-i)}
\in U(\gl_{k})^{\otimes n}$. We will identify the algebra $U(\gl_{k})$ and its image under the
diagonal embedding
$g\mapsto\sum_{i=1}^{n}g_{(i)}\in U(\gl_{k})^{\otimes n}$. We will use similar conventions for $U(\gl_{n})^{\otimes k}$.
Define actions of~$U(\gl_{k})^{\otimes n}$ and $U(\gl_{n})^{\otimes k}$ on $\mathfrak{P}_{kn}$ by the formulas
\begin{gather}\label{action k}
\rho^{\langle k,n\rangle}\colon \ \big(e^{\langle k\rangle}_{ab}\big)_{(i)}\mapsto \xi_{ai}\partial_{bi},\\
\label{action n}
\rho^{\langle n,k\rangle}\colon \ \big(e^{\langle n\rangle}_{ij}\big)_{(a)}\mapsto \xi_{ai}\partial_{aj}.
\end{gather}
Then $\psi_{1}$ and $\psi_{2}$ are isomorphisms of $U(\gl_{k})^{\otimes n}$- and $U(\gl_{n})^{\otimes k}$-modules, respectively.

For any $i,j=1\lc n$, $i\neq j$, define the following elements of $U(\gl_{k})^{\otimes n}$
\begin{gather*}
\Omega^{+}_{(ij)}=\frac{1}{2} \underset{a=1}{\overset{k}{\sum}}\big(e^{\langle k\rangle}_{aa}\big)_{(i)} \big(e^{\langle k\rangle}_{aa}\big)_{(j)}+\underset{1\leq a<b\leq k}{\sum}\big(e^{\langle k\rangle}_{ab}\big)_{(i)} \big(e^{\langle k\rangle}_{ba}\big)_{(j)},\\
\Omega^{-}_{(ij)}=\frac{1}{2} \underset{a=1}{\overset{k}{\sum}}\big(e^{\langle k\rangle}_{aa}\big)_{(i)} \big(e^{\langle k\rangle}_{aa}\big)_{(j)}+\underset{1\leq a<b\leq k}{\sum}\big(e^{\langle k\rangle}_{ba}\big)_{(i)} \big(e^{\langle k\rangle}_{ab}\big)_{(j)}.
\end{gather*}

Fix sequences of pairwise distinct complex numbers $\bar{z}=(z_{1}\lc z_{k})$ and $\bar{\alpha}=(\alpha_{1}\lc \alpha_{n})$. For each $i=1\lc n$, define \textit{the trigonometric Gaudin Hamiltonians} $H^{\langle k,n\rangle}_{i}(\bar{\alpha},\bar{z})\in U(\gl_{k})^{\otimes n}$ by the following formula:
\begin{gather*}
H^{\langle k,n\rangle}_{i}(\bar{\alpha},\bar{z})=\sum_{a=1}^{k}\left(z_{a}-\frac{e^{\langle k\rangle}_{aa}}{2}\right)\big(e^{\langle k\rangle}_{aa}\big)_{(i)}+\sum_{\substack{j=1\\j\neq i}}^{n}\frac{\alpha_{i}\Omega^{+}_{(ij)}+\alpha_{j}\Omega^{-}_{(ij)}}{\alpha_{i}-\alpha_{j}}.
\end{gather*}
For each $i=1\lc n$, define \textit{the trigonometric dynamical Hamiltonians} $G^{\langle n,k\rangle}_{i}(\bar{z},\bar{\alpha})\in U(\gl_{n})^{\otimes k}$ by the following formula:
\begin{gather*}
G^{\langle n,k\rangle}_{i}(\bar{z},\bar{\alpha})= -\frac{(e^{\langle n\rangle}_{ii})^{2}}{2}+\sum_{a=1}^{k}z_{a}\big(e^{\langle n\rangle}_{ii}\big)_{(a)} \\
\hphantom{G^{\langle n,k\rangle}_{i}(\bar{z},\bar{\alpha})= }{}
+\sum_{\substack{j=1\\j\neq i}}^{n}\frac{\alpha_{j}}{\alpha_{i}-\alpha_{j}}\big(e^{\langle n\rangle}_{ij}e^{\langle n\rangle}_{ji}-e^{\langle n\rangle}_{ii}\big)+\sum_{j=1}^{n}\sum_{1\leq a<b\leq k}\big(e^{\langle n\rangle}_{ij}\big)_{(a)}(e^{\langle n\rangle}_{ji})_{(b)}.
\end{gather*}

Denote $-\bar{z}+1 =(-z_{1}+1\lc -z_{k}+1)$. Let $\rho^{\langle k,n\rangle}$ and $\rho^{\langle n,k\rangle}$ be the $U(\gl_{k})^{\otimes n}$ and $U(\gl_{n})^{\otimes k}$-actions on~$\mathfrak{P}_{kn}$ defined in formulas~\eqref{action k} and~\eqref{action n}, respectively. The following can be checked by a straightforward computation.

\begin{prop}\label{Gaudin Dyn duality}
For any $i=1\lc n$, we have
\[ \rho^{\langle k,n\rangle}\bigl(H_{i}^{\langle k,n\rangle}(\bar{\alpha},\bar{z})\bigr) = -\rho^{\langle n,k\rangle}\bigl(G_{i}^{\langle n,k\rangle}(-\bar{z}+1,\bar{\alpha})\bigr).\]
\end{prop}

Proposition \ref{Gaudin Dyn duality} is a~part of Theorem~4.4 in~\cite{TU2}. A similar identity for the case, when instead of the space~$\mathfrak{P}_{kn}$, we have the space $P_{kn}=S^{k}\C^{n}=S^{n}\C^{k}$ of polynomials in $kn$ commutative variables, can be found in~\cite{TV4}.

\subsection{Bethe ansatz method for trigonometric Gaudin model}\label{BA for trig Gaudin}
Fix sequences $\boldsymbol{l}=(l_{1}\lc l_{k})\in\Z_{\geq 0}^{k}$ and $\boldsymbol{m}=(m_{1}\lc m_{n})\in\Z_{\geq 0}^{n}$ such that $\sum_{a=1}^{k}l_{a}=\sum_{i=1}^{n}m_{i}$.
Let $\mathfrak{P}_{kn}[\boldsymbol{l},\boldsymbol{m}]\subset\mathfrak{P}_{kn}$ be the span of all monomials $\xi_{11}^{d_{11}}\cdots\xi_{k1}^{d_{k1}}\cdots \xi_{1n}^{d_{1n}}\cdots\xi_{kn}^{d_{kn}}$ such that $\sum_{a=1}^{k}d_{ai}=m_{i}$ and $\sum_{i=1}^{n}d_{ai}=l_{a}$. Assume that $\mathfrak{P}_{kn}[\boldsymbol{l},\boldsymbol{m}]\neq \{0\}$. We have
\[ \mathfrak{P}_{kn}[\boldsymbol{l},\boldsymbol{m}]=\big\{p\in \mathfrak{P}_{kn}\,|\,e^{\langle k\rangle}_{aa}p=l_{a}p,\, e^{\langle n\rangle}_{ii}p=m_{i}p,\,a=1\lc k,\,i=1\lc n\big\}.\]

Under the map $\psi_{1}$, the space $\mathfrak{P}_{kn}[\boldsymbol{l},\boldsymbol{m}]$ correspond to the weight subspace of weight $(l_{1}\lc l_{k})$ of the subrepresentation $L^{\langle k\rangle}_{\omega_{m_{1}}}\otimes\dots\otimes L^{\langle k\rangle}_{\omega_{m_{n}}}$ of $\mathfrak{X}_{k}^{\otimes n}=\big({\bigoplus}_{l=0}^{k}L^{\langle k\rangle}_{\omega_{l}}\big)^{\otimes n}$. Similarly, under the map $\psi_{2}$, the space $\mathfrak{P}_{kn}[\boldsymbol{l},\boldsymbol{m}]$ correspond to the weight subspace of weight $(m_{1}\lc m_{n})$ of the subrepresentation $L^{\langle n\rangle}_{\omega_{l_{1}}}\otimes\dots\otimes L^{\langle n\rangle}_{\omega_{l_{k}}}$ of $\mathfrak{X}_{n}^{\otimes k}=\big({\bigoplus}_{l=0}^{n}L^{\langle n\rangle}_{\omega_{l}}\big)^{\otimes k}$.

It is easy to check that all trigonometric Gaudin and dynamical Hamiltonians commute with elements $e^{\langle k\rangle}_{11}\lc e^{\langle k\rangle}_{kk},e^{\langle n\rangle}_{11}\lc e^{\langle n\rangle}_{nn}$. Therefore, $H^{\langle k,n\rangle}_{1}(\bar{\alpha},\bar{z}) \lc H^{\langle k,n\rangle}_{n}(\bar{\alpha},\bar{z})$, $G^{\langle n,k\rangle}_{1}(\bar{z},\bar{\alpha})\lc G^{\langle n,k\rangle}_{n}(\bar{z},\bar{\alpha})$ act on the subspace $\mathfrak{P}_{kn}[\boldsymbol{l},\boldsymbol{m}]$. We will be interested in the common eigenvectors of the Hamiltonians in the subspace $\mathfrak{P}_{kn}[\boldsymbol{l},\boldsymbol{m}]$.

For each $m\in\Z_{\geq 0}$, let $\omega_{m}$ be a partition given by $\omega_{m}=(1\lc 1,0,0,\dots)$ with $m$ ones. Define the sequence $\boldsymbol{l}_{0}=\big(l^{0}_{1}\lc l^{0}_{k}\big)$ by $l^{0}_{a}=\sum_{i=1}^{n}(\omega_{m_{i}})_{a}$.

For any sequence of integers $(c_{1}\lc c_{k})$ and for each $a=1\lc k-1$, define a transformation
\[ r_{a}\colon \ (c_{1}\lc c_{k})\mapsto (c_{1}\lc c_{a}-1,c_{a+1}+1\lc c_{k}).\]
Since $\sum_{a=1}^{k}l_{a}=\sum_{a=1}^{k}l^{0}_{a}=\sum_{i=1}^{n}m_{i}$, there exist integers $\bar{l}_{1}\lc\bar{l}_{k-1}$ such that $\boldsymbol{l}=r_{1}^{\bar{l}_{1}}\cdots r_{k-1}^{\bar{l}_{k-1}}\boldsymbol{l}_{0}$. It is easy to check that if $\bar{l}_{a}<0$ for some $a=1\lc k-1$, then $\mathfrak{P}_{kn}[\boldsymbol{l},\boldsymbol{m}]= 0$. Therefore, we can assume that $\bar{l}_{a}\geq 0$ for all $a=1\lc k-1$.

Put $\bar{l}_{0}=\bar{l}_{k}=0$. Then we have
\[ l_{a}=\sum_{i=1}^{n}(\omega_{m_{i}})_{a}+\bar{l}_{a-1}-\bar{l}_{a},\qquad a=1\lc k.\]
Therefore
\begin{equation}\label{bar l}
\bar{l}_{a}=\sum_{b=a+1}^{k}\left(l_{b}-\sum_{i=1}^{n}(\omega_{m_{i}})_{b}\right), \qquad a=0\lc k-1.
\end{equation}

Let $\boldsymbol{t}$ be a set of $\bar{l}_{1}+\dots +\bar{l}_{k-1}$ variables:
\[ \boldsymbol{t}=\big(t^{(1)}_{1}\lc t^{(1)}_{\bar{l}_{1}},t^{(2)}_{1}\lc t^{(2)}_{\bar{l}_{2}}\lc t^{(k-1)}_{1}\lc t^{(k-1)}_{\bar{l}_{k-1}}\big).\]

Fix sequences of pairwise distinct complex numbers $\bar{z}=(z_{1}\lc z_{k})$ and $\bar{\alpha}=(\alpha_{1}\lc \alpha_{n})$. Define \textit{the master function}:
\begin{gather}
\Phi (\boldsymbol{t},\bar{\alpha},\bar{z},\boldsymbol{l},\boldsymbol{m}) = \prod_{1\leq i<j\leq n}(\alpha_{i}-\alpha_{j})^{\min(m_{i},m_{j})}\prod_{i=1}^{n}\prod_{a=1}^{\bar{l}_{m_{i}}}\big(t_{a}^{(m_{i})}-\alpha_{i}\big)^{-1}
\prod_{i=1}^{n}\alpha_{i}^{\sum_{a=1}^{m_{i}}z_{a}+\frac{m_{i}}{2}}\nonumber\\
\hphantom{\Phi (\boldsymbol{t},\bar{\alpha},\bar{z},\boldsymbol{l},\boldsymbol{m}) =}{}
 \times \prod_{a=1}^{k-1}\prod_{b=1}^{\bar{l}_{a}}\big(t_{b}^{(a)}\big)^{z_{a+1}-z_{a}+1}
\prod_{a=1}^{k-1}\prod_{1\leq b<b'\leq \bar{l}_{a}}\big(t_{b}^{(a)}-t_{b'}^{(a)}\big)^{2} \nonumber\\
\hphantom{\Phi (\boldsymbol{t},\bar{\alpha},\bar{z},\boldsymbol{l},\boldsymbol{m}) =}{}
\times \prod_{a=1}^{k-2}\prod_{b=1}^{\bar{l}_{a}}\prod_{b'=1}^{\bar{l}_{a+1}}\big(t_{b}^{(a)}-t_{b'}^{(a+1)}\big)^{-1}.\label{Phi}
\end{gather}

The following equations are called the Gaudin Bethe ansatz equations:
\begin{equation}\label{BAE}
\left(\frac{1}{\Phi}\frac{\partial \Phi}{\partial t^{(a)}_{b}}\right)(\boldsymbol{t},\bar{\alpha},\bar{z},\boldsymbol{l},\boldsymbol{m})=0,\qquad a=1\lc k-1,\quad b=1\lc\bar{l}_{a}.
\end{equation}

We will call a solution $\boldsymbol{t}$ of the Gaudin Bethe ansatz equation \eqref{BAE} Gaudin admissible if
\begin{equation}\label{Gaudin admissible}
t_{i}^{(a)}\neq t_{j}^{(a)},\quad t_{i'}^{(b)}\neq t_{j'}^{(b+1)},\quad t_{i}^{(a)}\neq \alpha_{l}, \qquad t_{i}^{(a)}\neq 0
\end{equation}
for all $a=1\lc k-1$, $i,j=1\lc \bar{l}_{a}$, $i\neq j$, $b=1\lc k-2$, $i'=1\lc \bar{l}_{b}$, $j'=1\lc \bar{l}_{b+1}$, $l=1\lc n$.

We will also need a function constructed in \cite{MaV} and denoted there as $\phi (z,t)$. This function was introduced to obtain a hypergeometric solution of the trigonometric Knizhnik--Zamolodchikov (KZ) equations. The explicit formulas for $\phi (z,t)$ are rather lengthy, and we will not need them to formulate the statements below, so we omit them and instead, indicate how notations in \cite{MaV} match our notation. The parameters $z_{1}\lc z_{n}$ in \cite{MaV} correspond to $\alpha_{1}\lc \alpha_{n}$ in our paper, and variables $t_{b}^{(a)}$ in \cite{MaV} correspond to $t_{a}^{(b)}$ in our paper. We will write $\phi (\bar{\alpha}, \boldsymbol{t})$ for $\phi (z,t)$ with $z_{1}\lc z_{n}$ replaced by $\alpha_{1}\lc \alpha_{n}$ and $t_{b}^{(a)}$ replaced by $t_{a}^{(b)}$. The Lie algebra $\gl_{N}$ in \cite{MaV} corresponds to $\gl_{k}$ here, and for the $\gl_{N}$-weights $\Lambda_{1}\lc \Lambda_{n}$, $\nu$ in \cite{MaV}, we should take the $\gl_{k}$-weights $\omega_{m_{1}}\lc\omega_{m_{n}}$, $(l_{1}\lc l_{k})$, respectively. Then under the identification $\psi_{1}$, $\phi (\bar{\alpha}, \boldsymbol{t})$ becomes a $\mathfrak{P}_{kn}[\boldsymbol{l},\boldsymbol{m}]$-valued function.

\begin{thm}\label{m function to eigenvalues}
Let $\boldsymbol{t}$ be a Gaudin admissible solution of the Gaudin Bethe ansatz equations \eqref{BAE}. Suppose that $\phi (\bar{\alpha}, \boldsymbol{t})\neq 0$. Then $\phi (\bar{\alpha}, \boldsymbol{t})$ is a common eigenvector of the Gaudin Hamiltonians, and for each $i=1\lc n$, the corresponding eigenvalue $h^{\langle k,n\rangle}_{i}(\boldsymbol{t},\bar{\alpha},\bar{z},\boldsymbol{l},\boldsymbol{m})$ of $H^{\langle k,n\rangle}_{i}(\bar{\alpha},\bar{z})$ is given by
\begin{equation}\label{Phi to h}
h^{\langle k,n\rangle}_{i}(\boldsymbol{t},\bar{\alpha},\bar{z},\boldsymbol{l},\boldsymbol{m})=\left(\alpha_{i}\frac{\partial}{\partial\alpha_{i}}\ln\Phi\right)(\boldsymbol{t},\bar{\alpha},\bar{z}-\boldsymbol{l},\boldsymbol{l},\boldsymbol{m}),
\end{equation}
where $\bar{z}-\boldsymbol{l}=(z_{1}-l_{1},z_{2}-l_{2}\lc z_{k}-l_{k})$.
\end{thm}
\begin{proof}
The theorem can be proved by applying the steepest descend method to hypergeometric solutions of the trigonometric KZ equations. We refer a reader to the work \cite{RV}, where the method was applied to hypergeometric solutions of the rational KZ equations. Theorem \ref{m function to eigenvalues} is the modification of Corollary 4.16 in \cite{RV} to the trigonometric case.
\end{proof}

\subsection[Spaces of quasi-polynomials and eigenvalues of trigonometric Gaudin Hamiltonians]{Spaces of quasi-polynomials and eigenvalues\\ of trigonometric Gaudin Hamiltonians}\label{qpol and BA}

Fix a pair $(\boldsymbol{l}, \boldsymbol{m})$ like in the previous section. Assume additionally that $l_{a}\neq 0$ and $m_{i}\neq 0$ for all $a=1\lc k$, $i=1\lc n$. Assume that $\mathfrak{P}_{kn}[\boldsymbol{l},\boldsymbol{m}]\neq \{0\}$. Define the sequence of partitions $\bar{\lambda}=\big(\lambda^{(1)}\lc\lambda^{(k)}\big)$ by $\lambda^{(a)}=(l_{a},0,0,\dots)$, $a=1\lc k$. Recall that for each $m\in\Z_{\geq 0}$, $\omega_{m}$~is a~partition given by $\omega_{m}=(1\lc 1,0,0,\dots)$ with $m$ ones. Define a sequence of partitions $\bar{\mu}=(\omega_{m_{1}}\lc\omega_{m_{n}})$.

Let $\bar{z}=(z_{1}\lc z_{k})$ be a sequence of complex numbers such that $z_{a}-z_{b}\notin\Z$ for $a\neq b$. Let $\bar{\alpha}=(\alpha_{1}\lc\alpha_{n})$ be a sequence of pairwise distinct non-zero complex numbers. Let $V$ be a space of quasi-polynomials with the data $\big(\bar{z},\bar{\lambda};\bar{\alpha},\bar{\mu}\big)$. Then~$V$ has a basis of the form
\[
\big\{x^{z_{1}}q_{1}(x),x^{z_{2}}q_{2}(x)\lc x^{z_{k}}q_{k}(x)\big\},
\]
where $q_{1}(x)\lc q_{k}(x)$ are polynomials and $\deg q_{a}(x)=l_{a}$.

For each $a=1\lc k-1$, $b=1\lc k$, define
\begin{gather*}
T_{b}(x)=\prod_{\substack{i=1\\m_{i}\geq b}}^{n}(x-\alpha_{i}),\\
y_{a}(x)=\frac{\Wr\big(x^{z_{k}}q_{k}(x),x^{z_{k-1}}q_{k-1}(x)\lc x^{z_{a+1}}q_{a+1}(x)\big)}{\prod_{b=a+1}^{k}\big(x^{z_{b}-k+b}T_{b}(x)\big)}.
\end{gather*}

One can check that for each $a=1\lc k-1$, $y_{a}(x)$ is a polynomial of degree $\bar{l}_{a}$. The polynomials $q_{1}(x)\lc q_{k}(x)$ can be normalized in such a way that the polynomials $y_{0}(x)\lc y_{n-1}(x)$ are monic. Write
\[ y_{a}(x)=\prod_{b=1}^{\bar{l}_{a}}\big(x-\tilde{t}^{(a)}_{b}\big).\]
We will call the space $V$ Gaudin admissible if the tuple
\[ \tilde{\boldsymbol{t}}=\big(\tilde t^{(1)}_{1}\lc \tilde t^{(1)}_{\bar{l}_{1}},\tilde t^{(2)}_{1}\lc \tilde t^{(2)}_{\bar{l}_{2}}\lc \tilde t^{(k-1)}_{1}\lc \tilde t^{(k-1)}_{\bar{l}_{k-1}}\big)\]
satisfies conditions \eqref{Gaudin admissible}.

The following theorem was proved in \cite{MV}.
\begin{thm}\label{q-pol to BAE}
Let $V$ be Gaudin admissible.
Then $\tilde{\boldsymbol{t}}$
is a Gaudin admissible solution of the Gaudin Bethe ansatz equations \eqref{BAE}.
\end{thm}

Define functions $\beta_{1}(x)\lc \beta_{k}(x)$ by the following formula:
\[
x^{k}D_{V}=\left(x\frac{{\rm d}}{{\rm d}x}\right)^{k}+\sum_{a=1}^{k}\beta_{a}(x)\left(x\frac{{\rm d}}{{\rm d}x}\right)^{k-a}.
\]
By Lemma~\ref{properties of beta}, the functions $\beta_{1}(x)\lc \beta_{k}(x)$ are rational.

Let $\tilde{\boldsymbol{t}}$ be the Gaudin admissible solution of the Gaudin Bethe ansatz equation corresponding to $V$, like in Theorem~\ref{q-pol to BAE}. Suppose that $\phi\big(\bar{\alpha},\tilde{\boldsymbol{t}}\big)\neq 0$. Denote $\bar{z}+\boldsymbol{l}=(z_{1}+l_{1},z_{2}+l_{2}\lc z_{k}+l_{k})$. According to Theorem~\ref{m function to eigenvalues}, $\phi\big(\bar{\alpha},\tilde{\boldsymbol{t}}\big)$ is a common eigenvector of the trigonometric Gaudin Hamiltonians, and for each $i=1\lc n$, the corresponding eigenvalue of $H^{\langle k,n\rangle}_{i}(\bar{\alpha},\bar{z}+\boldsymbol{l})$ is $h_{i}^{V}\coloneqq h^{\langle k,n\rangle}_{i}\big(\tilde{\boldsymbol{t}},\bar{\alpha},\bar{z}+\boldsymbol{l},\boldsymbol{l},\boldsymbol{m}\big)$. We will also call $\phi\big(\tilde{\boldsymbol{t}},\bar{\alpha}\big)$ the Bethe vector $v_{V}$ corresponding to~$V$.
\begin{prop}\label{prop5.4} The following holds
\begin{equation}\label{p5.4}
h_{i}^{V}=\frac{1}{\alpha_{i}}\Res_{x=\alpha_{i}}\left(\frac{1}{2}\beta_{1}^{2}(x)-\beta_{2}(x)\right)+\frac{m_{i}^{2}}{2}-m_{i}.
\end{equation}
\end{prop}
\begin{proof}
For each function $g$ of $x$, write $\ln'(g)=(\ln(g))'$, where $(\cdot)'$ is the differentiation with respect to $x$. By an analog of Proposition 2 for differential operators, see \cite{TU1}, we have
\begin{align}
D_{V}={} & \left(\frac{{\rm d}}{{\rm d}x}-\ln'\left(\frac{x^{z_{1}-k+1}T_{1}(x)}{y_{1}(x)}\right)\right)
\left(\frac{{\rm d}}{{\rm d}x}-\ln'\left(\frac{x^{z_{2}-k+2}T_{2}(x)y_{1}}{y_{2}(x)}\right)\right)\cdots\nonumber\\
&{}\times \left(\frac{{\rm d}}{{\rm d}x} -\ln'\left(\frac{x^{z_{k-1}-1}T_{k-1}(x)y_{k-2}(x)}{y_{k-1}(x)}\right)\right)\left(\frac{{\rm d}}{{\rm d}x} -\ln'\big(x^{z_{k}}T_{k}y_{k-1}(x)\big)\right).
\label{D_V}
\end{align}
Multiplying each side of \eqref{D_V} by $x^{k}$, we get
\begin{align}
x^{k}D_{V}={} & \left(x\frac{{\rm d}}{{\rm d}x}-x\ln'\left(\frac{T_{1}(x)}{y_{1}(x)}\right)-z_{1}\right) \left(x\frac{{\rm d}}{{\rm d}x}-x\ln'\left(\frac{T_{2}(x)y_{1}}{y_{2}(x)}\right)-z_{2}\right)\cdots\nonumber\\
&{}\times \left(x\frac{{\rm d}}{{\rm d}x}-x\ln'\left(T_{k}y_{k-1}(x)\right)-z_{k}\right).\label{x^kD_V}
\end{align}
Put $y_{0}(x)=y_{k}(x)=1$. For each $a=1\lc k$, denote
\[
Y_{a}=-x\ln'\left(\frac{T_{a}(x)y_{a-1}(x)}{y_{a}(x)}\right)-z_{a}.
\]
By formula \eqref{x^kD_V}, we have
\begin{equation}\label{p5.4_1}
\beta_{2}(x)=\sum_{1\leq a<b\leq k}Y_{a}(x)Y_{b}(x)+\sum_{a=1}^{k}xY_{a}'(x),\qquad \beta_{1}(x)=\sum_{a=1}^{k}Y_{a}(x).
\end{equation}

Since $\tilde{\boldsymbol{t}}$ is Gaudin admissible, for each $i=1\lc n$, $a=1\lc k-1$, $\alpha_{i}$ is not a root of the polynomial $y_{a}(x)$. Also, for each $i=1\lc n$, $\alpha_{i}$ is a root of the polynomial $T_{a}(x)$ if and only if $a\leq m_{i}$. Using this, we can compute:
\begin{gather}
\frac{1}{\alpha_{i}}\Res_{x=\alpha_{i}} \Bigg(\sum_{1\leq a<b\leq k}Y_{a}(x)Y_{b}(x)\Bigg)\nonumber \\
\qquad{} =\sum_{b=1}^{\bar{l}_{a}}\frac{\alpha_{i}}{\alpha_{i}-\tilde{t}_{b}^{(m_{i})}}+\sum_{a=1}^{m_{i}}\sum_{\substack{b=1\\b\neq a}}^{k}\Bigg(z_{b}+\sum_{\substack{j=1\\m_{j}\geq b}}^{n}\frac{\alpha_{i}}{\alpha_{i}-\alpha_{j}}\Bigg)+m_{i}(m_{i}-1),
\\
\frac{1}{\alpha_{i}}\Res_{x=\alpha_{i}}\Bigg(\sum_{a=1}^{k}xY_{a}'(x)\Bigg)=\frac{m_{i}(m_{i}-1)}{2},
\\
\label{p5.4_2}
\frac{1}{\alpha_{i}}\Res_{x=\alpha_{i}}\Bigg(\frac{1}{2}\Bigg(\sum_{a=1}^{k}Y_{a}(x)\Bigg)^{2}\Bigg)=\sum_{a=1}^{m_{i}}\sum_{b=1}^{k}\Bigg(z_{b}+\sum_{\substack{j=1\\m_{j}\geq b}}^{n}\frac{\alpha_{i}}{\alpha_{i}-\alpha_{j}}\Bigg)+m_{i}^{2}.
\end{gather}

From formulas \eqref{p5.4_1}--\eqref{p5.4_2}, we get
\begin{gather}
\frac{1}{\alpha_{i}}\Res_{x=\alpha_{i}} \left(\frac{1}{2}\beta_{1}^{2}(x)-\beta_{2}(x)\right) \nonumber \\
\qquad{} =\sum_{b=1}^{\bar{l}_{a}}\frac{\alpha_{i}}{\tilde{t}_{b}^{(m_{i})}-\alpha_{i}}+\sum_{a=1}^{m_{i}}z_{a}+\sum_{\substack{j=1\\j\neq i}}^{n}\frac{\alpha_{i}\min(m_{i},m_{j})}{\alpha_{i}-\alpha_{j}}-\frac{m_{i}^{2}}{2}+\frac{3}{2}m_{i}.\label{p5.4_3}
\end{gather}

On the other hand, using formula \eqref{Phi}, we can compute
\begin{gather}
\left(\alpha_{i}\frac{\partial}{\partial\alpha_{i}}\ln\Phi\right) (\tilde{\boldsymbol{t}},\bar{\alpha},\bar{z},\boldsymbol{l},\boldsymbol{m})\nonumber\\
\qquad{} = \sum_{b=1}^{\bar{l}_{a}}\frac{\alpha_{i}}{\tilde{t}_{b}^{(m_{i})}-\alpha_{i}}+\sum_{a=1}^{m_{i}}z_{a}+\sum_{\substack{j=1\\j\neq i}}^{n}\frac{\alpha_{i}\min(m_{i},m_{j})}{\alpha_{i}-\alpha_{j}}+\frac{m_{i}}{2}.\label{p5.4_4}
\end{gather}
Comparing formulas \eqref{p5.4_3}, \eqref{p5.4_4}, and \eqref{Phi to h}, we get relation \eqref{p5.4}.
\end{proof}

\subsection{Bethe ansatz method for XXX-type spin chain model}\label{BA for XXX}
Fix sequences $\boldsymbol{l}=(l_{1}\lc l_{k})\in\Z_{\geq 0}^{k}$ and $\boldsymbol{m}=(m_{1}\lc m_{n})\in\Z_{\geq 0}^{n}$ such that $\sum_{a=1}^{k}l_{a}=\sum_{i=1}^{n}m_{i}$. Assume that $\mathfrak{P}_{kn}[\boldsymbol{l},\boldsymbol{m}]\neq \{0\}$. Unlike in the previous section, we do not assume that $l_{a}\neq 0$ and $m_{i}\neq 0$ for all $a=1\lc k$, $i=1\lc n$.
For each $i=0\lc n-1$, define
\begin{equation}\label{m bar}
\bar{m}_{i}=\sum_{j=i+1}^{n}\left(m_{j}-\sum_{a=1}^{k}(\omega_{l_{a}})_{j}\right).
\end{equation}
The numbers $\bar{m}_{1}\lc \bar{m}_{n-1}$ are the $(\gl_{k},\gl_{n})$-dual analogs of the numbers $\bar{l}_{1}\lc \bar{l}_{k-1}$, see formula~\eqref{bar l}. Recall that $\mathfrak{P}_{kn}[\boldsymbol{l},\boldsymbol{m}]\neq\{0\}$ implies $\bar{l}_{a}\geq 0$, $a=0\lc k-1$. Similarly, $\mathfrak{P}_{kn}[\boldsymbol{l},\boldsymbol{m}]\neq\{0\}$ implies $\bar{m}_{i}\geq 0$, $i=0\lc n-1$.

Let $\boldsymbol{t}$ be a set of $\bar{m}_{1}+\dots +\bar{m}_{n-1}$ variables:
\[
\boldsymbol{t}=\big(t^{(1)}_{1}\lc t^{(1)}_{\bar{m}_{1}},t^{(2)}_{1}\lc t^{(2)}_{\bar{m}_{2}}\lc t^{(n-1)}_{1}\lc t^{(n-1)}_{\bar{m}_{n-1}}\big).
\]

Fix sequences of pairwise distinct complex numbers $\bar{z}=(z_{1}\lc z_{k})$ and $\bar{\alpha}=(\alpha_{1}\lc \alpha_{n})$. We have $\bar{m}_{0}=0$. Also, put $\bar{m}_{n}=0$. The XXX Bethe ansatz equations is the following system of $\bar{m}_{1}+\dots +\bar{m}_{n-1}$ equations:
\begin{gather}\label{XXXBAE}
\frac{\alpha_{i+1}}{\alpha_{i}}=\prod_{\substack{a=1\\l_{a}=i}}^{k}\frac{t_{b}^{(l_{a})}-z_{a}+1}{t_{b}^{(l_{a})}-z_{a}}\prod_{a=1}^{\bar{m}_{i-1}}\frac{t_{b}^{(i)}-t_{a}^{(i-1)}+1}{t_{b}^{(i)}-t_{a}^{(i-1)}}\prod_{a=1}^{\bar{m}_{i+1}}\frac{t_{b}^{(i)}-t_{a}^{(i+1)}}{t_{b}^{(i)}-t_{a}^{(i-1)}-1}\prod_{\substack{a=1\\a\neq b}}^{\bar{m}_{i}}\frac{t_{b}^{(i)}-t_{a}^{(i)}-1}{t_{b}^{(i)}-t_{a}^{(i)}+1},\!\!\!
\end{gather}
where $i=1\lc n-1$, $b=1\lc \bar{m}_{i}$.

A solution $\boldsymbol{t}$ of the XXX Bethe ansatz equations \eqref{XXXBAE} is called XXX-admissible if $t^{(i)}_{a}\neq t^{(i)}_{b}$, $t^{(j)}_{a'}\neq t^{(j+1)}_{b'}$ for any $i=1\lc n-1$, $a,b=1\lc \bar{m}_{i}$, $a\neq b$, $j=1\lc n-2$, $a'=1\lc \bar{m}_{j}$, $b'=1\lc \bar{m}_{j+1}$.

For each $i,j=1\lc n$, define
\begin{gather}\label{X i}
\mathcal{X}_{i}(x, \boldsymbol{t}, \bar{z}, \bar{\alpha})=\alpha_{i}\prod_{\substack{a=1\\l_{a}\geq i}}^{k}\frac{x-z_{a}+1}{x-z_{a}}\prod_{a=1}^{\bar{m}_{i-1}}\frac{x-t^{(i-1)}_{a}+1}{x-t^{(i-1)}_{a}}\prod_{a=1}^{\bar{m}_{i}}\frac{x-t^{(i)}_{a}-1}{x-t^{(i)}_{a}},
\\
\label{tilde E}
\tilde{E}_{j}(x, \boldsymbol{t}, \bar{z}, \bar{\alpha})=\sum_{1\leq i_{1}<\dots <i_{j}\leq n}\mathcal{X}_{i_{1}}(x)\mathcal{X}_{i_{2}}(x-1)\dots \mathcal{X}_{i_{j}}(x-j+1).
\end{gather}
In the last formula $\mathcal{X}_{i}(x)=\mathcal{X}_{i}(x, \boldsymbol{t}, \bar{z}, \bar{\alpha})$, $i=1\lc n$.

Introduce a new variable $u$. Consider the following polynomial in $u$:
\[
E(u,x,\boldsymbol{t}, \bar{z}, \bar{\alpha})=u^{n}+\sum_{j=1}^{n}\tilde{E}_{j}(x,\boldsymbol{t}, \bar{z}, \bar{\alpha})u^{n-i},
\]
which is also a rational function of $x$ regular at infinity. Let $E_{a}(u,\boldsymbol{t}, \bar{z}, \bar{\alpha})$, $a\in\Z_{\geq 0}$ be the coefficients of the Laurent series at infinity of $E(u,x,\boldsymbol{t}, \bar{z}, \bar{\alpha})$ as a function of $x$:
\begin{equation}\label{E expansion}
E(u,x,\boldsymbol{t}, \bar{z}, \bar{\alpha})=\sum_{a=0}^{\infty}x^{-a}E_{a}(u,\boldsymbol{t}, \bar{z}, \bar{\alpha}).
\end{equation}

In \cite{MTV6}, a certain function $\psi_{i}(\boldsymbol{t},\bar{z})$ of $\boldsymbol{t}$ called the \textit{universal weight function for the XXX-type spin chain model} was defined. This function takes values in tensor products of highest weight $\gl_{n}$-modules. In the case that we need, $\psi_{i}(\boldsymbol{t},\bar{z})$ is a $\mathfrak{P}_{kn}[\boldsymbol{l},\boldsymbol{m}]$-valued function. If $\boldsymbol{t}$ is an XXX-admissible solution of the XXX Bethe ansatz equations \eqref{XXXBAE}, and $\psi_{i}(\boldsymbol{t},\bar{z})\neq 0$, then $\psi_{i}(\boldsymbol{t},\bar{z})$ is a~common eigenvector of
the higher transfer matrices for the XXX-type spin chain model. Higher transfer matrices are series in $x^{-1}$, whose coefficients generate a large commutative subalgebra called the XXX Bethe subalgebra inside the Yangian $Y(\gl_{n})$. The XXX Bethe subalgebra depends on parameters $\bar{\alpha}=(\alpha_{1}\lc\alpha_{n})$. The algebra $Y(\gl_{n})$ acts on $\mathfrak{P}_{kn}$. This action depends on parameters $\bar{z}=(z_{1}\lc z_{k})$. Therefore, we have a homomorphism $\rho^{Y}_{\bar{z}}\colon Y(\mathfrak{gl}_{n})\rightarrow\End (\mathfrak{P}_{kn})$.

The images of the trigonometric dynamical Hamiltonians under the action $\rho^{\langle n,k\rangle}\colon (U(\mathfrak{gl}_{n}))^{\otimes k}\allowbreak \rightarrow \End (\mathfrak{P}_{kn})$ introduced in formula \eqref{action n} can be considered as elements of the image of the XXX Bethe subalgebra under the map $\rho^{Y}_{\bar{z}}$, see \cite[Appendix~B]{MTV6}. In particular, if $\boldsymbol{t}$ is an XXX-admissible solution of the XXX Bethe ansatz equations~\eqref{XXXBAE}, and $\psi_{i}(\boldsymbol{t},\bar{z})\neq 0$, then $\psi_{i}(\boldsymbol{t},\bar{z})$ is a common eigenvector of the dynamical Hamiltonians, and the corresponding eigenvalue can be computed using \cite[Proposition~B.1]{MTV6}. We will formulate the result in the following theorem:
\begin{thm}\label{Bethe eigenvectors for G}
Let $\boldsymbol{t}$ be an XXX-admissible solution of the XXX Bethe ansatz equations~\eqref{XXXBAE}. Then for each $i=1\lc n$, we have
\[
G^{\langle n,k\rangle}_{i}(\bar{z},\bar{\alpha})\psi_{i}(\boldsymbol{t},\bar{z})=g^{\langle n,k\rangle}_{i}(\boldsymbol{t},\bar{z},\bar{\alpha})\psi_{i}(\boldsymbol{t},\bar{z}),
\]
where
\begin{equation}\label{g}
g^{\langle n,k\rangle}_{i}(\boldsymbol{t},\bar{z},\bar{\alpha})=-\frac{1}{\alpha_{i}}\Res_{u=\alpha_{i}}\frac{E_{2}(u,\boldsymbol{t}, \bar{z}, \bar{\alpha})}{\prod_{j=1}^{n}(u-\alpha_{i})}+\sum_{\substack{j=1\\j\neq i}}^{n}\frac{\alpha_{j}m_{i}m_{j}}{\alpha_{i}-\alpha_{j}}-\frac{m_{i}^{2}}{2},
\end{equation}
and $E_{2}(u,\boldsymbol{t}, \bar{z}, \bar{\alpha})$ is the coefficient in the expansion \eqref{E expansion}.
\end{thm}

\subsection[Spaces of quasi-exponentials and eigenvalues of trigonometric dynamical Hamiltonians]{Spaces of quasi-exponentials and eigenvalues\\ of trigonometric dynamical Hamiltonians}\label{qexp to BAE}

Assume again that $l_{a}\neq 0$ and $m_{i}\neq 0$ for all $a=1\lc k$, $i=1\lc n$. Let the data $\big(\bar{\alpha},\bar{\mu};\bar{z},\bar{\lambda}\big)$ be like in Section~\ref{qpol and BA}, and let $W$ be a space of quasi-exponentials with the difference data $\big(\bar{\alpha},\bar{\mu}';-\bar{z},\bar{\lambda}'\big)$. Then~$W$ has a basis of the form
\[
\big\{\alpha_{1}^{x}r_{1}(x),\alpha_{2}^{x}r_{2}(x)\lc \alpha_{n}^{x}r_{n}(x)\big\},
\]
where $r_{1}(x)\lc r_{n}(x)$ are polynomials and $\deg r_{i}(x)=m_{i}$.

For each $i=1\lc n$, define
\begin{equation}\label{difference T}
T_{i}(x)=\prod_{\substack{a=1\\l_{a}\geq i}}^{k}(x+z_{a}+l_{a}-i).
\end{equation}

The following lemma is a special case of Lemma 3.7 in \cite{MTV4}:
\begin{lem}\label{T is local frame}
For each $i=0\lc n-1$, $j_{1}\lc j_{n-i}\in\{1\lc n\}$, the functions
\[
\frac{\mathcal{W}{\rm r}\big(\alpha_{j_{1}}^{x}r_{j_{1}}(x),\alpha_{j_{2}}^{x}r_{j_{2}}(x)\lc \alpha_{j_{n-i}}^{x}r_{j_{n-i}}(x)\big)}{\prod_{l=i+1}^{n}\big(\alpha_{j_{n-l+1}}^{x}T_{j}(x)\big)}
\]
 are polynomials.
\end{lem}

For each $i=0\lc n-1$, $j=1\lc n$, define
\begin{gather}\label{difference y}
y_{i}(x)=\frac{\mathcal{W}{\rm r}\big(\alpha_{n}^{x}r_{n}(x),\alpha_{n-1}^{x}r_{n-1}(x)\lc \alpha_{i+1}^{x}r_{i+1}(x)\big)}{\prod_{j=i+1}^{n}\big(\alpha_{j}^{x}T_{j}(x)\big)},
\qquad \tilde{T}_{j}(x)=\prod_{\substack{a=1\\l_{a}=j}}^{k}(x+z_{a}).
\end{gather}
According to Lemma \ref{T is local frame}, the functions $y_{0}(x)\lc y_{n-1}(x)$ are polynomials.
\begin{lem}
For each $i=1\lc n-1$, there exists a polynomial $\tilde{y}_{i}$ such that
\begin{equation}\label{y i is fertile}
\mathcal{W}{\rm r}\left(y_{i}(x),\frac{\alpha_{i}^{x}}{\alpha_{i+1}^{x}}\tilde{y}_{i}(x)\right)=\frac{\alpha_{i}^{x}}{\alpha_{i+1}^{x}}\tilde{T}_{i}(x)y_{i-1}(x)y_{i+1}(x+1).
\end{equation}
\end{lem}
\begin{proof}
Set
\[\tilde{y}_{i}(x)=\alpha_{i+1}\frac{\mathcal{W}{\rm r}\big(\alpha_{n}^{x}r_{n}(x)\lc \alpha_{i+2}^{x}r_{i+2}(x),\alpha_{i}^{x}r_{i}(x)\big)}{\alpha_{n}^{x}\cdots\alpha_{i+2}^{x}\alpha_{i}^{x}\prod_{j=i+1}^{n}(T_{j}(x))},\qquad i=1\lc n-1.\]
By Lemma \ref{y i is fertile}, $\tilde{y}_{1}(x)\lc \tilde{y}_{n-1}(x)$ are polynomials, and \eqref{y i is fertile} follows from discrete Wronskian identities \eqref{Wr1} and \eqref{Wr2}.
\end{proof}

Denote $u_{i}(x)=y_{i}(x+i/2)$, $i=0\lc n-1$. Then equations \eqref{y i is fertile} become
\begin{equation}\label{u i is fertile}
\mathcal{W}{\rm r}\left(u_{i}(x),\frac{\alpha_{i}^{x}}{\alpha_{i+1}^{x}}\tilde{y}_{i}(x+i/2)\right)=\frac{\alpha_{i}^{x}}{\alpha_{i+1}^{x}}\tilde{T}_{i}(x+i/2)u_{i-1}(x+1/2)u_{i+1}(x+1/2),
\end{equation}
where $i=1\lc n-1$.

It is easy to see that for each $i=0\lc n-1$, $\deg u_{i} = \deg y_{i} = \bar{m}_{i}$, where $\bar{m}_{0}\lc \bar{m}_{n-1}$ are given by formula \eqref{m bar}. In particular, $\deg u_{0} = \deg y_{0}=0$. One can normalize polynomials $r_{1}(x)\lc r_{n}(x)$ so that the polynomials $y_{0}(x)\lc y_{n-1}(x)$ (and hence $u_{0}(x)\lc u_{n-1}(x)$) are monic. For each $i=1\lc n-1$, write
\[u_{i}(x)=\prod_{a=1}^{\bar{m}_{i}}\big(x-s^{(i)}_{a}\big).\]

We will call the space $W$ XXX-admissible if for each $i=1\lc n-1$, the polynomial $u_{i}(x)$ has only simple roots, different from the roots of the polynomials $u_{i-1}(x+1/2)$, $u_{i+1}(x+1/2)$, $\tilde{T}_{i}(x+i/2)$, and $u_{i}(x+1)$.

The following theorem is a part of Theorem~7.4 in~\cite{MV}:
\begin{thm}
Let W be XXX-admissible, then relations \eqref{u i is fertile} imply
\begin{equation}\label{BAEs}
\frac{\alpha_{i+1}}{\alpha_{i}}=\prod_{\substack{a=1\\l_{a}=i}}^{k}\frac{s_{b}^{(l_{a})}-\check{z}_{a}+1/2}{s_{b}^{(l_{a})}-\check{z}_{a}-1/2}\prod_{|j-i|=1}\prod_{a=1}^{\bar{m}_{j}}\frac{s_{b}^{(i)}-s_{a}^{(j)}+1/2}{s_{b}^{(i)}-s_{a}^{(j)}-1/2}\prod_{\substack{a=1\\a\neq b}}^{\bar{m}_{i}}\frac{s_{b}^{(i)}-s_{a}^{(i)}-1}{s_{b}^{(i)}-s_{a}^{(i)}+1},
\end{equation}
where $i=1\lc n-1$, $b=1\lc \bar{m}_{i}$, and $\check{z}_{a}=-z_{a}-l_{a}/2+1/2$ for each $a=1\lc k$.
\end{thm}
A tuple of polynomials $u_{1}(x)\lc u_{n-1}(x)$ such that relations \eqref{u i is fertile} hold for some polynomials $\tilde{y}_{1}(x)\lc \tilde{y}_{n-1}(x)$ is called a \textit{fertile} tuple in \cite{MV}.

Let us call the equations \eqref{XXXBAE} the XXX Bethe ansatz equations associated to $\bar{z}=(z_{1}\lc z_{k})$. For each $i=1\lc n-1$, $a=1\lc \bar{m}_{i}$, set $t^{(i)}_{a}=s^{(i)}_{a}-i/2$. Then, using \eqref{BAEs}, it is easy to check that $\boldsymbol{t}=\big(t^{(1)}_{1}\lc t^{(n-1)}_{\bar{m}_{n-1}}\big)$ is an XXX-admissible solution of the XXX Bethe ansatz equations associated to $-\bar{z}-\bar{l}+\bar{1}=(-z_{1}-l_{1}+1, -z_{2}-l_{2}+1\lc -z_{k}-l_{k}+1)$. Therefore, to each XXX-admissible space of quasi-exponentials $W$ with the difference data $\big(\bar{\alpha},\bar{\mu}';-\bar{z},\bar{\lambda}'\big)$, corresponds a vector $v_{W}=\psi\big(\boldsymbol{t},-\bar{z}-\bar{l}+\bar{1}\big)\in\mathfrak{P}_{kn}[\boldsymbol{l},\boldsymbol{m}]$, which, provided that $v_{W}\neq 0$, is an eigenvector of the trigonometric dynamical Hamiltonians $G^{\langle n,k\rangle}_{1}\big({-}\bar{z}-\bar{l}+\bar{1},\bar{\alpha}\big)\lc G^{\langle n,k\rangle}_{n}\big({-}\bar{z}-\bar{l}+\bar{1},\bar{\alpha}\big)$, and the associated eigenvalues are given by the formula~\eqref{g}, where we should substitute $z_{a}\rightarrow -z_{a}-l_{a}+1$, $a=1\lc k$. We will call $v_{W}$ the Bethe vector corresponding to~$W$.

We are now going to relate the eigenvalues of the trigonometric dynamical Hamiltonians associated with the eigenvector $v_{W}$ and the coefficients of the fundamental difference operator~$S_{W}$ of the space~$W$.

Let $y_{0}(x)\lc y_{n-1}(x)$, $T_{1}(x)\lc T_{n}(x)$ be the polynomials given by \eqref{difference y} and \eqref{difference T}, respectively. Put $y_{n}(x)=1$. Define
\begin{equation}\label{difference Y}
Y_{i}=\alpha_{i}\frac{T_{i}(x+1)y_{i-1}(x+1)y_{i}(x)}{T_{i}(x)y_{i-1}(x)y_{i}(x+1)},\qquad i=1\lc n.
\end{equation}
Comparing formulas \eqref{S2}, \eqref{g i}, and \eqref{difference Y}, we get
\[
S_{W}=(T-Y_{1}(x))(T-Y_{2}(x))\cdots(T-Y_{n}(x)).
\]

For each $i=1\lc n-1$, write
\[y_{i}(x)=\prod_{a=1}^{\bar{m}_{i}}\big(x-\tilde{t}^{(i)}_{a}\big).\]
Then we have
\[ Y_{i}(x)=\alpha_{i}\prod_{\substack{a=1\\l_{a}\geq i}}^{k}\frac{x+z_{a}+l_{a}-i+1}{x+z_{a}+l_{a}-i}\prod_{a=1}^{\bar{m}_{i-1}}\frac{x-\tilde{t}^{(i-1)}_{a}+1}{x-\tilde{t}^{(i-1)}_{a}}
\prod_{a=1}^{\bar{m}_{i}}\frac{x-\tilde{t}^{(i)}_{a}-1}{x-\tilde{t}^{(i)}_{a}},\qquad i=1\lc n.\]

Since $y_{i}(x)=u_{i}(x-i/2)$, we have $s^{(i)}_{a}=\tilde{t}^{(i)}_{a}-i/2$, $i=1\lc n-1$, $a=1\lc\bar{m}_{i}$. Therefore, for the solution $\boldsymbol{t}=\big(t^{(1)}_{1}\lc t^{(n-1)}_{\bar{m}_{n-1}}\big)$ of the XXX Bethe ansatz equations corresponding to the space $W$, we get $t^{(i)}_{a}=s^{(i)}_{a}-i/2=\tilde{t}^{(i)}_{a}-i$. Denote this solution as $\tilde{\boldsymbol{t}}-\boldsymbol{i}$.

Comparing the last formula for $Y_{i}(x)$ with the formula \eqref{X i} for $\mathcal{X}_{i}(x, \boldsymbol{t}, \bar{z}, \bar{\alpha})$, we have
\begin{equation}\label{X Y}
\mathcal{X}_{i}\big(x, \tilde{\boldsymbol{t}}-\boldsymbol{i}, -\bar{z}-\bar{l}+\bar{1}, \bar{\alpha}\big)=Y_{i}(x+i-1).
\end{equation}

Let $\check{E}_{1}(x)\lc\check{E}_{n}(x)$ be the coefficients of the fundamental difference operator $S_{W}$ of the space~$W$:
\[ S_{W}=T^{n}+\sum_{i=1}^{n}\check{E}_{i}(x)T^{n-i}.\]
For each $i=1\lc n$, we have
\begin{equation}\label{check E}
\check{E}_{i}(x)=\sum_{1\leq i_{1}<\dots <i_{j}\leq n}Y_{i_{1}}(x+i_{1}-1)Y_{i_{2}}(x+i_{2}-2)\cdots Y_{i_{j}}(x+i_{j}-j).
\end{equation}

Comparing formulas \eqref{tilde E}, \eqref{check E}, and \eqref{X Y}, we get $\tilde{E}_{i}\big(x,\tilde{\boldsymbol{t}}-\boldsymbol{i}, -\bar{z}-\bar{l}+\bar{1}, \bar{\alpha}\big)=\check{E}_{i}(x)$. This, together with Theorem \ref{Bethe eigenvectors for G}, proves the following:
\begin{prop}\label{eigenv of trig D}
Let $W$ be an XXX-admissible space of quasi-exponentials $W$ with the difference data $\big(\bar{\alpha},\bar{\mu}';-\bar{z},\bar{\lambda}'\big)$. Let~$v_{W}$ be the Bethe vector corresponding to~$W$.
Write the fundamental difference operator $S_{W}$ of the space $W$ in the following form:
\[ S_{W}=\sum_{a=0}^{\infty}x^{-a}E_{a}(T),\]
where $E_{1}(T),E_{2}(T),\dots$ are some polynomials in $T$. Then we have
\[ G^{\langle n,k\rangle}_{i}\big({-}\bar{z}-\boldsymbol{l}+\bar{1},\bar{\alpha}\big)v_{W}=g^{W}_{i}v_{W},\]
where
\begin{equation*}
g^{W}_{i}=-\frac{1}{\alpha_{i}}\Res_{u=\alpha_{i}}\frac{E_{2}(u)}{\prod_{j=1}^{n}(u-\alpha_{i})}+\sum_{\substack{j=1\\j\neq i}}^{n}\frac{\alpha_{j}m_{i}m_{j}}{\alpha_{i}-\alpha_{j}}-\frac{m_{i}^{2}}{2}.
\end{equation*}
\end{prop}

\subsection[Quotient difference operator and duality for trigonometric Gaudin and dynamical Hamiltonians]{Quotient difference operator and duality for trigonometric Gaudin\\ and dynamical Hamiltonians}\label{qdo and duality}

Fix a pair $(\boldsymbol{l}, \boldsymbol{m})$ like in the previous section. Let the data $\big(\bar{z},\bar{\lambda};\bar{\alpha},\bar{\mu}\big)$ be like in Section \ref{qpol and BA}. Let~$V$ be a Gaudin admissible space of quasi-polynomials with the data $\big(\bar{z},\bar{\lambda};\bar{\alpha},\bar{\mu}\big)$.

Recall the maps $\mathfrak{T}_{1}$ and $\mathfrak{T}_{3}$, see formulas \eqref{T_1} and \eqref{T_3}, respectively. Set $W=\mathfrak{T}_{1}(\mathfrak{T}_{3}(V))$. Then $W$ is a space of quasi-exponentials with the difference data $\big(\bar{\alpha},\bar{\mu}';-\bar{z},\bar{\lambda}'\big)$. In this section, we will relate the map $V\mapsto W=\mathfrak{T}_{1}(\mathfrak{T}_{3}(V))$ with the $(\gl_{k},\gl_{n})$-duality of the trigonometric Gaudin and dynamical Hamiltonians.

We will need the following lemma.
\begin{lem}\label{simple spectrum}
For generic $\bar{\alpha},\bar{z}$, the common eigenspaces
of the trigonometric dynamical Hamiltonians $G^{\langle n,k\rangle}_{1}(\bar{z},\bar{\alpha})\lc G^{\langle n,k\rangle}_{n}(\bar{z},\bar{\alpha})$
in $\mathfrak{P}_{kn}$ are one-dimensional.
\end{lem}
\begin{proof}
For every monomial $p\in\mathfrak{P}_{kn}$, we have
$\big(e^{\langle n\rangle}_{ii}\big)_{(a)}p=m^{a}_{i}(p)p$ for some $m^{a}_{i}(p)\in\Z$.
Moreover, if $p\neq
p'$, there exist $i$, $a$ such that $m^{a}_{i}(p)\neq m^{a}_{i}(p')$. Thus, if $z_{1}\lc z_{k}$ are linearly independent
over $\Z$, the common eigenspaces of the operators $K_{i}=\sum_{a=1}^{k} z_{a}\big(e^{\langle n\rangle}_{ii}\big)_{(a)}$,
\mbox{$i=1\lc n$}, in $\mathfrak{P}_{kn}$ are one-dimensional.
Therefore, the common eigenspaces of the
operators $G^{\langle n,k\rangle}_{1}(\bar{z},\bar{\alpha})\lc G^{\langle n,k\rangle}_{n}(\bar{z},\bar{\alpha})$
in $\mathfrak{P}_{kn}$ are one-dimensional
provided that $z_{1}\lc z_{k}$ are sufficiently large positive numbers linearly independent
over $\Z$.
Hence, the common eigenspaces for generic $\bar{\alpha}$, $\bar{z}$ are one-dimensional.
\end{proof}

Let $v_{V}\in\mathfrak{P}_{kn}[\boldsymbol{l},\boldsymbol{m}]$ be the Bethe vector corresponding to $V$, see Section \ref{qpol and BA}. Assume that $v_{V}\neq 0$. Then the vector $v_{V}$ is an eigenvector of the trigonometric Gaudin Hamiltonians $H^{\langle k,n\rangle}_{1}(\bar{\alpha},\bar{z}+\boldsymbol{l})\lc H^{\langle k,n\rangle}_{n}(\bar{\alpha},\bar{z}+\boldsymbol{l})$. Denote the associated eigenvalues as $h^{V}_{1}\lc h^{V}_{k}$, respectively.

Assume that the space $W=\mathfrak{T}_{1}(\mathfrak{T}_{3}(V))$ is XXX-admissible. Let $v_{W}\in\mathfrak{P}_{kn}[\boldsymbol{l},\boldsymbol{m}]$ be the Bethe vector corresponding to $W$, see Section \ref{qexp to BAE}. Assume that $v_{W}\neq 0$. Then the vector~$v_{W}$ is an eigenvector of the trigonometric dynamical Hamiltonians~$G^{\langle n,k\rangle}_{1}\big({-}\bar{z}-\boldsymbol{l}+\bar{1},\bar{\alpha}\big)\lc G^{\langle n,k\rangle}_{n}\big({-}\bar{z}-\boldsymbol{l}+\bar{1},\bar{\alpha}\big)$. Denote the associated eigenvalues as $g^{W}_{1}\lc g^{W}_{n}$, respectively.

\begin{thm}\label{discrete main 2}
The following holds:
\begin{equation}\label{discrete main 2 formula}
h^{V}_{i}=-g^{W}_{i},\qquad i=1\lc n.
\end{equation}
\end{thm}

Before proving the theorem, let us discuss how it explains the relation between the map $V\mapsto W=\mathfrak{T}_{1}(\mathfrak{T}_{3}(V))$ and the $(\gl_{k},\gl_{n})$-duality. By Proposition~\ref{Gaudin Dyn duality}, for each $i=1\lc n$, we have
\begin{equation}\label{Gaudin Dyn duality 2}
G^{\langle n,k\rangle}_{n}\big({-}\bar{z}-\boldsymbol{l}+\bar{1},\bar{\alpha}\big)v_{V}=-H^{\langle k,n\rangle}_{i}(\bar{\alpha},\bar{z}+\boldsymbol{l})v_{V}=-h^{V}_{i}v_{V}.
\end{equation}

Therefore, starting with the space $V$ and the corresponding vector $v_{V}$, we have two different ways to obtain a common eigenvector of the trigonometric dynamical Hamiltonians. First, by the $(\gl_{k},\gl_{n})$-duality, $v_{V}$ is itself a common eigenvector of the dynamical Hamiltonians, see formula~\eqref{Gaudin Dyn duality 2}. Second, the map $V\mapsto W=\mathfrak{T}_{1}(\mathfrak{T}_{3}(V))$ gives the vector~$v_{W}$. Theorem~\ref{discrete main 2} and Lemma~\ref{simple spectrum} assure that for generic $\bar{z},\bar{\alpha}$, these two eigenvectors are the same up to a constant multiple.

Indeed, comparing formulas \eqref{discrete main 2 formula} and~\eqref{Gaudin Dyn duality 2}, we have
\begin{equation*}
G^{\langle n,k\rangle}_{n}\big({-}\bar{z}-\boldsymbol{l}+\bar{1},\bar{\alpha}\big)v_{V}=g^{W}_{i}v_{V},
\end{equation*}
which means that the vectors $v_{V}$ and $v_{W}$ belong to the same eigenspace. Then Lemma \ref{simple spectrum} implies that $v_{W}$ is proportional to $v_{V}$.

\begin{proof}[Proof of Theorem \ref{discrete main 2}]
Denote $U=\mathfrak{T}_{3}(V)\in\mathcal{E}\big(\bar{\alpha},\bar{\mu};\bar{z}+\bar{\lambda'}_{1},\bar{\lambda}\big)$.
By Lemma \ref{p_{W}}, the fundamental difference operator $S_{U}=T^{M}+\sum_{i=1}^{M}b_{i}(x)T^{M-i}$ of $U$ has rational coefficients $b_{1}(x)\lc b_{M}(x)$, which are regular at infinity. Therefore, there exist polynomials $B_{0}(u), B_{1}(u), B_{2}(u),\dots$ such that
\begin{equation}\label{5.6.1}
S_{U}=\sum_{a=0}^{\infty}x^{-a}B_{a}(T).
\end{equation}
Moreover, Lemma \ref{p_{W}} gives an explicit formula for the polynomial $B_{0}(x)$:
\begin{equation}\label{5.6.5}
B_{0}(u)=p_{\bar{\alpha},\bar{\mu}}(u)=\prod_{i=1}^{n}(u-\alpha_{i})^{m_{i}}.
\end{equation}

Consider the regularized fundamental difference operator $\bar{S}_{U}=q_{\bar{z},\bar{\lambda}}(x)S_{U}$ of $U$, where $q_{\bar{z},\bar{\lambda}}(x)\allowbreak =\prod_{a=1}^{k}(x-z_{a}-l_{a})$, see Section~\ref{BD}. Since $\deg q_{\bar{z},\bar{\lambda}}(x)= k$, the coefficients $\bar{b}_{1}(x)\lc\bar{b}_{M}(x)$ in the expansion $\bar{S}_{U}=T^{M}+\sum_{i=1}^{M}\bar{b}_{i}(x)T^{M-i}$ are polynomials in $x$ of degree at most $k$.

Define numbers $A_{ia}$, $i=1\lc M$, $a=1\lc k$ by $\bar{S}_{U}=\sum_{i=1}^{M}\sum_{a=1}^{k}A_{ia}x^{a}T^{i}$. Then we have
\begin{equation}\label{5.6.2}
S_{U}=\frac{1}{\prod_{a=1}^{k}(x-z_{a}-l_{a})}\sum_{i=1}^{M}\sum_{a=1}^{k}A_{ia}x^{a}T^{i}.
\end{equation}

Denote $\sum_{a=1}^{k}(z_{a}+l_{a})=Z$. Comparing formulas \eqref{5.6.1} and \eqref{5.6.2}, we get
\begin{gather}
B_{0}(u)=\sum_{i=1}^{M}A_{i,k}u^{i}, \qquad
B_{1}(u)=\sum_{i=1}^{M} (A_{i,k-1}+ZA_{i,k} )u^{i},\nonumber \\
B_{2}(u)=\sum_{i=1}^{M}\big(A_{i,k-2}+ZA_{i,k-1}+Z^{2}A_{i,k}\big)u^{i}.\label{5.6.4}
\end{gather}

Let $\bar{D}_{V}$ be the regularized fundamental differential operator of $V$. Since $U=\mathfrak{T}_{3}(V)$, by Theorem~\ref{bisp dual}, we have
\begin{equation}\label{5.6.3}
\bar{D}_{V}=\sum_{i=1}^{M}\sum_{a=1}^{k}A_{ia}x^{i}\left(x\frac{{\rm d}}{{\rm d}x}\right)^{a}.
\end{equation}

Let $D_{V}$ be the fundamental differential operator of $V$. We have $\bar{D}_{V}=p_{\bar{\alpha},\bar{\mu}}(x)\big(x^{k}D_{V}\big)$, where $p_{\bar{\alpha},\bar{\mu}}(x)=\prod_{i=1}^{n}(x-\alpha_{i})^{m_{i}}$, see Section~\ref{BD}. Write
\begin{equation*}
x^{k}D_{V}=\left(x\frac{{\rm d}}{{\rm d}x}\right)^{k}+\sum_{a=1}^{k}\beta_{a}(x)\left(x\frac{{\rm d}}{{\rm d}x}\right)^{k-a}.
\end{equation*}
Then formula \eqref{5.6.3} gives
\begin{equation}\label{5.6.6}
\beta_{a}=\frac{\sum_{i=1}^{M}A_{i,k-a}x^{i}}{\prod_{i=1}^{n}(x-\alpha_{i})^{m_{i}}},\qquad a=1\lc k.
\end{equation}
By Proposition \ref{prop5.4}, we have
\begin{equation}\label{5.6.7}
h^{V}_{i}=\frac{1}{\alpha_{i}}\Res_{x=\alpha_{i}}\left(\frac{1}{2}\beta_{1}^{2}(x)-\beta_{2}(x)\right)+\frac{m_{i}^{2}}{2}-m_{i}.
\end{equation}

Using formulas \eqref{5.6.4}, \eqref{5.6.5}, and \eqref{5.6.6}, one can check
\begin{equation*}
\Res_{x=\alpha_{i}}\left(\frac{1}{2}\beta_{1}^{2}(x)-\beta_{2}(x)\right)=\Res_{u=\alpha_{i}}\left(\frac{1}{2}\frac{B^{2}_{1}(u)}{B^{2}_{0}(u)}-\frac{B_{2}(u)}{B_{0}(u)}\right).
\end{equation*}
Therefore, formula \eqref{5.6.7} gives
\begin{equation}\label{5.6.11}
h^{V}_{i}=\frac{1}{\alpha_{i}}\Res_{u=\alpha_{i}}\left(\frac{1}{2}\frac{B^{2}_{1}(u)}{B^{2}_{0}(u)}-\frac{B_{2}(u)}{B_{0}(u)}\right)+\frac{m_{i}^{2}}{2}-m_{i}.
\end{equation}

Consider the space $W=\mathfrak{T}_{1}(U)\in\mathcal{E}\big(\bar{\alpha},\bar{\mu}';-\bar{z},\bar{\lambda}'\big)$.
We have
\begin{equation}\label{7.1.1}
\prod_{i=1}^{n}(T-\alpha_{i})^{m_{i}+1}=S_{W}^{\ddagger}S_{U},
\end{equation}
where the involutive automorphism $(\cdot)^{\ddagger}$ is defined in formula~\eqref{ddagger}.

The fundamental difference operator $S_{W}$ of $W$ can be written in the form
\begin{equation*}
S_{W}=\sum_{a=0}^{\infty}x^{-a}E_{a}(T).
\end{equation*}

Substituting this into formula \eqref{7.1.1}, we have
\[\prod_{i=1}^{n}(T-\alpha_{i})^{m_{i}+1}=\left(\sum_{a=0}^{\infty}E_{a}(T)(-x)^{-a}\right)\left(\sum_{a=0}^{\infty}x^{-a}B_{a}(T)\right).\]

Writing the right hand side of the last formula in the form $\sum_{a=0}^{\infty}x^{-a}P_{a}(T)$ with some polynomials $P_{0}(x),P_{1}(x),P_{2}(x),\dots$ and comparing it to the left hand side, we see that $P_{a}(u)=0$ for all $a\geq 1$, and
\begin{equation}\label{5.6.8}
E_{0}(u)B_{0}(u)=P_{0}(u)=\prod_{i=1}^{n}(u-\alpha_{i})^{m_{i}+1}.
\end{equation}
From $P_{1}(u)=0$, we get
\begin{equation}\label{5.6.9}
E_{0}(u)B_{1}(u)-E_{1}(u)B_{0}(u)=0.
\end{equation}
From $P_{2}(u)=0$, we get
\begin{equation}\label{5.6.10}
E_{2}(u)B_{0}(u)+E_{0}(u)B_{2}(u)+uE'_{1}(u)B_{0}(u)-uE'_{0}(u)B_{1}(u)-E_{1}(u)B_{1}(u)=0.
\end{equation}
In the last formula we used that for every polynomial $P(u)$, we have
\[
P(T)x^{-1}=x^{-1}P(T)-x^{-2}TP'(T)+\sum_{a\geq 3}x^{-a}\tilde{P}_{a}(T)
\]
for some polynomials $\tilde{P}_{3}(u),\tilde{P}_{4}(u),\dots $.

Using relations \eqref{5.6.9} and \eqref{5.6.10}, one can check
\[\frac{1}{2}\frac{B^{2}_{1}(u)}{B^{2}_{0}(u)}-\frac{B_{2}(u)}{B_{0}(u)}
=-\left(\frac{1}{2}\frac{E^{2}_{1}(u)}{E^{2}_{0}(u)}-\frac{E_{2}(u)}{E_{0}(u)}\right)+u\left(\frac{E_{1}(u)}{E_{0}(u)}\right)'.\]

Therefore, formula \eqref{5.6.11} gives
\begin{gather}
h^{V}_{i}= -\frac{1}{\alpha_{i}}\Res_{u=\alpha_{i}}\left(\frac{1}{2}\frac{E^{2}_{1}(u)}{E^{2}_{0}(u)}-\frac{E_{2}(u)}{E_{0}(u)}\right)+ \Res_{u=\alpha_{i}}\left(u\left(\frac{E_{1}(u)}{E_{0}(u)}\right)'\right)+\frac{m_{i}^{2}}{2}-m_{i}.\label{5.6.17}
\end{gather}

Let $g^{W}_{1}\lc g^{W}_{n}$ be the eigenvalues of the trigonometric dynamical Hamiltonians $G^{\langle n,k\rangle}_{1}\big({-}\bar{z}-\boldsymbol{l}+\bar{1},\bar{\alpha}\big)\lc G^{\langle n,k\rangle}_{n}\big({-}\bar{z}-\boldsymbol{l}+\bar{1},\bar{\alpha}\big)$, respectively, associated with the Bethe vector~$v_{W}$. By Proposition~\ref{eigenv of trig D}, we have
\begin{equation}\label{5.6.14}
g^{W}_{i}=-\frac{1}{\alpha_{i}}\Res_{u=\alpha_{i}}\frac{E_{2}(u)}{\prod_{j=1}^{n}(u-\alpha_{i})}+\sum_{\substack{j=1\\j\neq i}}^{n}\frac{\alpha_{j}m_{i}m_{j}}{\alpha_{i}-\alpha_{j}}-\frac{m_{i}^{2}}{2}.
\end{equation}

We will use again \cite[Proposition B.1]{MTV6}, which gives the following explicit formula for the quotient $E_{1}(u)/\prod_{i=1}^{n}(u-\alpha_{i})$:
\begin{equation}\label{5.6.12}
\frac{E_{1}(u)}{\prod_{i=1}^{n}(u-\alpha_{i})}=\sum_{j=1}^{n}\frac{\alpha_{j}m_{j}}{\alpha_{j}-u}.
\end{equation}

From formulas \eqref{5.6.5} and \eqref{5.6.8}, we get
\begin{equation}\label{5.6.13}
E_{0}(u)=\prod_{i=1}^{n}(u-\alpha_{i}).
\end{equation}
Using \eqref{5.6.12} and \eqref{5.6.13}, we can rewrite \eqref{5.6.14} in the following way:
\begin{equation}\label{5.6.15}
g^{W}_{i}=\frac{1}{\alpha_{i}}\Res_{u=\alpha_{i}}\left(\frac{1}{2}\frac{E^{2}_{1}(u)}{E^{2}_{0}(u)}-\frac{E_{2}(u)}{E_{0}(u)}\right)-\frac{m_{i}^{2}}{2}.
\end{equation}

Using \eqref{5.6.12} and \eqref{5.6.13} again, we compute
\begin{equation}\label{5.6.16}
\frac{1}{\alpha_{i}}\Res_{u=\alpha_{i}}\left(u\left(\frac{E_{1}(u)}{E_{0}(u)}\right)'\right)=m_{i}.
\end{equation}

Comparing formulas \eqref{5.6.17}, \eqref{5.6.15}, and \eqref{5.6.16}, we get
\eqref{discrete main 2 formula}.
Theorem \ref{discrete main 2} is proved.
\end{proof}
\subsection{Non-reduced data}\label{non-reduced data}
In the previous section, we related the quotient difference operator and the $(\gl_{k},\gl_{n})$-duality of the trigonometric Gaudin and dynamical Hamiltonians acting on the space $\mathfrak{P}_{kn}[\boldsymbol{l},\boldsymbol{m}]$, where $\boldsymbol{l}=(l_{1}\lc l_{k})$ and $\boldsymbol{m}=(m_{1}\lc m_{n})$ are such that $l_{a}\neq 0$, $a=1\lc k$ and $m_{i}\neq 0$, \mbox{$i=1\lc n$}. In this section, we are going to extend this result to all nontrivial subspaces $\mathfrak{P}_{kn}[\boldsymbol{l},\boldsymbol{m}]$, that is, we are going to include the cases when some $l_{a}$, $m_{i}$ are zero.

Fix $\boldsymbol{l}=(l_{1}\lc l_{k})\in \Z^{k}_{\geq 0}$. For each $a=1\lc k$, let $q_{a}(x)$ be a polynomial of degree $l_{a}$ such that $q_{a}(0)\neq 0$. Fix complex numbers $z_{1}\lc z_{k}$ such that $z_{a}-z_{b}\notin\Z$ if $a\neq b$. Denote by $V$ the space spanned by the functions $x^{z_{a}}q_{a}(x)$, $a=1\lc k$.

Define
\[ V^{\red}=\prod_{\substack{a=1\\l_{a}=0}}^{k}\left(x\frac{{\rm d}}{{\rm d}x}-z_{a}\right)V. \]
Denote $k'=\dim V^{\red}$. Fix $\alpha\in \C^{*}$. Let $(e_{1}>\dots > e_{k})$ be the sequence of exponents of $V$ at $\alpha$, and let $\big(e^{\red}_{1}> \dots >e^{\red}_{k'}\big)$ be the sequence of exponents of~$V^{\red}$ at~$\alpha$.

\begin{lem}\label{non-reduced lemma}
Define a partition $\mu=(\mu_{1},\mu_{2},\dots)$ by $e^{\red}_{a}=k'+\mu_{a}-a$, $a=1\lc k'$, $\mu_{k'+1}=0$. Then $e_{a}=k+\mu_{a}-a$, $a=1\lc k$.

Conversely, if a partition $\mu$ is such that $e_{a}=k+\mu_{a}-a$, $a=1\lc k$, then $\mu_{k'+1}=0$ and $e^{\red}_{a}=k'+\mu_{a}-a$, $a=1\lc k'$.
\end{lem}
\begin{proof}
It is enough to prove the lemma for the case when $l_{1}=0$, and $l_{2}\lc l_{k}$ are not zero.
Let $D_{V}$ and $D_{V^{\red}}$ be the monic linear differential operators of order $k$ and $k-1$, respectively, annihilating $V$ and $V^{\red}$, respectively. Then
\begin{equation}\label{non-reduced 1}
x^{k}D_{V}=x^{k-1}D_{V^{\red}}\left(x\frac{{\rm d}}{{\rm d}x}-z_{1}\right).
\end{equation}

Define functions $b_{1}(x)\lc b_{k}(x)$, $b_{1}^{\red}(x)\lc b_{k-1}^{\red}(x)$ by
\begin{gather*}
x^{k}D_{V}=\sum_{a=0}^{k}\frac{b_{a}(x)}{(x-\alpha)^{a}}\left(x\frac{{\rm d}}{{\rm d}x}\right)^{k-a},
\\
x^{k-1}D_{V^{\red}}=\sum_{a=0}^{k-1}\frac{b^{\red}_{a}(x)}{(x-\alpha)^{a}}\left(x\frac{{\rm d}}{{\rm d}x}\right)^{k-1-a}.
\end{gather*}

Using formulas \eqref{a_{i}(x)}, \eqref{Wr_a}, and \eqref{Wr}, one can check that $b_{1}(x)\lc b_{k}(x)$, $b_{1}^{\red}(x)\lc b_{k-1}^{\red}(x)$ are regular at $\alpha$. Define polynomials $I(r)$ and $I^{\red}(r)$ by
\begin{gather*}
I(r)=\sum_{a=1}^{k}b_{a}(\alpha)\alpha^{k-a}r(r-1)(r-2)\cdots (r-k+a+1),
\\
I^{\red}(r)=\sum_{a=1}^{k-1}b^{\red}_{a}(\alpha)\alpha^{k-1-a}r(r-1)(r-2)\cdots (r-k+a+2).
\end{gather*}

Notice that $\{e_{1}\lc e_{k}\}$ is the set of roots of the polynomial $I(r)$. Indeed, substituting a~series $\sum_{i=0}^{\infty}A_{i}(x-\alpha)^{i+r}$ into the differential equation $D_{V}f=0$, and looking at the coefficient for the lowest power of $(x-\alpha)$, we get $I(r)=0$. Similarly, $\big\{e^{\red}_{1}\lc e^{\red}_{k'}\big\}$ is the set of roots of the polynomial $I^{\red}(r)$. The polynomials $I(r)$ and $I^{\red}(r)$ are called the indicial polynomials of the differential equations $D_{V}f=0$ and $D_{V^{\red}}f=0$, respectively.

Using formula \eqref{non-reduced 1}, we obtain the following relations:
\begin{equation}\label{non-reduced 2}
b_{a}(x)=b_{a}^{\red}(x)-z_{1}(x-\alpha)b_{a-1}^{\red}(x),\qquad a=1\lc k,
\end{equation}
where we assume that $b^{\red}_{k}(x)=0$. Relations~\eqref{non-reduced 2} imply $b_{a}(\alpha)=b_{a}^{\red}(\alpha)$, $a=1\lc k$. Since~$D_{V}$ and $D_{V^{\red}}$ are monic, we also have $b_{0}(x)=b^{\red}_{0}(x)=1$. Therefore, $I(r)=rI^{\red}(r-1)$, which implies the lemma.
\end{proof}

Let $\{\alpha_{1}\lc\alpha_{n}\}$ be a set including all non-zero singular points of $V$. Assume that $\alpha_{i}\neq\alpha_{j}$ if $i\neq j$, and $\alpha_{i}\neq 0$ for all $i=1\lc n$. Suppose that for each $i=1\lc n$, the sequence of exponents of $V$ at $\alpha_{i}$ is given by
\[ (k,k-1\lc k-m_{i}+1,k-m_{i}-1,k-m_{i}-2\lc 1,0)\]
for some $m_{i}\in \Z$, $0\leq m_{i}\leq k$.

Define a sequence of partitions $\bar{\lambda}=\big(\lambda^{(1)}\lc \lambda^{(k)}\big)$ by $\lambda^{(a)}=(l_{a},0,0,\dots)$, $a=1\lc k$. Define a sequence of partitions $\bar{\mu}=\big(\mu^{(1)}\lc \mu^{(n)}\big)$ by $\mu^{(i)}=(1,1\lc 1,0,0,\dots)$ with $m_{i}$ ones, $i=1\lc n$. Define sequences $\bar{\lambda}^{\red}$, $\bar{\mu}^{\red}$, $\bar{z}^{\red}$, and $\bar{\alpha}^{\red}$ by removing all zero partitions from the sequences $\bar{\lambda}$, $\bar{\mu}$, and removing corresponding numbers from the sequences $\bar{z}=(z_{1}\lc z_{n})$, $\bar{\alpha}=(\alpha_{1}\lc \alpha_{n})$. We will call the data $\big(\bar{z},\bar{\lambda};\bar{\alpha},\bar{\mu}\big)$ reduced if $\big(\bar{z},\bar{\lambda};\bar{\alpha},\bar{\mu}\big)=\big(\bar{z}^{\red},\bar{\lambda}^{\red};\bar{\alpha}^{\red},\bar{\mu}^{\red}\big)$, and non-reduced otherwise.
\begin{prop}\label{non-reduced prop 1}
$V^{\red}$ is a space of quasi-polynomials with the data $\big(\bar{z}^{\red},\bar{\lambda}^{\red};\bar{\alpha}^{\red},\bar{\mu}^{\red}\big)$.
\end{prop}
\begin{proof}
Recall that $V$ is spanned by the functions $x^{z_{a}}q_{a}(x)$, $a=1\lc k$, where $q_{1}(x)\lc q_{k}(x)$ are polynomials such that $\deg q_{a}=l_{a}$, and $q_{a}(0)\neq 0$, $a=1\lc k$. Then the space $V^{\red}$ is spanned by the functions $x^{z_{a}}\tilde{q}_{a}(x)$, $a=1\lc k$, where
\begin{equation}\label{non-reduced 3}
\tilde{q}_{b}(x)=\prod_{\substack{a=1\\l_{a}=0}}^{k}\left(x\frac{{\rm d}}{{\rm d}x}+z_{b}-z_{a}\right)q_{b}(x).
\end{equation}
If $l_{b}\neq 0$, then for each $a$ in the product on the left hand side of formula \eqref{non-reduced 3}, we have $z_{b}-z_{a}\notin\Z$, which yields $\deg \tilde{q}_{a}(x)=\deg q_{a}(x)$, $a=1\lc k$. If $l_{b}=0$, then formula \eqref{non-reduced 3} implies $\tilde{q}_{b}(x)=0$. This shows that the space $V^{\red}$ has a basis
\[\big\{x^{z_{a}}\tilde{q}_{a}(x)\,|\,z_{a}\text{ is present in }\bar{z}^{\red}\big\},\]
and the degrees of the polynomials $\tilde{q}_{a}(x)$ appearing in this basis correspond to the sequence~$\bar{\lambda}^{\red}$.

Notice that $\bar{\alpha}^{\red}$ is the set of all singular points of~$V$, and the sequences of exponents of~$V$ at these points correspond to the sequence $\bar{\mu}^{\red}$. Therefore, the proposition follows from Lem\-ma~\ref{non-reduced lemma}.
\end{proof}

Recall the maps $\mathfrak{T}_{1}$ and $\mathfrak{T}_{3}$, see 
\eqref{T_1} and~\eqref{T_3}, respectively. Set $W^{\red}=\mathfrak{T}_{1}\big(\mathfrak{T}_{3}\big(V^{\red}\big)\big)$. Then $W^{\red}$ is a space of quasi-exponentials with the difference data $\big(\bar{\alpha}^{\red},\big(\bar{\mu}^{\red}\big)';\allowbreak -\bar{z}^{\red},\big(\bar{\lambda}^{\red}\big)'\big)$. We are going to construct a space $W$ such that
\[W^{\red}=\prod_{\substack{i=1\\m_{i}=0}}^{n}(T-\alpha_{i})W.\] For this we will need the following lemma:
\begin{lem}\label{non-reduced lemma 1}
Fix $\alpha, \beta\in\C^{*}$, and a polynomial $p(x)$. Assume that $\alpha\neq\beta$. Then there exists a~unique polynomial $\tilde{p}(x)$ such that $\deg \tilde{p}(x)=\deg p(x)$, and
\begin{equation}\label{non-reduced 4}
(T-\beta)\alpha^{x}\tilde{p}(x)=\alpha^{x}p(x).
\end{equation}
\end{lem}
\begin{proof}
Relation \eqref{non-reduced 4} is the same as relation
\begin{equation}\label{non-reduced 5}
\alpha\tilde{p}(x+1)-\beta \tilde{p}(x)=p(x).
\end{equation}
Let $a_{0}\lc a_{m}$ be the coefficients of $p(x)$: $p(x)=a_{m}x^{m}+a_{m-1}x^{m-1}+\dots +a_{1}x+a_{0}$.
Substituting a polynomial $\tilde{p}(x)=\tilde{a}_{m}x^{m}+\tilde{a}_{m-1}x^{m-1}+\dots +\tilde{a}_{1}x+\tilde{a}_{0}$ into equation \eqref{non-reduced 5} and comparing coefficients for powers of $x$, we get
\[
\tilde{a}_{m-i}(\alpha-\beta)=a_{m-i}-\alpha\sum_{j=0}^{i-1}\binom{m-j}{m-i}\tilde{a}_{m-j}, \qquad i=0\lc m,
\]
which is a recursion that allows to find the numbers $\tilde{a}_{1}\lc \tilde{a}_{n}$ uniquely.
\end{proof}

For any $\beta\in\C^{*}$, define a linear operator $(T-\beta)^{-1}$ on the space spanned by all functions of the form $\alpha^{x}p(x)$, where $\alpha\in \C^{*}$, $\alpha\neq\beta$, and $p(x)$ is a polynomial, by the formula{\samepage
\[(T-\beta)^{-1}\alpha^{x}p(x)=\alpha^{x}\tilde{p}(x),\]
where $\tilde{p}(x)$ is the polynomial from Lemma \ref{non-reduced lemma 1}.}

Let $1\leq i_{1}<i_{2}< \dots <i_{l}\leq n$ be such that $m_{i}=0$ if $i=i_{s}$ for some $s=1\lc l$, and $m_{i}\neq 0$ otherwise. Denote by $W$ the space spanned by the functions
\[(T-\alpha_{i_{1}})^{-1}(T-\alpha_{i_{2}})^{-1}\dots(T-\alpha_{i_{l}})^{-1}f,\qquad f\in W^{\red}, \qquad\text{and}\qquad \alpha_{i_{1}}^{x}\lc \alpha_{i_{l}}^{x}.\]

Let $S_{W}$ be the fundamental difference operator of $W$. Let $S_{W^{\red}}$ be the fundamental difference operator of $W^{\red}$. Then we have
\begin{equation}\label{non-reduced 7}
S_{W}=S_{W^{\red}}\prod_{\substack{i=1\\m_{i}=0}}^{n}(T-\alpha_{i}).
\end{equation}
Together with Lemma \ref{unique kernel}, this shows that the order of $\alpha_{i_{1}}\lc\alpha_{i_{l}}$ in the definition of $W$ does not matter.

Recall that $W^{\red}$ is a space of quasi-exponentials with the difference data $\big(\bar{\alpha}^{\red},\big(\bar{\mu}^{\red}\big)';\allowbreak -\bar{z}^{\red},\big(\bar{\lambda}^{\red}\big)'\big)$. Then the equality $\deg \tilde{p}(x) = \deg p(x)$ in Lemma~\ref{non-reduced lemma 1} implies that the space $W$ has a basis of the form
\[
\big\{\alpha_{i}^{x}r_{i}(x),\,i=1\lc n \big\},
\]
where $r_{1}(x)\lc r_{n}(x)$ are polynomials such that $\deg r_{i}(x) = m_{i}$, $i=1\lc n$.

Fix $z\in \C$. Let $(\tilde{e}_{1}>\dots > \tilde{e}_{n})$ be the sequence of discrete exponents of~$W$ at~$z$. Denote $n'=n-l=\dim W^{\red}$. Let $\big(\tilde{e}^{\red}_{1}>\dots >\tilde{e}^{\red}_{n'}\big)$ be the sequence of discrete exponents of~$W^{\red}$ at~$z$.
\begin{lem}\label{non-reduced lemma 2}
Define a partition $\lambda=(\lambda_{1},\lambda_{2},\dots)$ by $\tilde{e}^{\red}_{i}=n'+\lambda_{i}-i$, $i=1\lc n'$, $\lambda_{n'+1}=0$. Then $\tilde{e}_{i}=n+\lambda_{i}-i$, $i=1\lc n$.

Conversely, if a partition $\lambda$ is such that $\tilde{e}_{i}=n+\lambda_{i}-i$, $i=1\lc n$, then $\lambda_{n'+1}=0$ and $\tilde{e}^{\red}_{i}=n'+\lambda_{i}-i$, $i=1\lc n'$.
\end{lem}
\begin{proof}
It is enough to prove the Lemma for the case $m_{1}=0$, and $m_{2}\lc m_{n}$ are not zero.

Let $f_{1}(x)\lc f_{n-1}(x)$ be a basis of $W^{\red}$ such that for each $i=1\lc n-1$, $T^{j}f_{i}(z)=0$, $j=0\lc \tilde{e}^{\red}_{i}-1$, and $T^{\tilde{e}^{\red}_{i}}f_{i}(z)\neq 0$. Set
\[
\tilde{f}_{i}(x)=(T-\alpha_{1})^{-1}f_{i}(x)-\alpha_{1}^{x-z}(T-\alpha_{i})^{-1}f_{i}(z),\qquad i=1\lc n.
\]
Then $\tilde{f}_{i}(x)\in W$, $(T-\alpha_{1})\tilde{f}_{i}(x)=f_{i}(x)$, and $\tilde{f}_{i}(z)=0$, $i=1\lc n-1$.

Since $T^{j}-\alpha_{1}^{j}=\big(\sum_{s=0}^{j-1}\alpha_{1}^{j-1-s}T^{s}\big)(T-\alpha_{1})$, we have
\[
T^{j}\tilde{f}_{i}(x)=\alpha_{1}^{j}\tilde{f}_{i}(x)+\sum_{s=0}^{j-1}\alpha_{1}^{j-1-s}T^{s}f_{i}(x).
\]

The last relation implies $T^{j}\tilde{f}_{i}(z)=0$, $j=0\lc \tilde{e}^{\red}_{i}$, and $T^{\tilde{e}^{\red}_{i}+1}\tilde{f}_{i}(z)=T^{\tilde{e}^{\red}_{i}}f_{i}(z)\neq 0$.

Since $\{\alpha_{1}^{x},\tilde{f}_{1}(x)\lc \tilde{f}_{n-1}(x)\}$ is a basis of $W$, the sequence of discrete exponents of $W$ at $z$ is given by
\[
\big(\tilde{e}^{\red}_{1}+1>\dots > \tilde{e}^{\red}_{n-1}+1>0 \big),
\]
which implies the lemma.
\end{proof}

Notice that for each $a=1\lc k$, the sequence of discrete exponents of $W^{\red}$ at $-z_{a}$ is given by
\[
(n',n'-1\lc n'-l_{a}+1,n'-l_{a}-1\lc 1,0).
\]
Therefore, by Lemma \ref{non-reduced lemma 2}, for each $a=1\lc k$, the sequence of discrete exponents of~$W$ at~$-z_{a}$ is given by
\[
(n,n-1\lc n-l_{a}+1,n-l_{a}-1\lc 1,0).
\]
Consider the space $\mathfrak{P}_{kn}[\boldsymbol{l},\boldsymbol{m}]$, where $\boldsymbol{l}=(l_{1}\lc l_{k})$ and $\boldsymbol{m}=(m_{1}\lc m_{n})$. One can repeat all constructions in Section \ref{qpol and BA} for the space $V$. Assume that $V$ satisfies conditions similar to those for a Gaudin admissible space in Section \ref{qpol and BA}. Then we obtain a vector $v_{V}\in \mathfrak{P}_{kn}[\boldsymbol{l},\boldsymbol{m}]$ such that
\[
H^{\langle k,n\rangle}_{i}(\bar{\alpha},\bar{z}+\boldsymbol{l})v_{V}=h_{i}^{V}v_{V},\qquad i=1\lc n
\]
for some numbers $h^{V}_{1}\lc h_{n}^{V}$. We will assume that $v_{V}\neq 0$.

Similarly, one can repeat all constructions in Section \ref{qexp to BAE} for the space $W$. Assume that $W$ satisfies conditions similar to those for an XXX-admissible space in Section \ref{qexp to BAE}. Then we obtain a vector $v_{W}\in \mathfrak{P}_{kn}[\boldsymbol{l},\boldsymbol{m}]$ such that
\[
G^{\langle n,k\rangle}_{i}\big({-}\bar{z}-\boldsymbol{l}+\bar{1},\bar{\alpha}\big)v_{W}=g_{i}^{W}v_{W},\qquad i=1\lc n
\]
for some numbers $g_{1}^{W}\lc g_{n}^{W}$. We will assume that $v_{W}\neq 0$.
\begin{thm}\label{discrete main 2 extended}
The following holds:
\[
h_{i}^{V}=-g_{i}^{W},\qquad i=1\lc n.
\]
\end{thm}
\begin{proof}
Define functions $\beta_{0}(x)\lc\beta_{k}(x)$, $\beta^{\red}_{0}(x)\lc\beta^{\red}_{k'}(x)$ by
\[
x^{k}D_{V}=\sum_{a=0}^{k}\beta_{a}(x)\left(x\frac{{\rm d}}{{\rm d}x}\right)^{k-a},\qquad x^{k'}D_{V^{\red}}=\sum_{a=0}^{k'}\beta^{\red}_{a}(x)\left(x\frac{{\rm d}}{{\rm d}x}\right)^{k'-a}.
\]

The eigenvalues $h_{1}^{V}\lc h_{n}^{W}$ can be expressed through $\beta_{1}(x)$, $\beta_{2}(x)$ using the same formula as in the case of reduced data, see \eqref{p5.4}. For convenience, we repeat this formula here:
\[
h_{i}^{V}=\frac{1}{\alpha_{i}}\Res_{x=\alpha_{i}}\left(\frac{1}{2}\beta_{1}^{2}(x)-\beta_{2}(x)\right)+\frac{m_{i}^{2}}{2}-m_{i}.
\]
Define also the following numbers:
\[
h^{V,\red}_{i}=\frac{1}{\alpha_{i}}\Res_{x=\alpha_{i}}\left(\frac{1}{2}\big(\beta^{\red}_{1}\big)^{2}(x)-\beta^{\red}_{2}(x)\right)+\frac{m_{i}^{2}}{2}-m_{i}.
\]

Suppose that $l_{1}=0$, and $l_{2}\lc l_{k}$ are not zero. Relation \eqref{non-reduced 1} implies
\[
\beta_{1}=\beta_{1}^{\red}-z_{1},\qquad \beta_{2}=\beta_{2}^{\red}-z_{1}\beta_{1}^{\red}.
\]

Using the last two formulas, it is easy to check that
\begin{equation}\label{non-reduced 6}
\Res_{x=\alpha_{i}}\left(\frac{1}{2}\beta_{1}^{2}(x)-\beta_{2}(x)\right)=\Res_{x=\alpha_{i}}\left(\frac{1}{2}\big(\beta^{\red}_{1}\big)^{2}(x)-\beta^{\red}_{2}(x)\right).
\end{equation}

By induction, formula \eqref{non-reduced 6} holds for any $l_{1}\lc l_{k}$. Therefore, we have $h_{i}^{V}=h^{V,\red}_{i}$, $i=1\lc n$.

Define polynomials $E_{0}(u),E_{1}(u),E_{2}(u),\dots$, $E^{\red}_{0}(u),E^{\red}_{1}(u),E^{\red}_{2}(u),\dots$ by
\[
S_{W}=\sum_{a=0}^{\infty}x^{-a}E_{a}(T),\qquad S_{W^{\red}}=\sum_{a=0}^{\infty}x^{-a}E^{\red}_{a}(T).
\]

The eigenvalues $g_{1}^{W}\lc g_{n}^{W}$ can be expressed through $E_{1}(u)$, $E_{2}(u)$ using the same formula as in the case of reduced data, see \eqref{5.6.15}. For convenience, we repeat this formula here:
\[
g_{i}^{W}=\frac{1}{\alpha_{i}}\Res_{u=\alpha_{i}}\left(\frac{1}{2}\frac{E^{2}_{1}(u)}{E^{2}_{0}(u)}-\frac{E_{2}(u)}{E_{0}(u)}\right)-\frac{m_{i}^{2}}{2}.
\]

Define also the following numbers
\[
g^{W,\red}_{i}=\frac{1}{\alpha_{i}}\Res_{u=\alpha_{i}}\left(\frac{1}{2}\frac{\big(E^{\red}_{1}(u)\big)^{2}}{\big(E^{\red}_{0}(u)\big)^{2}}-\frac{E^{\red}_{2}(u)}{E^{\red}_{0}(u)}\right)-\frac{m_{i}^{2}}{2}.
\]
Using relation \eqref{non-reduced 7}, we have
\[
E_{a}(u)=E_{a}^{\red}(u)\prod_{\substack{i=1\\m_{i}=0}}^{n}(u-\alpha_{i}),
\]
which implies $g_{i}^{W}=g^{W, \red}_{i}$, $i=1\lc n$.

In the proof of Theorem \ref{discrete main 2}, we already checked that $h^{V,\red}_{i}=-g^{W, \red}_{i}$ for all $i$ such that $m_{i}\neq 0$. If $m_{i}=0$, then $h^{V, \red}_{i}=g^{W, \red}_{i}=0$. Therefore, we have $h_{i}^{V}=-g_{i}^{W}$, $i=1\lc n$.

Theorem \ref{discrete main 2 extended} is proved.
\end{proof}

\appendix

\section{Discrete Wronskian identities}\label{appendixA}
In this section, we collect discrete Wronskian identities that were used in the paper. Identities \eqref{Wr1}--\eqref{Wr2} with proofs can also be found in \cite[Appendix~B]{MV2}.

Recall that $T$ is the shift operator defined by $Tf(x)=f(x+1)$. Recall that for any functions $f_{1}\lc f_{n}$, the discrete Wronskian $\mathcal{W}{\rm r}(f_{1}\lc f_{n})$ is the determinant of the matrix $\big(T^{j-1}f_{i}\big)_{i,j=1}^{n}$. Denote $T^{(n)}f=f(Tf)\big(T^{2}f\big)\cdots\big(T^{n-1}f\big)$.
We have the following obvious relations:
\begin{gather}\label{Wr1}
\mathcal{W}{\rm r}(hf_{1}\lc hf_{n})=\big(T^{(n)}h\big)\mathcal{W}{\rm r}(f_{1}\lc f_{n})\qquad \text{for any }h,
\\
\label{Wr3}
\mathcal{W}{\rm r}(1,f_{1}\lc f_{n})=\mathcal{W}{\rm r}((T-1)f_{1}\lc (T-1)f_{n}).
\end{gather}
Assume that $f_{1}\neq 0$. Combining formulas \eqref{Wr1} and \eqref{Wr3}, we get
\begin{equation}\label{Wr4}
\mathcal{W}{\rm r}(f_{1},f_{2}\lc f_{n})=\big(T^{(n)}f_{1}\big)\mathcal{W}{\rm r}\left((T-1)\frac{f_{2}}{f_{1}}\lc (T-1)\frac{f_{n}}{f_{1}}\right).
\end{equation}
\begin{prop}\label{disc wronsc conj} For any functions $f_{1}\lc f_{n}$, $h_{1}\lc h_{m}$, where $f_{1}\neq 0$, the following holds:
\begin{gather}
\mathcal{W}{\rm r}(\mathcal{W}{\rm r}(f_{1}\lc f_{n},h_{1})\lc\mathcal{W}{\rm r}(f_{1}\lc f_{n},h_{m}))\nonumber\\
\qquad{} =\big(T^{(m-1)}\mathcal{W}{\rm r}(Tf_{1}\lc Tf_{n})\big)\mathcal{W}{\rm r}(f_{1}\lc f_{n},h_{1}\lc h_{m}).\label{Wr2}
\end{gather}
\end{prop}
\begin{proof}
We will prove the proposition by induction on $n$.
Let $n=1$. Denote $f_{1}=f$. Using formula~\eqref{Wr4}, we compute
\[\mathcal{W}{\rm r}(f,h_{i})=\big(T^{(2)}f\big)\mathcal{W}{\rm r}\left((T-1)\frac{h_{i}}{f}\right)=\big(T^{(2)}f\big) (T-1)\frac{h_{i}}{f},\qquad i=1\lc m.\]
Therefore,
\begin{align*}
\mathcal{W}{\rm r}(\mathcal{W}{\rm r}(f,h_{1})\lc\mathcal{W}{\rm r}(f,h_{m}))& =\big(T^{(m)}T^{(2)}f\big)\mathcal{W}{\rm r}\left((T-1)\frac{h_{1}}{f}\lc (T-1)\frac{h_{m}}{f}\right)\\
&=\big(T^{(m-1)}Tf\big)\big(T^{(m+1)}f\big)\mathcal{W}{\rm r}\left((T-1)\frac{h_{1}}{f}\lc (T-1)\frac{h_{m}}{f}\right)\\
&=\big(T^{(m-1)}Tf\big)\mathcal{W}{\rm r}(h_{1}\lc h_{m}).
\end{align*}
Assume that formula \eqref{Wr2} is true for some $n\geq 1$. For functions $f_{1}\lc f_{n+1}$, $h_{1}\lc h_{m}$, define $\tilde{f}_{i}=(T-1)(f_{i}/f_{1})$, $\tilde{h}_{j}=(T-1)(h_{j}/f_{1})$, $i=2\lc n+1$, $j=1\lc m$. Then we compute
\begin{gather}
\mathcal{W}{\rm r}(\mathcal{W}{\rm r}(f_{1}\lc f_{n+1},h_{1})\lc\mathcal{W}{\rm r}(f_{1}\lc f_{n+1},h_{m}))\nonumber\\
\qquad{}=\big(T^{(m)}T^{(n+2)}f_{1}\big)\mathcal{W}{\rm r}\big(\mathcal{W}{\rm r}\big(\tilde{f}_{2}\lc \tilde{f}_{n+1},\tilde{h}_{1}\big)\lc\mathcal{W}{\rm r}\big(\tilde{f}_{2}\lc \tilde{f}_{n+1},\tilde{h}_{m}\big)\big)\nonumber\\
\qquad{}=\big(T^{(m)}T^{(n+2)}f_{1}\big)\big(T^{(m-1)}\mathcal{W}{\rm r}\big(T\tilde{f}_{2}\lc T\tilde{f}_{n+1}\big)\big)\mathcal{W}{\rm r}\big(\tilde{f}_{2}\lc \tilde{f}_{n+1},\tilde{h}_{1}\lc \tilde{h}_{m}\big)\nonumber\\
\qquad{}=\big(T^{(m-1)}\big[\big(T^{(n+1)}Tf_{1}\big)\mathcal{W}{\rm r}\big(T\tilde{f}_{2}\lc T\tilde{f}_{n+1}\big)\big]\big) \nonumber\\
\qquad\quad{}\times\big(T^{(n+m+1)}f_{1}\big)\mathcal{W}{\rm r}\big(\tilde{f}_{2}\lc \tilde{f}_{n+1},\tilde{h}_{1}\lc \tilde{h}_{m}\big)\nonumber\\
\qquad{}=\big(T^{(m-1)}\mathcal{W}{\rm r}(Tf_{1}\lc Tf_{n+1})\big)\mathcal{W}{\rm r}(f_{1}\lc f_{n+1},h_{1}\lc h_{m}).\label{Wr5}
\end{gather}
Here, on the first step, we used formulas \eqref{Wr1} and \eqref{Wr4}, on the second step, we used the assumption hypothesis, on the third step, we used
\[
T^{(m)}T^{(n+2)}f_{1}=\big(T^{(m-1)}T^{(n+1)}Tf_{1}\big)\big(T^{(n+m+1)}f_{1}\big),
\] and on the fourth step, we used formula \eqref{Wr4} again.

Computation \eqref{Wr5} proves the induction step finishing the proof of the proposition.\end{proof}

\subsection*{Acknowledgements}

The author would like to thank Vitaly Tarasov for many helpful discussions and for his valuable comments on a draft of this text.
The author would like to thank referees for their contribution to improvement of the paper.

\pdfbookmark[1]{References}{ref}
\LastPageEnding


\begin{thebibliography}{99}
\footnotesize\itemsep=0pt

\bibitem{J}
Jur\v{c}o B., Classical {Y}ang--{B}axter equations and quantum integrable
 systems, \href{https://doi.org/10.1063/1.528305}{\textit{J.~Math. Phys.}} \textbf{30} (1989), 1289--1293.

\bibitem{MaV}
Markov Y., Varchenko A., Hypergeometric solutions of trigonometric {KZ}
 equations satisfy dynamical difference equations, \href{https://doi.org/10.1006/aima.2001.2027}{\textit{Adv. Math.}}
 \textbf{166} (2002), 100--147, \href{https://arxiv.org/abs/math.QA/0103226}{arXiv:math.QA/0103226}.

\bibitem{MR}
Molev A., Ragoucy E., Higher-order {H}amiltonians for the trigonometric
 {G}audin model, \href{https://doi.org/10.1007/s11005-019-01170-2}{\textit{Lett. Math. Phys.}} \textbf{109} (2019), 2035--2048,
 \href{https://arxiv.org/abs/1802.06499}{arXiv:1802.06499}.

\bibitem{MTV6}
Mukhin E., Tarasov V., Varchenko A., Bethe eigenvectors of higher transfer
 matrices, \href{https://doi.org/10.1088/1742-5468/2006/08/p08002}{\textit{J.~Stat. Mech. Theory Exp.}} \textbf{2006} (2006), P08002,
 44~pages, \href{https://arxiv.org/abs/math.QA/0605015}{arXiv:math.QA/0605015}.

\bibitem{MTV2}
Mukhin E., Tarasov V., Varchenko A., Bispectral and
 $(\mathfrak{gl}_{N},\mathfrak{gl}_{M})$ dualities, \href{https://doi.org/10.1007/s11853-007-0003-y}{\textit{Funct. Anal.
 Other Math.}} \textbf{1} (2006), 47--69, \href{https://arxiv.org/abs/math.QA/0510364}{arXiv:math.QA/0510364}.

\bibitem{MTV4}
Mukhin E., Tarasov V., Varchenko A., Bispectral and
 {$({\mathfrak{gl}}_N,{\mathfrak{gl}}_M)$} dualities, discrete versus
 differential, \href{https://doi.org/10.1016/j.aim.2007.11.022}{\textit{Adv. Math.}} \textbf{218} (2008), 216--265,
 \href{https://arxiv.org/abs/math.QA/0605172}{arXiv:math.QA/0605172}.

\bibitem{MTV1}
Mukhin E., Tarasov V., Varchenko A., A generalization of the {C}apelli
 identity, in Algebra, Arithmetic, and Geometry: in Honor of
 {Y}u.{I}.~{M}anin, {V}ol.~{II}, \textit{Progr. Math.}, Vol.~270, \href{https://doi.org/10.1007/978-0-8176-4747-6_12}{Birkh\"auser
 Boston}, Boston, MA, 2009, 383--398, \href{https://arxiv.org/abs/math.QA/0610799}{arXiv:math.QA/0610799}.

\bibitem{MV2}
Mukhin E., Varchenko A., Solutions to the {$XXX$} type {B}ethe ansatz equations
 and flag varieties, \href{https://doi.org/10.2478/BF02476011}{\textit{Cent. Eur.~J. Math.}} \textbf{1} (2003), 238--271,
 \href{https://arxiv.org/abs/math.QA/0211321}{arXiv:math.QA/0211321}.

\bibitem{MV}
Mukhin E., Varchenko A., Quasi-polynomials and the {B}ethe ansatz, in Groups,
 Homotopy and Configuration Spaces, \textit{Geom. Topol. Monogr.}, Vol.~13,
 \href{https://doi.org/10.2140/gtm.2008.13.385}{Geom. Topol. Publ.}, Coventry, 2008, 385--420, \href{https://arxiv.org/abs/math.QA/0604048}{arXiv:math.QA/0604048}.

\bibitem{RV}
Reshetikhin N., Varchenko A., Quasiclassical asymptotics of solutions to the
 {KZ} equations, in Geometry, Topology, \& Physics, Conf. Proc. Lecture Notes
 Geom. Topology,~IV, Int. Press, Cambridge, MA, 1995, 293--322,
 \href{https://arxiv.org/abs/hep-th/9402126}{arXiv:hep-th/9402126}.

\bibitem{TU1}
Tarasov V., Uvarov F., Duality for {B}ethe algebras acting on polynomials in
 anticommuting variables, \href{https://doi.org/10.1007/s11005-020-01329-2}{\textit{Lett. Math. Phys.}} \textbf{110} (2020),
 3375--3400, \href{https://arxiv.org/abs/1907.02117}{arXiv:1907.02117}.

\bibitem{TU2}
Tarasov V., Uvarov F., Duality for {K}nizhnik--{Z}amolodchikov and dynamical
 operators, \href{https://doi.org/10.3842/SIGMA.2020.035}{\textit{SIGMA}} \textbf{16} (2020), 035, 10~pages,
 \href{https://arxiv.org/abs/1904.07309}{arXiv:1904.07309}.

\bibitem{TV4}
Tarasov V., Varchenko A., Duality for {K}nizhnik--{Z}amolodchikov and dynamical
 equations, \href{https://doi.org/10.1023/A:1019787006990}{\textit{Acta Appl. Math.}} \textbf{73} (2002), 141--154,
 \href{https://arxiv.org/abs/math.QA/0112005}{arXiv:math.QA/0112005}.

\bibitem{TV2}
Tarasov V., Varchenko A., Dynamical differential equations compatible with
 rational {qKZ} equations, \href{https://doi.org/10.1007/s11005-004-6363-z}{\textit{Lett. Math. Phys.}} \textbf{71} (2005),
 101--108, \href{https://arxiv.org/abs/math.QA/0403416}{arXiv:math.QA/0403416}.

\end{thebibliography}
\end{document}